\documentclass[11pt]{article}
\usepackage[english]{babel}
\usepackage[T1]{fontenc}
\usepackage[utf8]{inputenc}
\usepackage{amsmath}
\usepackage{amssymb}
\usepackage{amsfonts}
\usepackage{amsthm}
\usepackage{bm}
\usepackage{mathtools}
\usepackage{upgreek}

\newcommand{\eps}{\varepsilon}
\newcommand{\C}{\mathbb{C}}
\newcommand{\N}{\mathbb{N}}

\newcommand{\R}{\mathbb{R}}

\newcommand{\Z}{\mathbb{Z}}

\newcommand{\boC}{\mathcal{C}}
\newcommand{\boD}{\mathcal{D}}
\newcommand{\boE}{\mathcal{E}}

\newcommand{\boH}{\mathcal{H}}
\newcommand{\boI}{\mathcal{I}}
\newcommand{\boJ}{\mathcal{J}}

\newcommand{\boN}{\mathcal{N}}
\newcommand{\boO}{\mathcal{O}}

\newcommand{\boR}{\mathcal{R}}

\newcommand{\boT}{\mathcal{T}}

\newcommand{\boV}{\mathcal{V}}

\newcommand{\boX}{\mathcal{X}}

\newcommand{\boZ}{\mathcal{Z}}

\newcommand{\ga}{\mathfrak{a}}

\newcommand{\gc}{\mathfrak{c}}

\newcommand{\gu}{\mathfrak{u}}

\newcommand{\ch}{{\rm ch}}

\DeclareMathOperator{\Dom}{{\rm Dom}}

\DeclareMathOperator{\Ker}{{\rm Ker}}

\newcommand{\sh}{{\rm sh}}

\DeclareMathOperator{\Span}{{\rm Span}}
\DeclareMathOperator{\supp}{{\rm supp}}
\renewcommand{\th}{{\rm th}}

\newtheorem*{claim}{Claim}
\newtheorem{cor}{Corollary}
\newtheorem{lemma}{Lemma}
\newtheorem{prop}{Proposition}
\newtheorem{step}{Step}
\newtheorem{thm}{Theorem}
\theoremstyle{definition}

\newtheorem*{merci}{Acknowledgments}
\newtheorem*{remark}{Remark}
\newtheorem*{remarks}{Remarks}
\theoremstyle{remark}
\newtheorem{case}{Case}

\begin{document}

\title{Asymptotic stability in the energy space for dark solitons of the Gross-Pitaevskii equation}
\author{
\renewcommand{\thefootnote}{\arabic{footnote}} Fabrice B\'ethuel
\footnotemark[1], Philippe Gravejat \footnotemark[2], Didier Smets
\footnotemark[3]}
\footnotetext[1]{Laboratoire Jacques-Louis Lions, Universit\'e Pierre et Marie Curie, Bo\^ite Courrier 187, 75252 Paris Cedex 05, France. E-mail: {\tt bethuel@ann.jussieu.fr}}
\footnotetext[2]{Centre de Math\'ematiques Laurent Schwartz, \'Ecole Polytechnique, 91128 Palaiseau Cedex,
France. E-mail: {\tt gravejat@math.polytechnique.fr}}
\footnotetext[3]{Laboratoire Jacques-Louis Lions, Universit\'e Pierre et Marie Curie, Bo\^ite Courrier 187,
75252 Paris Cedex 05, France. E-mail: {\tt smets@ann.jussieu.fr}}
\maketitle

\begin{abstract}
We pursue our work \cite{BetGrSm1} on the dynamical stability of dark solitons for the one-dimensional Gross-Pitaevskii equation. In this paper, we prove their asymptotic stability under small perturbations in the energy space. In particular, our results do not require smallness in some weighted spaces or a priori spectral assumptions. Our strategy is reminiscent of the one used by Martel and Merle in various works regarding generalized Korteweg-de Vries equations. The important feature of our contribution is related to the fact that while Korteweg-de Vries equations possess unidirectional dispersion, Schr\"odinger equations do not.
\end{abstract}

\section{Introduction}

We consider the one-dimensional Gross-Pitaevskii equation
\renewcommand{\theequation}{GP}
\begin{equation}
\label{GP}
i \partial_t \Psi + \partial_{xx} \Psi + \Psi \big( 1 - |\Psi|^2 \big) = 0,
\end{equation}
for a function $\Psi: \R \times \R \to \C$, supplemented with the boundary condition at infinity
\renewcommand{\theequation}{\arabic{equation}}
\setcounter{equation}{0}
\begin{equation}
\label{cond:GP}
|\Psi(x, t)| \to 1, \quad {\rm as} \quad |x| \to + \infty.
\end{equation}
The three-dimensional version of \eqref{GP} was introduced in the context of Bose-Einstein condensation in \cite{Pitaevs1, Gross1}. It is also used as a model in other areas of physics such as nonlinear optics \cite{KivsLut1} and quantum fluid mechanics \cite{Coste1}. In nonlinear optics, the Gross-Pitaevskii equation appears as an envelope equation in optical fibers, and is mostly relevant in the one and two dimensional cases. In dimension one, the case studied in this paper, it models the propagation of dark pulses in slab waveguides, and the boundary condition \eqref{cond:GP} corresponds to a non-zero background.

On a mathematical level, the Gross-Pitaevskii equation is a defocusing nonlinear Schr\"odinger equation. It is Hamiltonian, and in dimension one, it owns the remarkable property to be integrable by means of the inverse scattering method \cite{ShabZak2}. The Hamiltonian is given 
by the Ginzburg-Landau energy 
$$\boE(\Psi) := \frac{1}{2} \int_\R |\partial_x \Psi|^2 + \frac{1}{4} \int_\R (1 - |\Psi|^2)^2.$$
A soliton with speed $c$ is a travelling-wave solution of \eqref{GP} of the form
$$\Psi(x, t) := U_c(x - c t).$$
Its profile $U_c$ is a solution to the ordinary differential equation
\begin{equation}
\label{eq:solc}
- i c \partial_x U_c + \partial_{xx} U_c + U_c \big( 1 - |U_c|^2 \big) = 0.
\end{equation}
The solutions to \eqref{eq:solc} with finite Ginzburg-Landau energy are explicitly known. For $|c| \geq \sqrt{2}$, they are the constant functions of unitary modulus, while for $|c| < \sqrt{2}$, up to the invariances of the problem, i.e. multiplication by a constant of modulus one and translation, they are uniquely given by the expression
\begin{equation}
\label{form:solc}
U_c(x) := \sqrt{\frac{2 - c^2}{2}} \th \Big( \frac{\sqrt{2 - c^2} x}{2} \Big) + i \frac{c}{\sqrt{2}}.
\end{equation}
Notice that solitons $U_c$ with speed $c \neq 0$ do not vanish on $\R$. They are called dark solitons, with reference to nonlinear optics where $|\Psi|^2$ refers to the intensity of light. Instead, since it vanishes at one point, $U_0$ is called the black soliton. Notice also, this turns out to be an important feature, that solitons $U_c$ with $c \simeq \sqrt{2}$ have indefinitely small energy.

Our goal in this paper is to study the \eqref{GP} flow for initial data that are close to dark solitons, and in particular to analyze the stability of solitons. Since we deal with an infinite dimensional dynamical system, the notion of stability relies heavily on the way to measure distances. A preliminary step is to address the Cauchy problem with respect to these distances. In view of the Hamiltonian $\boE$, the natural energy space for \eqref{GP} is given by 
$$\boX(\R) := \big\{ \Psi \in H^1_{\rm loc}(\R), \ \Psi' \in L^2(\R) \ {\rm and} \ 1 - |\Psi|^2 \in L^2(\R) \big\}.$$
Due to the non-vanishing conditions at infinity, it is not a vector space. Yet $\boX(\R)$ can be 
given a structure of complete metric space through the distance
$$d(\Psi_1, \Psi_2) := \big\| \Psi_1 - \Psi_2 \big\|_{L^\infty(\R)} + \big\| \Psi_1' - \Psi_2' \big\|_{L^2(\R)} + \big\| |\Psi_1| - |\Psi_2| \big\|_{L^2(\R)}.$$
In space dimension one, for an initial datum $\Psi^0 \in \boX(\R)$, the Gross-Pitaevskii equation possesses a unique global solution $\Psi \in \boC^0(\R, \boX(\R))$, and moreover $\Psi - \Psi^0 \in \boC^0(\R, H^1(\R))$ (see e.g. \cite{Zhidkov1, Gallo1, Gerard2} and Appendix \ref{sub:weak-conv}). In the sequel, stability is based on the distance $d$.

It is well-known that a perturbation of a soliton $U_\gc$ at initial time cannot remain a perturbation of the same soliton for all time. This is related to the fact that there is a continuum of solitons with different speeds. If $\Psi^0 = U_{c}$ for $c \simeq \gc$, but $c \neq \gc$, then $\Psi(x,t) = U_{c}(x - c t)$ diverges from $U_\gc(x - \gc t)$, as $t \to + \infty$. The notion of orbital stability is tailored to deal with such situations. It means that a solution corresponding to a perturbation of a soliton $U_\gc$ at initial time remains a perturbation of the family of solitons with same speed for all time. Orbital stability of dark and black solitons was proved in \cite{LinZhiw1, BeGrSaS1} (see also \cite{BetGrSa2, GeraZha1, BetGrSm1}).

\begin{thm}[\cite{LinZhiw1, BeGrSaS1}]
\label{thm:orbistab_psi}
Let $\gc \in (- \sqrt{2}, \sqrt{2})$. Given any positive number $\eps$, there exists a positive number $\delta$ such that, if
$$ d \big( \Psi^0, U_\gc \big) \leq \delta,$$
then
$$\sup_{t \in \R} \inf_{(a, \theta) \in \R^2} d \big( \Psi(\cdot, t), e^{i \theta} U_\gc(\cdot -a) \big) \leq \eps.$$
\end{thm}

The proof of Theorem \ref{thm:orbistab_psi} is mainly variational. Following the strategy developed in \cite{CazeLio1} or \cite{Weinste2, GriShSt1}, it combines minimizing properties of the solitons with the Hamiltonian nature of \eqref{GP} through conservation of energy and momentum.

In the sequel, our focus is put on the notion of asymptotic stability. For a finite dimensional system, asymptotic stability of a stationary state or orbit means that any small perturbation of the given state at initial time eventually converges to that state as time goes to infinity. For a finite dimensional Hamiltonian system, this is excluded by the symplectic structure. In infinite dimension, one may take advantage of different topologies to define relevant notions of asymptotic stability. Our main result is

\begin{thm}
\label{thm:stabasympt_psi}
Let $\gc \in (- \sqrt{2}, \sqrt{2}) \setminus \{ 0 \}$. There exists a positive number $\delta_\gc$, depending only on $\gc$, such that, if
$$d \big( \Psi^0, U_\gc \big) \leq \delta_\gc,$$
then there exist a number $\gc^* \in (- \sqrt{2}, \sqrt{2}) \setminus \{ 0 \}$, and two functions $b \in \boC^1(\R, \R)$ and $\theta \in \boC^1(\R, \R)$ such that
$$b'(t) \to \gc^*, \quad {\rm and} \quad \theta'(t) \to 0,$$
as $t \to + \infty$, and for which we have
\begin{equation}\label{eq:venuti}
e^{- i \theta(t)} \Psi \big( \cdot + b(t), t \big) \to U_{\gc^*} \quad {\rm in} \ L^\infty_{\rm loc}(\R), \quad {\rm and} \quad e^{- i \theta(t)} \partial_x \Psi \big( \cdot + b(t), t \big) \rightharpoonup \partial_x U_{\gc^*} \quad {\rm in} \ L^2(\R),
\end{equation}
in the limit $t \to + \infty$.
\end{thm}

Whereas Theorem \ref{thm:orbistab_psi} establishes that the solution remains close to the whole family of dark solitons, Theorem \ref{thm:stabasympt_psi} describes a convergence to some orbit on that family. In particular, it expresses the fact that in the reference frame of the limit soliton, the perturbation is dispersed towards infinity.

Concerning the topology, the convergence in $L_{\rm loc}^\infty(\R)$ in \eqref{eq:venuti} cannot be improved into a convergence in $L^\infty(\R)$, for instance due to the presence of additional small solitons, or to a possible phenomenon of slow phase winding at infinity. Similarly, the weak convergence of the gradients in $L^2(\R)$ cannot be improved into a strong convergence in $L^2(\R)$, due to the Hamiltonian nature of the equation. Yet, it is not impossible that the latter could be improved into a strong convergence in $L_{\rm loc}^2(\R)$, but we have no proof of that fact. 

\begin{remarks}
$(i)$ Note that the case $\gc = 0$ is excluded from the statement of Theorem \ref{thm:stabasympt_psi}. For $\gc \neq 0$, if $\delta_\gc$ is chosen sufficiently small, it follows from the Sobolev embedding theorem and Theorem \ref{thm:orbistab_psi} that $\Psi$ does not vanish on $\R \times \R$. We rely heavily on this property for proving Theorem \ref{thm:stabasympt_psi}, in particular in the next subsection where we introduce the hydrodynamical framework.

\noindent $(ii)$ Complementing Theorem \ref{thm:stabasympt_psi} with information from Theorem \ref{thm:orbistab_psi}, one derives a control on $|\gc - \gc^*|$ relative to $d(\Psi^0, U_\gc)$, and in particular it directly follows from the two statements that $|\gc - \gc^*| \to 0$, as $d(\Psi^0, U_\gc) \to 0$. We will actually prove uniform estimates, valid for all times, stating that
$$d \big( e^{- i \theta(t)} \Psi \big( \cdot + b(t), t), U_{\gc^*} \big) + \big| b'(t) - \gc \big|
\leq A_\gc\, d \big( \Psi^0, U_\gc \big),$$
where $A_\gc$ depends only on $\gc$ (see Theorem \ref{thm:orbistab} below). We believe that the functions $b(t) - \gc^*$ and $\theta(t)$ need not be bounded, and in particular need not have limits as $t \to + \infty$, unless additional (regularity/localization) assumptions are made on the initial perturbation.

\noindent $(iii)$ Finally, we mention that our proofs make no determinant use of the integrability of the Gross-Pitaevskii equation, nor of the explicit nature of the solitons $U_c$. In particular, they could presumably be extended to related nonlinearities (e.g. those studied in \cite{Chiron7}) without major modifications.
\end{remarks}

As previously mentioned, the Gross-Pitaevskii equation is both nonlinear and dispersive. For constant coefficient linear equations, dispersion implies local convergence towards zero as a consequence of a stationary phase type argument. This property does not carry over to general coefficient or nonlinear equations. Persistent localized structures like ground states or solitons are characteristic counter-examples. In that situation, dispersion around the localized structure, through the linearized equation, seems more appropriate.

In \cite{SoffWei1, SoffWei2, SoffWei3}, Soffer and Weinstein studied the asymptotic stability of
ground states for the nonlinear Schr\"odinger equation with a potential in a regime for which the
nonlinear ground-state is a close continuation of the linear one. They establish dispersive
estimates for the linearized equation around the ground state in suitable weighted spaces, which
allow to implement a fixed point argument in a space of functions that vanish as time goes to
infinity. This was later extended to a fully nonlinear regime for the nonlinear Schr\"odinger equation without potential (see e.g. \cite{BuslPer1, BuslPer2, BuslSul1, Cuccagn1}) and with a potential (see e.g. \cite{GangSig1}). We refer to \cite{Cuccagn2} for a detailed historical survey of those and related works. In this context, the solutions behave at large time as a soliton plus a purely scattering linear perturbation. This reflects either a priori spectral assumptions or the use of weighted spaces for the initial perturbation. In particular this prevents to consider e.g. solutions of multi-soliton type, where a reference soliton is perturbed by one or more small solitons which do not disperse in time. These solutions are known to exist for a large class of nonlinearities.

A related equation where similar questions were addressed is given by the Korteveg-de Vries
equation, or its generalizations. In \cite{PegoWei1}, Pego and Weinstein studied the asymptotic
stability of solitons in spaces of exponentially localized perturbations (see also \cite{Mizumac1}
for perturbations with algebraic decay). Here also, one may check a posteriori that multi-solitons
are excluded from the assumptions. In a series of papers, Martel and Merle \cite{MartMer1, MartMer2,
Martel2, MartMer4, MartMer5, MartMer6} were able to establish the asymptotic stability of the
generalized Korteweg-de Vries equations in the energy space $H^1(\R)$, without additional a priori
assumptions. In particular, perturbations by multi-solitons are there handled too. Their method differs in many respects with the ones above and relies on a combination of variational and dynamical arguments in the form of localized monotonicity formulas. Our work was strongly motivated by the possibility to obtain such an extension for the Gross-Pitaevskii equation. Even though the later is also a nonlinear Schr\"odinger equation, the non-vanishing boundary conditions at spatial
infinity modifies the dispersive properties. More precisely, the dispersion relation for the
linearization of \eqref{GP} around the constant $1$ is given by
$$\omega^2 = 2 k^2 + k^4,$$
and therefore dispersive waves travel at speeds (positive or negative) greater than the speed of
sound $\sqrt{2}$. Instead, solitons are subsonic and hence the speeds of solitons and dispersive 
waves are decoupled. In contrast, for the focusing nonlinear Schr\"odinger equation with vanishing boundary conditions, the speeds of solitons and of dispersive waves overlap. The description of the dynamics in the energy space near solitons of the focusing nonlinear Schr\"odinger equation is certainly more delicate, since for instance in the case of a purely cubic nonlinearity the existence of breathers prevent asymptotic stability. 

In higher dimension, the Gross-Pitaevskii equation possesses localized structures as well, in particular in the form of vortices and travelling vortex pairs in dimension two, or vortex rings in
dimension three (see e.g. \cite{BethSau1, BetOrSm1, BetGrSa1, Maris7}). Some of these have been proved to be orbitally stable \cite{ChirMar2}. Their asymptotic stability remains an interesting open question which partially motivated the present work. 

In the remaining part of this introduction, we present the main ingredients leading to the proof of Theorem \ref{thm:stabasympt_psi}.

\subsection{Hydrodynamical form of the Gross-Pitaevskii equation}
\label{sub:hydro}

As mentioned above, when $c \neq 0$ the soliton $U_c$ does not vanish and may thus be written under the form
$$U_c := \varrho_c e^{i \varphi_c},$$
for smooth real functions $\varrho_c$ and $\varphi_c$. In view of formula \eqref{form:solc}, the maps $\eta_c := 1 - \varrho_c^2$ and $v_c := -\partial_x \varphi_c$ are given by 
\begin{equation}
\label{form:etavc}
\eta_c(x) = \frac{2 - c^2}{2 \ch \big( \frac{\sqrt{2 - c^2}}{2} x \big)^2}, \quad {\rm and} \quad v_c(x) = \frac{c \eta_c(x)}{2 \big( 1 - \eta_c(x) \big)} = \frac{c (2 - c^2)}{2 \big( 2 \ch \big( \frac{\sqrt{2 - c^2}}{2} x \big)^2 - 2 + c^2 \big)}.
\end{equation}
In the sequel, we set
$$Q_{c, a} := \big( \eta_{c, a}, v_{c, a} \big) := \big( \eta_c(\cdot - a), v_c(\cdot - a) \big),$$
for $0 < |c| < \sqrt{2}$ and $a \in \R$. More generally, provided a solution $\Psi$ to \eqref{GP} does not vanish, it may be lifted without loss of regularity as
$$\Psi := \varrho e^{i \varphi},$$
where $\varrho := |\Psi|$. The functions $\eta := 1 - \varrho^2$ and $v := - \partial_x \varphi$ are solutions, at least formally, to the so-called hydrodynamical form of \eqref{GP}, namely
\renewcommand{\theequation}{HGP}
\begin{equation}
\label{HGP}
\left\{ \begin{array}{ll}
\partial_t \eta = \partial_x \big( 2 \eta v - 2 v \big),\\[5pt]
\displaystyle \partial_t v = \partial_x \Big( v^2 - \eta + \partial_x \Big( \frac{\partial_x \eta}{2 (1 - \eta)} \Big) - \frac{(\partial_x \eta)^2}{4 (1 - \eta)^2} \Big).
\end{array} \right.
\end{equation}
The Ginzburg-Landau energy $\boE(\Psi)$, rewritten in terms of $(\eta, v)$, is given by
$$E(\eta, v) := \int_\R e(\eta, v) := \frac{1}{8} \int_\R \frac{(\partial_x \eta)^2}{1 - \eta} + \frac{1}{2} \int_\R (1 - \eta) v^2 + \frac{1}{4} \int_\R \eta^2,$$
so that the energy space for \eqref{HGP} is the open subset
$$\boN\boV(\R) := \Big\{ (\eta, v) \in X(\R), \ {\rm s.t.} \ \max_{x \in \R} \eta(x) < 1 \Big\},$$
where the Hilbert space $X(\R) := H^1(\R) \times L^2(\R)$ is equipped with the norm 
$$\| (\eta, v) \|_{X(\R)}^2 := \| \eta \|_{H^1(\R)}^2 + \| v \|_{L^2(\R)}^2.$$

It is shown in \cite{Tartous0} (see also Proposition \ref{prop:Cauchy-Q}) that if $\Psi \in \boC^0(\R, X(\R))$ is a solution to \eqref{GP} with $\inf_{\R \times \R} |\Psi| > 0$, then $(\eta, v) \in \boC^0(\R, \boN\boV(\R))$ is a solution to \eqref{HGP} and the energy $E(\eta, v)$ is a conserved quantity, as well as the momentum
$$P(\eta, v) := \frac{1}{2} \int_\R \eta v.$$

\subsection{Orbital stability in the hydrodynamical framework}
\label{sub:hydrostab}

The following is a quantitative version of Theorem \ref{thm:orbistab_psi} in the hydrodynamical framework (therefore for $\gc \neq 0$).

\begin{thm}[\cite{LinZhiw1, BetGrSm1}]
\label{thm:orbistab}
Let $\gc \in (- \sqrt{2}, \sqrt{2}) \setminus \{ 0 \}$. There exists a positive number $\alpha_\gc$, depending only on $\gc$, with the following properties. Given any $(\eta_0, v_0) \in X(\R)$ such that
\renewcommand{\theequation}{\arabic{equation}}
\setcounter{equation}{3}
\begin{equation}
\label{cond:alpha}
\alpha^0 := \big\| (\eta_0, v_0)- Q_{\gc, \ga} \big\|_{X(\R)} \leq \alpha_\gc,
\end{equation}
for some $\ga \in \R$, there exist a unique global solution $(\eta, v) \in \boC^0(\R, \boN\boV(\R))$ to \eqref{HGP} with initial data $(\eta_0, v_0)$, and two maps $c \in \boC^1(\R, (- \sqrt{2}, \sqrt{2}) \setminus \{ 0 \})$ and $a \in \boC^1(\R, \R)$ such that the function $\eps$ defined by
\begin{equation}
\label{def:eps}
\eps(\cdot, t) := \big( \eta(\cdot + a(t), t), v(\cdot + a(t), t) \big) - Q_{c(t)},
\end{equation}
satisfies the orthogonality conditions
\begin{equation}
\label{eq:ortho}
\langle \eps(\cdot, t), \partial_x Q_{c(t)} \rangle_{L^2(\R)^2} = P'(Q_{c(t)})(\eps(\cdot, t)) = 0,
\end{equation}
for any $t \in \R$. Moreover, there exist two positive numbers $\sigma_\gc$ and $A_\gc$, depending only and continuously on $\gc$, such that
\begin{equation}
\label{eq:max-eta}
\max_{x \in \R} \eta(x, t) \leq 1 - \sigma_\gc,
\end{equation}
\begin{equation}
\label{eq:modul0}
\big\| \eps(\cdot, t) \big\|_{X(\R)} + \big| c(t) - \gc \big| \leq A_\gc \alpha^0,\\
\end{equation}
and
\begin{equation}
\label{eq:modul1}
\big| c'(t) \big| + \big| a'(t) - c(t) \big|^2 \leq A_\gc \big\| \eps(\cdot, t) \big\|_{X(\R)}^2,
\end{equation}
for any $t\in \R$.
\end{thm}

The proof of Theorem \ref{thm:orbistab} is essentially contained in \cite{BetGrSm1}. However, since the statement in \cite{BetGrSm1} slightly differs from the statement presented here, in particular regarding the quadratic dependence of $c'(t)$, we provide the few additional details in Section \ref{sec:orbistab} below. The main ingredient is a spectral estimate which we recall now for future reference (see also Section \ref{sec:Hc} for additional information). The functional $E - c P$ is a conserved quantity of the flow whenever $c$ is fixed, and it plays a particular role in the analysis since the solitons $Q_c$ are solutions of the equation
$$E'(Q_c) - c P'(Q_c) = 0.$$
In particular, 
$$\big[ E - c P \big] \big( Q_c + \eps \big) = \big[ E - c P \big] \big( Q_c \big) + \frac{1}{2} H_c \big( \eps \big) + \boO \big( \| \eps \|_{X(\R)}^3 \big),$$
as $\eps \to 0$ in $X(\R)$. In this formula, $H_c$ denotes the quadratic form on $X(\R)$ corresponding to the unbounded linear operator 
$$\boH_c := E''(Q_c) - c P''(Q_c).$$
The operator $\boH_c$ is self-adjoint on $L^2(\R) \times L^2(\R)$, with domain $\Dom(\boH_c) := H^2(\R) \times L^2(\R)$. It has a unique negative eigenvalue which is simple, and its kernel is given by
\renewcommand{\theequation}{\arabic{equation}}
\setcounter{equation}{9}
\begin{equation}
\label{eq:Ker-Hc}
\Ker(\boH_c) = \Span(\partial_x Q_c).
\end{equation}
Moreover, under the orthogonality conditions
\begin{align}
\label{cond:orthc}
\langle \eps, \partial_x Q_c \rangle_{L^2(\R)^2} = P'(Q_c)(\eps) = 0,
\end{align}
we have
$$H_c(\eps) \geq \Lambda_c \| \eps \|_{X(\R)}^2,$$
where the positive number $\Lambda_c$ depends only and continuously on $c \in (- \sqrt{2}, \sqrt{2}) \setminus \{ 0 \}$. The first orthogonality relation in \eqref{cond:orthc} is related to the invariance by translation of $E$ and $P$, which is reflected in the fact that $\partial_x Q_c$ is in the kernel of $\boH_c$. There is probably more freedom regarding the second orthogonality relation in \eqref{cond:orthc}. Our choice was motivated by the possibility to obtain the quadratic dependence of $c'(t)$ stated in Theorem \ref{thm:orbistab}.

The pair $\eps$ obtained in Theorem \ref{thm:orbistab} satisfies the equation
\begin{equation}
\label{eq:poureps}
\partial_t \eps = J \boH_{c(t)}(\eps) + J \boR_{c(t)} \eps + \big( a'(t) - c(t) \big) \big( \partial_x \eps + \partial_x Q_{c(t)} \big) - c'(t) \partial_c Q_{c(t)},
\end{equation}
where $J$ is the symplectic operator
\begin{equation}
\label{def:J}
J = - 2 S \partial_x := \begin{pmatrix} 0 & - 2 \partial_x \\ - 2 \partial_x & 0 \end{pmatrix},
\end{equation}
and the remainder term $\boR_{c(t)} \eps$ is given by
$$\boR_{c(t)} \eps := E'(Q_{c(t)} + \eps) - E'(Q_{c(t)}) - E''(Q_{c(t)})(\eps).$$

\subsection{Asymptotic stability in the hydrodynamical framework}
\label{sub:hydroasymptstab}

An important part of the paper is devoted to the following theorem, from which we will eventually deduce Theorem \ref{thm:stabasympt_psi}.

\begin{thm}
\label{thm:stabasympt}
Let $\gc \in (- \sqrt{2}, \sqrt{2}) \setminus \{ 0 \}$. There exists a positive constant $\beta_\gc \leq \alpha_\gc$, depending only on $\gc$, with the following properties. Given any $(\eta_0, v_0) \in X(\R)$ such that
$$\big\| (\eta_0, v_0) - Q_{\gc, \ga} \big\|_{X(\R)} \leq \beta_\gc,$$
for some $\ga \in \R$, there exist a number $\gc^* \in (- \sqrt{2}, \sqrt{2}) \setminus \{ 0 \}$ and a map $b \in \boC^1(\R, \R)$ such that the unique global solution $(\eta, v) \in \boC^0(\R, \boN\boV(\R))$ to \eqref{HGP} with initial data $(\eta_0, v_0)$ satisfies
$$\big( \eta(\cdot + b(t), t), v(\cdot + b(t), t) \big) \rightharpoonup Q_{\gc^*} \quad {\rm in} \ X(\R),$$
and
$$b'(t) \to \gc^*,$$
as $t \to + \infty$.
\end{thm}

In order to prove Theorem \ref{thm:stabasympt}, a main step is to substitute the uniform estimates \eqref{eq:modul0} and \eqref{eq:modul1} by suitable convergence estimates. We present the main ingredients in the proof of Theorem \ref{thm:stabasympt} in the next subsections. 

\subsubsection{Construction of a limit profile}
\label{sub:lim-prof}

Let $\gc \in (- \sqrt{2}, \sqrt{2}) \setminus \{ 0 \}$ be fixed and let $(\eta_0, v_0) \in X(\R)$ be any pair satisfying the assumptions of Theorem \ref{thm:stabasympt}. Since $\beta_\gc \leq \alpha_\gc$ in the assumptions of Theorem \ref{thm:stabasympt}, by Theorem \ref{thm:orbistab}, we may consider the unique globally defined solution $(\eta, v)$ to \eqref{HGP} with initial datum $(\eta_0, v_0)$.

We fix an arbitrary sequence of times $(t_n)_{n \in \N}$ tending to $+ \infty$. In view of \eqref{eq:modul0} and \eqref{eq:modul1}, we may assume, going to a subsequence if necessary, that 
there exist $\eps_0^* \in X(\R)$ and $c_0^* \in [- \sqrt{2}, \sqrt{2}]$ such that
\begin{equation}
\label{eq:assump1}
\eps(\cdot, t_n) = \big( \eta(\cdot + a(t_n), t_n), v(\cdot + a(t_n), t_n) \big) - Q_{c(t_n)} \rightharpoonup \eps_0^* \quad {\rm in} \ X(\R),
\end{equation}
and
\begin{equation}
\label{eq:assump2}
c(t_n) \to c_0^*,
\end{equation}
as $n \to + \infty$. In the next two subsections, we will eventually come to the conclusion (see Corollary \ref{cor:cayest}) that necessarily
$$\eps_0^* \equiv 0,$$
by establishing smoothness and rigidity properties for the solution of \eqref{HGP} with initial datum given by $Q_{c_0^*} + \eps_0^*$.

More precisely, we first impose the constant $\beta_\gc$ to be sufficiently small so that, when $\alpha^0$ appearing in Theorem \ref{thm:orbistab} satisfies $\alpha^0 \leq \beta_\gc$, then in view of \eqref{eq:modul0} and \eqref{eq:modul1}, we have
\begin{equation}
\label{eq:lvmh1}
\min \big\{ c(t)^2, a'(t)^2 \big\} \geq \frac{\gc^2}{2}, \qquad \max \big\{ c(t)^2, a'(t)^2 \big\} \leq 1 + \frac{\gc^2}{2}, 
\end{equation}
and also
\begin{equation}
\label{eq:lvmh2}
\big\| \eta_\gc(\cdot) - \eta(\cdot + a(t), t) \big\|_{L^\infty(\R)} \leq \min \Big\{ \frac{\gc^2}{4}, \frac{2 - \gc^2}{64} \Big\},
\end{equation}
for any $t\in \R$. In particular, we deduce that $c_0^* \in (- \sqrt{2}, \sqrt{2}) \setminus \{ 0 \}$ and therefore $Q_{c_0^*}$ is well-defined and different from the black soliton.

It follows from \eqref{eq:modul0} that 
\begin{equation}
\label{eq:encorezero}
\big| c_0^* - \gc \big| \leq A_\gc \beta_\gc,
\end{equation}
and from \eqref{eq:modul0}, \eqref{eq:assump1} and the weak lower semi-continuity of the norm that the function
$$(\eta_0^*, v_0^*) := Q_{c_0^*} + \eps_0^*,$$
satisfies
\begin{equation}
\label{eq:encoreune}
\big\| (\eta_0^*, v_0^*) - Q_\gc \big\|_{X(\R)} \leq A_\gc \beta_\gc + \big\| Q_\gc - Q_{c_0^*} \big\|_{X(\R)}.
\end{equation}
We next impose a supplementary smallness assumption on $\beta_\gc$ so that 
$$\big\| (\eta_0^*, v_0^*) - Q_\gc \big\|_{X(\R)} \leq \alpha_\gc.$$
Applying Theorem \ref{thm:orbistab} yields a unique global solution $(\eta^*, v^*) \in
\boC^0(\R, \boN\boV(\R))$ to \eqref{HGP} with initial data $(\eta_0^*, v_0^*)$, and two maps $c^* \in \boC^1(\R, (- \sqrt{2}, \sqrt{2}) \setminus \{ 0 \})$ and $a^* \in \boC^1(\R, \R)$ such that the function $\eps^*$ defined by
\begin{equation}
\label{def:eps*}
\eps^*(\cdot, t) := \big( \eta^*(\cdot + a^*(t), t), v(\cdot + a^*(t), t) \big) - Q_{c^*(t)},
\end{equation}
satisfies the orthogonality conditions
\begin{equation}
\label{eq:orthobis}
\langle \eps^*(\cdot, t), \partial_x Q_{c^*(t)} \rangle_{L^2(\R)^2} = P'(Q_{c^*(t)})(\eps^*(\cdot, t)) = 0,
\end{equation}
as well as the estimates
\begin{align}
\label{eq:modul0bis}
\big\| \eps^*(\cdot, t) \big\|_{X(\R)} + \big| c^*(t) - \gc \big| & \leq A_\gc \big\| (\eta_0^*, v_0^*) -Q_{\gc} \big\|_{X(\R)},\\
\label{eq:modul1bis}
\big| {c^*}'(t) \big| + \big| a{^*}'(t) - c^*(t) \big|^2 & \leq A_\gc \big\| \eps^*(\cdot, t) \big\|_{X(\R)}^2,
\end{align}
for any $t\in \R$.

We finally restrict further the definition of $\beta_\gc$, if needed, in such a way that \eqref{eq:modul0bis} and \eqref{eq:modul1bis}, together with \eqref{eq:encorezero} and \eqref{eq:encoreune}, imply that
\begin{equation}
\label{eq:lvmh1bis}
\min \big\{ c^*(t)^2, (a^*)'(t)^2 \big\} \geq \frac{\gc^2}{2}, \qquad \max \big\{ c^*(t)^2, (a^*)'(t)^2 \big\} \leq 1 + \frac{\gc^2}{2}, 
\end{equation}
and 
\begin{equation}
\label{eq:lvmh2bis}
\big\| \eta_\gc(\cdot) - \eta^*(\cdot + a^*(t), t) \big\|_{L^\infty(\R)} \leq \min \Big\{ \frac{\gc^2}{4}, \frac{2 - \gc^2}{64} \Big\},
\end{equation}
for any $t\in \R$.

The following proposition, based on the weak continuity of the flow map for the Gross-Pitaevskii equation, allows to improve the convergence properties of the initial data, as stated in \eqref{eq:assump1}, into convergence properties for the flow under \eqref{HGP} and for the modulation parameters.

\begin{prop}
\label{prop:reprod}
Let $t \in \R$ be fixed. Then,
\begin{equation}
\label{sologne1}
\big( \eta(\cdot + a(t_n), t_n + t), v(\cdot + a(t_n), t_n + t) \big) \rightharpoonup \big( \eta^*(\cdot, t), v^*(\cdot, t) \big) \quad {\rm in} \ X(\R),
\end{equation}
while
\begin{equation}
\label{sologne2}
a(t_n + t) - a(t_n) \to a^*(t), \quad {\rm and} \quad c(t_n + t) \to c^*(t),
\end{equation}
as $n \to + \infty$. In particular, we have
\begin{equation}
\label{sologne3}
\eps(\cdot, t_n + t) \rightharpoonup \eps^*(\cdot, t) \quad {\rm in} \ X(\R),
\end{equation}
as $n \to + \infty$.
\end{prop}

\subsubsection{Localization and smoothness of the limit profile}
\label{sub:prop-lim}

In order to prove localization of the limit profile, we rely heavily on a monotonicity formula. 

Let $(\eta, v)$ be as in Theorem \ref{thm:orbistab} and assume that \eqref{eq:lvmh1} and \eqref{eq:lvmh2} hold. Given real numbers $R$ and $t$, we define the quantity
$$I_R(t) \equiv I_R^{(\eta, v)}(t) := \frac{1}{2} \int_\R \big[ \eta v \big](x + a(t), t) \Phi(x - R) \, dx,$$
where $\Phi$ is the function defined on $\R$ by 
\begin{equation}
\label{eq:defiphi}
\Phi(x) := \frac{1}{2} \Big( 1 + \th \big( \nu_\gc x \big) \Big),
\end{equation}
with $\nu_\gc := \sqrt{2 - \gc^2}/8$. The function $I_R(t)$ represents the amount of momentum of $(\eta(\cdot, t), v(\cdot, t))$ located from a (signed) distance $R$ to the right of the soliton.

We have

\begin{prop}
\label{prop:mono}
Let $R \in \R$, $t \in \R$, and $\sigma \in [- \sigma_\gc, \sigma_\gc]$, with 
$\sigma_\gc := (2 - \gc^2)/(4 \sqrt{2})$. Under the above assumptions, we have
\begin{equation}
\label{eq:mono}
\begin{split}
\frac{d}{dt} \big[ I_{R + \sigma t}(t) \big] \geq & \frac{(2 - \gc^2)^2}{2^{11}} \int_\R \big[ (\partial_x \eta)^2 + \eta^2 + v^2 \big](x + a(t), t) \Phi'(x - R - \sigma t) \, dx\\
& - 24 \frac{(2 - \gc^2)^2}{\gc^4} e^{- 2 \nu_\gc |R + \sigma t|}.
\end{split}
\end{equation}
As a consequence, we obtain
\begin{equation}
\label{eq:monobis}	
I_R(t_1) \geq I_R(t_0) - 768 \frac{\sqrt{2 - \gc^2}}{\gc^4} e^{- 2 \nu_\gc |R|},
\end{equation}
for any real numbers $t_0 \leq t_1$.
\end{prop}

Specifying for the limit profile $(\eta^*, v^*)$, we set $I_R^*(t) := I_R^{(\eta^*, v^*)}(t)$ for any $R \in \R$ and any $t \in \R$. We claim

\begin{prop}
\label{prop:local0}
Given any positive number $\delta$, there exists a positive number $R_\delta$, depending only on $\delta$, such that we have
\begin{align*}
\big| I_R^*(t) \big| \leq \delta, & \quad \forall R \geq R_\delta,\\
\big| I_R^*(t) - P(\eta^*, v^*) \big| \leq \delta, & \quad \forall R \leq - R_\delta,
\end{align*}
for any $t \in \R$.
\end{prop}

The proof of Proposition \ref{prop:local0} relies on a contradiction argument. The rough idea is that if some positive quantity $\delta$ of momentum for $(\eta^*, v^*)$ were transferred from time $t = 0$ to time $t = T$ and from the interval $(- \infty, R + a^*(0))$ towards the interval $(R + a^*(T), + \infty)$, then a similar transfer would hold for the function $(\eta, v)$ from time $t = t_n$ to time $t = t_n + T$ and from the interval $(- \infty, R + a(t_n))$ towards the interval $(R + a(t_n + T), + \infty)$, for any sufficiently large $n$. On the other hand, assuming that $t_{n + 1} \geq t_n + T$, the monotonicity formula implies that the momentum for $(\eta, v)$ at time $t_{n + 1}$ and inside the interval $(R + a(t_{n + 1}), + \infty)$ is greater (up to exponentials) than the momentum for $(\eta, v)$ at time $t_n + T$ and inside the interval $(R + a(t_n + T), + \infty)$. The combination of those two information would yield that the momentum for $(\eta, v)$ at time $t_n$ and inside $(R + a(t_n), + \
\infty)$ tends to $+ \infty$ as $n \to + \infty$, which is forbidden by the finiteness of the energy of $(\eta, v)$.

From Proposition \ref{prop:local0}, and using once more Proposition \ref{prop:mono}, we obtain

\begin{prop}
\label{prop:local}
Let $t\in \R$. We have
$$\int_t^{t + 1} \int_\R \big[ (\partial_x \eta^*)^2 + (\eta^*)^2 + (v^*)^2 \big](x + a^*(s), s) e^{2 \nu_\gc |x|} \, dx \, ds \leq \frac{2^{21}}{\gc^4(2-\gc^2)}.$$
\end{prop}

In order to prove the smoothness of the limit profile, we rely on the following smoothing type estimate for localized solutions of the inhomogeneous linear Schr\"odinger equation.

\begin{prop}
\label{prop:smoothing}
Let $\lambda \in \R$ and consider a solution $u \in \boC^0(\R, L^2(\R))$ to the linear Schr\"odinger equation
\renewcommand{\theequation}{LS}
\begin{equation}
\label{eq:LS}
i \partial_t u + \partial_{xx} u = F,
\end{equation}
with $F \in L^2(\R, L^2(\R))$. Then, there exists a positive constant $K_\lambda$, depending only on $\lambda$, such that
\renewcommand{\theequation}{\arabic{equation}}
\setcounter{equation}{30}
\begin{equation}
\label{eq:smoothing}
 \lambda^2 \int_{- T}^T \int_\R |\partial_x u(x, t)|^2 e^{\lambda x} \, dx \, dt \leq K_\lambda \int_{- T - 1}^{T + 1} \int_\R \Big( |u(x, t)|^2 + |F(x, t)|^2 \Big) e^{\lambda x} \, dx \, dt,
\end{equation}
for any positive number $T$.
\end{prop}

Applying Proposition \ref{prop:smoothing} to the derivatives of $\Psi^*$, the solution to \eqref{GP} associated to the solution $(\eta^*, v^*)$ of \eqref{HGP}, and then expressing the information in terms of $(\eta^*, v^*)$, we obtain

\begin{prop}
\label{prop:smooth}
The pair $(\eta^*, v^*)$ is indefinitely smooth and exponentially decaying on $\R \times\R$. Moreover, given any $k \in \N$, there exists a positive constant $A_{k, \gc}$, depending only on $k$ and $\gc$, such that
\begin{equation}
\label{eq:smooth}
\int_\R \big[ (\partial_x^{k + 1} \eta^*)^2 + (\partial^k_x \eta^*)^2 + (\partial_x^k v^*)^2 \big](x + a^*(t), t) e^{2 \nu_\gc |x|} \, dx \leq A_{k, \gc},
\end{equation}
for any $t\in \R$.
\end{prop}

The proof of Proposition \ref{prop:smoothing} and Proposition \ref{prop:smooth}, as well as additional remarks concerning smoothing properties for localized solutions are gathered in Appendix \ref{sec:smoothing}.

\subsubsection{Rigidity for the limit profile}
\label{sub:rigidity}

Our main task is now to show that the limit profile constructed above is exactly a soliton, which amounts to prove that $\eps_0^* \equiv 0$.

Recall from \eqref{eq:poureps} that $\eps^*$ satisfies the equation
\begin{equation}
\label{eq:pourepsbis}
\partial_t \eps^* = J \boH_{c^*(t)}(\eps^*) + J \boR_{c^*(t)} \eps^* + \big( {a^*}'(t) - c^*(t) \big) \big( \partial_x Q_{c^*(t)} + \partial_x \eps^* \big) - {c^*}'(t) \partial_c Q_{c^*(t)}.
\end{equation}
Our strategy is to derive suitable integral estimates on $\eps^*$. Since the linear operator 
$\boH_c$ has a kernel given by $\partial_x Q_c$, it turns out that it is more convenient to
derive first integral estimates for the quantity $\boH_{c^*}(\eps^*)$ (so that the component along the kernel is eliminated) rather than directly on $\eps^*$. This idea was already successfully used by Martel and Merle in \cite{MartMer6} (see also \cite{Martel2}) for the generalized Korteweg-de Vries equation. The smoothness and decay obtained in the previous subsection allow us to perform as many differentiations as we wish.

More precisely, we define the pair
\begin{equation}
\label{def:u*}
u^*(\cdot, t) := S \boH_{c^*(t)}(\eps^*(\cdot, t)).
\end{equation}
Since $S \boH_{c^*(t)}(\partial_x Q_{c^*(t)}) = 0$, we deduce from \eqref{eq:pourepsbis} that
\begin{equation}
\label{eq:pouru*}
\begin{split}
\partial_t u^* = & \ S \boH_{c^*(t)} \big( J S u^* \big) + S \boH_{c^*(t)} \big( J \boR_{c^*(t)} \eps^* \big) - (c^*)'(t) S \boH_{c^*(t)}(\partial_c Q_{c^*(t)})\\
& \ + (c^*)'(t) S \partial_c \boH_{c^*(t)}(\eps^*) + \big( (a^*)'(t) - c^*(t) \big) S \boH_{c^*(t)}(\partial_x \eps^*).
\end{split}
\end{equation}

At spatial infinity, the operator $\boH_c$ is asymptotically of constant coefficients, and therefore almost commutes with $J$. Therefore the linear operator in \eqref{eq:pouru*}, namely $\boH_{c^*} J$, coincides in that limit with the linear operator $J\boH_{c^*}$ appearing in \eqref{eq:pourepsbis}. It is thus not surprising that a monotonicity formula similar in spirit to the monotonicity of the localized momentum for $\eps^*$ (see Proposition \ref{prop:mono}) also holds for $u^*$. More precisely, {\it decreasing further the value of $\beta_\gc$ if necessary}, we obtain

\begin{prop}
\label{prop:monou1}
There exist two positive numbers $A_*$ and $R_*$, depending only on $\gc$, such that we have
\footnote{In \eqref{eq:petitplateau}, we have use the notation
$$\big\| (f, g) \big\|_{X(\Omega)}^2 : = \int_\Omega \Big( (\partial_x f)^2 + f^2 + g^2 \Big),$$
in which $\Omega$ denotes a measurable subset of $\R$.}
\begin{equation}
\label{eq:petitplateau}
\frac{d}{dt} \bigg( \int_\R x u^*_1(x, t) u^*_2(x, t) \, dx \bigg) \geq \frac{2 - \gc^2}{64} \big\| u^*(\cdot, t) \big\|_{X(\R)}^2 - A_* \| u^*(\cdot, t)\|_{X(B(0, R_*))}^2,
\end{equation}
for any $t \in \R$. 
\end{prop}

In order to get rid of the non-positive local term $\| u^*(\cdot, t) \|_{X(B(0, R_*))}^2$ in the right-hand side of \eqref{eq:petitplateau}, we invoke a second monotonicity type formula. If $M$ is a smooth, bounded, two-by-two symmetric matrix-valued function, then 
\begin{equation}
\label{eq:pignon}
\frac{d}{dt} \big\langle M u^*, u^* \big\rangle_{L^2(\R)^2} = 2 \big\langle S M u^*,\boH_{c^*}(J S u^*) \big\rangle_{L^2(\R)^2} + \text{``super-quadratic terms''.}
\end{equation}
For $c \in (- \sqrt{2}, \sqrt{2}) \setminus \{ 0 \}$, let $M_c$ be given by
\begin{equation}
\label{def:Mc}
M_c := \begin{pmatrix} - \frac{c \partial_x \eta_c}{2 (1 - \eta_c)^2} & - \frac{\partial_x \eta_c}{\eta_c} \\ - \frac{\partial_x \eta_c}{\eta_c} & 0 \end{pmatrix}.
\end{equation}
The choice of $M_c$ is motivated by the following key observation.
 
\begin{lemma}
\label{lem:loc-coer}
Let $c \in (- \sqrt{2}, \sqrt{2}) \setminus \{ 0 \}$ and $u \in X^3(\R)$. Then,
\begin{equation}
\label{eq:loc-virial}
\begin{split}
G_c(u) := & 2 \big\langle S M_c u, \boH_c(J S u) \big\rangle_{L^2(\R)^2}\\
= & 2 \int_\R \big( \eta_c + \partial_{xx} \eta_c \big) \Big( u_2 - \frac{c \eta_c}{2 (\eta_c + \partial_{xx} \eta_c)} u_1 - \frac{c \partial_x \eta_c}{2 (1 - \eta_c) (\eta_c + \partial_{xx} \eta_c)} \partial_x u_1 \Big)^2\\
& + \frac{3}{2} \int_\R \frac{\eta_c^2}{\eta_c + \partial_{xx} \eta_c} \Big( \partial_x u_1 - \frac{\partial_x \eta_c}{\eta_c} u_1 \Big)^2.
\end{split}
\end{equation}
\end{lemma}

Notice that the quadratic form $G_c(u)$ in \eqref{eq:loc-virial} is pointwise non-negative (and non-singular) since
$$\eta_c + \partial_{xx}^2 \eta_c = \eta_c \big( 3 - c^2 - 3 \eta_c \big) \geq \frac{c^2}{2} \eta_c > 0.$$
It also follows from \eqref{eq:loc-virial} that
$$\Ker(G_c) = \Span(Q_c).$$
In our situation, $u^* = S\boH_{c^*}(\eps^*)$ is not proportional to $Q_{c^*}$. By the orthogonality relation \eqref{eq:orthobis}, we indeed have $P'(Q_{c^*(t)})(\eps^*) = 0$. Since one has $H_c(\partial_c Q_c) = P'(Q_c)$, it follows that
\begin{equation}
\label{eq:ortho-u*}
0 = \langle H_{c^*}(\partial_c Q_{c^*}), \eps^* \rangle_{L^2(\R)^2} = \langle \boH_{c^*}(\eps^*),\partial_c Q_{c^*} \rangle_{L^2(\R)^2} = \langle u^*, S \partial_c Q_{c^*} \rangle_{L^2(\R)^2}.
\end{equation}
On the other hand,
\begin{equation}
\label{eq:angle}
\big\langle Q_{c^*} , S \partial_c Q_{c^*} \big\rangle = \frac{1}{2} \frac{d}{dc} \big\langle Q_c, S Q_c \big\rangle_{|c = c^*} = 2 \frac{d}{dc} \Big( P(Q_c) \Big)_{|c = c^*} = - 2 \big( 2 - c_*^2 \big)^\frac{1}{2} \neq 0,
\end{equation}
which prevents $u^*$ from being proportional to $Q_{c^*}$. This leads to

\begin{prop}
\label{prop:coer-Gc}
Let $c \in (- \sqrt{2}, \sqrt{2}) \setminus \{ 0 \}$. There exists a positive number $\Lambda_c$, depending only and continuously on $c$, such that
\begin{equation}
\label{eq:coer-Gc}
G_c(u) \geq \Lambda_c \int_\R \big[ (\partial_x u_1)^2 + (u_1)^2 + (u_2)^2 \big] (x) e^{- \sqrt{2} |x|} \, dx,
\end{equation}
for any pair $u \in X^1(\R)$ verifying
\begin{equation}
\label{eq:ortho-u}
\langle u, S \partial_c Q_c \rangle_{L^2(\R)^2} = 0.
\end{equation}
\end{prop}

Coming back to \eqref{eq:pignon}, we can prove

\begin{prop}
\label{prop:monou2}
There exists a positive number $B_*$, depending only on $\gc$, such that
\begin{equation}
\label{eq:moyenplateau}
\begin{split}
\frac{d}{dt} \Big( \big\langle M_{c^*(t)} u^*(\cdot, t), u^*(\cdot, t) \big\rangle_{L^2(\R)^2} \Big) \geq & \frac{1}{B_*} \int_\R \big[ (\partial_x u_1^*)^2 + (u_1^*)^2 + (u_2^*)^2 \big] (x, t) e^{- \sqrt{2} |x|} \, dx\\
& - B_* \big\| \eps^*(., t) \big\|_{X(\R)}^\frac{1}{2} \big\| u^*(\cdot, t) \big\|_{X(\R)}^2,
\end{split}
\end{equation}
for any $t \in \R$.
\end{prop}

Combining Proposition \ref{prop:monou1} and Proposition \ref{prop:monou2} yields

\begin{cor}
\label{cor:bomono}
Set
$$N(t) := \frac{1}{2} \begin{pmatrix} 0 & x \\ x & 0 \end{pmatrix} + A_* B_* e^{\sqrt{2} R_*} M_{c^*(t)}.$$
We have 
\begin{equation}
\label{eq:grandplateau}
\frac{d}{dt} \Big( \langle N(t) u^*(\cdot, t), u^*(\cdot, t) \rangle_{L^2(\R)^2} \Big) \geq \frac{2 - \gc^2}{128} \big\| u^*(\cdot, t) \big\|_{X(\R)}^2,
\end{equation}
for any $t \in \R$. In particular,
\begin{equation}
\label{eq:petitpignon}
\int_{- \infty}^{+ \infty} \big\| u^*(\cdot, t) \big\|_{X(\R)}^2 \, dt < + \infty.
\end{equation}
Therefore, there exists a sequence $(t_k^*)_{k \in \N}$ such that
\begin{equation}
\label{eq:derailleur}
\lim_{k \to + \infty} \big\| u^*(\cdot, t_k^*) \big\|_{X(\R)}^2 = 0.
\end{equation}
\end{cor}

Combining \eqref{eq:derailleur} with the inequality 
$$\big\| \eps^*(\cdot, t) \|_{X(\R)} \leq A_\gc \big\| u^*(\cdot, t) \big\|_{X(\R)},$$
(see \eqref{tremblay}), we obtain
\begin{equation}
\label{eq:derailleur2}
\lim_{k \to + \infty} \big\| \eps^*(\cdot, t_k^*) \big\|_{X(\R)}^2 = 0.
\end{equation}
Combining \eqref{eq:derailleur2} with the orbital stability in Theorem \ref{thm:orbistab},
we are finally led to

\begin{cor}
\label{cor:cayest}
We have
$$\eps_0^* \equiv 0.$$
\end{cor}

\subsubsection{Proof of Theorem \ref{thm:stabasympt} completed}

Let $\gc \in (- \sqrt{2}, \sqrt{2}) \setminus \{ 0 \}$ and let $(\eta_0, v_0)$ be as in the statement of Theorem \ref{thm:stabasympt}. It follows from the analysis in the previous three subsections that, given any sequence of times $(t_n)_{n \in \N}$ converging to $+ \infty$, there exists a subsequence $(t_{n_k})_{k \in \N}$ and a number $c_0^*$ (sufficiently close to $\gc$ as expressed e.g. in \eqref{eq:encorezero}) such that
$$\big( \eta(\cdot + a(t_{n_k}), t_{n_k}), v(\cdot + a(t_{n_k}), t_{n_k}) \big) \rightharpoonup Q_{c_0^*} \quad {\rm in} \ X(\R),$$
as $n \to + \infty$. By a classical argument for sequences, if we manage to prove that $c_0^*$ is independent of the sequence $(t_n)_{n \in \N}$, then it will follow that
\begin{equation}
\label{eq:convw2}
\big( \eta(\cdot + a(t), t), v(\cdot + a(t), t) \big) \rightharpoonup Q_{c_0^*} \quad {\rm in} \ X(\R),
\end{equation}
as $t \to + \infty$.

We argue by contradiction. Assume that for two different sequences $(t_n)_{n \in \N}$ and $(s_n)_{n \in \N}$, both tending to $+ \infty$, we have
\begin{equation}
\label{eq:cona}
\big( \eta(\cdot + a(t_n), t_n), v(\cdot + a(t_n), t_n) \big) \rightharpoonup Q_{c_1^*} \quad {\rm in} \ X(\R),
\end{equation}
and 
\begin{equation}
\label{eq:conb}
\big( \eta(\cdot + a(s_n), s_n), v(\cdot + a(s_n), s_n) \big) \rightharpoonup Q_{c_2^*} \quad {\rm in} \ X(\R),
\end{equation}
as $n \to + \infty$, with $c_1^*\neq c_2^*$ satisfying \eqref{eq:encorezero}. Without loss of generality, we may assume that $c_1^* < c_2^*$ and that the sequences $(t_n)_{n\in \N}$ and $(s_n)_{n\in \N}$ are strictly increasing and nested such that
\begin{equation}
\label{eq:imbriques}
t_n + 1 \leq s_n \leq t_{n + 1} - 1,
\end{equation}
for any $n \in \N$. The contradiction will follow essentially in the same way as for Proposition \ref{prop:local0}.

We set $\delta := P(Q_{c_1^*}) - P(Q_{c_2^*}) > 0$. In order to be able to use \eqref{eq:monobis}, we choose a positive number $R$ sufficiently large so that 
$$768 \frac{\sqrt{2 - \gc^2}}{\gc^4} e^{-2 \nu_\gc |R|} \leq \frac{\delta}{10}.$$ 
In particular, we have from Proposition \ref{prop:mono} and \eqref{eq:imbriques},
\begin{equation}
\label{eq:monoencore}
I_{\pm R}(s_n) \geq I_{\pm R}(t_n) - \frac{\delta}{10} \quad {\rm and} \quad I_{\pm R}(t_{n + 1}) \geq I_{\pm R}(s_n) - \frac{\delta}{10},
\end{equation}
for any $n \in \N$. Increasing the value of $R$ if necessary, we may also assume that
$$\bigg| \frac{1}{2} \int_\R \big( \Phi(x + R) - \Phi(x - R) \big) \eta_{c_i^*} v_{c_i^*}(x) \, dx - P(Q_{c_i^*}) \bigg| \leq \frac{\delta}{10},$$
for $i = 1, 2$ (and with $\Phi$ as in \eqref{eq:defiphi}). In particular, in view of the convergences \eqref{eq:cona} and \eqref{eq:conb}, there exists an integer $n_0$ such that
\begin{equation}
\label{eq:petitpetit}
\big| I_{- R}(t_n) - I_R(t_n) - P(Q_{c_1^*}) \big| \leq \frac{\delta}{5},
\end{equation}
and
\begin{equation}
\label{eq:petitpetitpetit}
\big| I_{- R}(s_n) - I_R(s_n) - P(Q_{c_2^*}) \big| \leq \frac{\delta}{5},
\end{equation}
for any $n \geq n_0$. Combining \eqref{eq:monoencore}, \eqref{eq:petitpetit} and \eqref{eq:petitpetitpetit}, we obtain
$$I_R(s_n) \geq I_R(t_n) + \frac{\delta}{2},$$
for any $n \geq n_0$, from which it follows again by \eqref{eq:monoencore} that
$$I_R(t_{n + 1}) \geq I_R(t_n) + \frac{2 \delta}{5},$$
for any $n \geq n_0$. Therefore, the sequence $(I_R(t_n))_{n \in \N}$ is unbounded, which is the desired contradiction.

At this stage, we have proved that \eqref{eq:convw2} holds, and therefore, in view of the statement of Theorem \ref{thm:stabasympt}, we set $\gc^* := c_0^*$. It is tempting to set also $b(t) := a(t)$, but we have not proved that $a'(t) \to \gc^*$ as $t \to + \infty$. We will actually not try to prove such a statement but rely instead on the weaker form given by \eqref{sologne2} which, once we now know that $a^*(t) = \gc^* t$ since $(\eta^*, v^*) = Q_{\gc^*}$, reads
$$a(t_n + t) - a(t_n) \to \gc^* t,$$
for any fixed $t \in \R$ and any sequence $(t_n)_{n \in \N}$ tending to $+ \infty$. The opportunity to replace the function $a$ by a function $b$ satisfying the required assumptions then follows from the next elementary real analysis lemma. The proof of Theorem \ref{thm:stabasympt} is here completed. \qed

\begin{lemma}
\label{lem:utile33}
Let $c \in \R$ and let $f : \R \to \R$ be a locally bounded function such that 
$$\lim_{x \to + \infty} f(x + y) - f(x) = c y,$$ 
for any $y \in \R$. Then there exists a function $g \in \boC^1(\R, \R)$ such that
$$\lim_{x \to + \infty} g'(x) = c, \quad {\rm and} \quad \lim_{x \to + \infty} |f(x) - g(x)| = 0.$$
\end{lemma}

\begin{proof}
Replacing $f(x)$ by $f(x) - c x$, we may assume that $c = 0$. It then suffices to replace $f$ by its
convolution by any fixed mollifier and the conclusion follows from the Lebesgue dominated convergence theorem. 
\end{proof}

\subsection{Asymptotic stability in the original framework: Proof of Theorem \ref{thm:stabasympt_psi}}

We first define $\delta_\gc$ in such a way that $\| (\eta_0, v_0) -Q_\gc \|_{X(\R)} \leq \beta_\gc$, whenever $d(\Psi^0, U_\gc) \leq \delta_\gc$. We next apply Theorem \ref{thm:stabasympt}
to the solution $(\eta, v) \in \boC^0(\R, \boN \boV(\R))$ to \eqref{HGP} corresponding to 
the solution $\Psi$ to \eqref{GP}. This provides us with a speed $\gc^*$ and a position function $b$. We now construct the phase function $\theta$, and then derive the convergences in the statement of Theorem \ref{thm:stabasympt_psi}. 

We fix a function $\chi \in \boC_c^\infty(\R, [0, 1])$ such that $\chi$ is real, even, and satisfies $\int_\R \chi(x) \, dx = 1$. In view of the expression of $U_{\gc^*}$ in \eqref{form:solc}, we have
$$\int_\R U_{\gc^*}(x) \chi(x) \, dx = i \frac{\gc^*}{\sqrt{2}} \neq 0.$$
Decreasing the value of $\beta_\gc$ if needed, we deduce from orbital stability that
$$\bigg| \int_\R \Psi(x + b(t), t) \chi(x) \, dx \bigg| \geq \frac{|\gc^*|}{2\sqrt{2}} >0,$$
for any $t \in \R$. In particular, there exists a unique $\vartheta : \R \to \R/(2 \pi \Z)$ such that
$$e^{- i \vartheta(t)} \int_\R \Psi(x + b(t), t) \chi(x) \, dx \in i \frac{\gc^*}{\sqrt{2}} \R^+,$$
for any $t \in \R$. Since $b \in \boC_b^1(\R, \R)$, and since both $\partial_x \Psi$ and $\partial_t \Psi$ belong to $\boC_b^0(\R, H_{\rm loc}^{- 1}(\R))$, it follows by the chain rule and transversality that $\vartheta \in \boC_b^1(\R, \R/(2\pi\Z))$. From Theorem \ref{thm:stabasympt} and the definition of $\vartheta$, we also infer that
\begin{equation}
\label{eq:rump1}
\begin{array}{lcll}
e^{- i \vartheta(t)} \partial_x \Psi(\cdot + b(t), t) & \rightharpoonup & \partial_x U_{\gc^*} & {\rm in} \ L^2(\R),\\
1 - \big| e^{- i \vartheta(t)} \Psi(\cdot + b(t), t) \big|^2 & \rightharpoonup & 1 - \big| U_{\gc^*} \big|^2 & {\rm in} \ L^2(\R),\\
e^{-i \vartheta(t)} \Psi(\cdot + b(t), t) & \to & U_{\gc^*} & {\rm in} \ L_{\rm loc}^\infty(\R),
\end{array}
\end{equation}
as $t\to + \infty$. Invoking the weak continuity of the Gross-Pitaevskii flow, as stated in Proposition \ref{prop:w-cont-Psi}, as well as its equivariance with respect to a constant phase shift and the fact that $U_{\gc^*}$ is an exact soliton of speed $\gc^*$, it follows that for any fixed $T \in \R$,
\begin{equation}
\label{eq:rump2}
\begin{array}{lcll}
e^{-i \vartheta(t)} \partial_x \Psi(\cdot + b(t), t + T) & \rightharpoonup & \partial_x U_{\gc^*}(\cdot - \gc^* T) & {\rm in} \ L^2(\R),\\
1 - \big| e^{- i \vartheta(t)} \Psi(\cdot + b(t), t + T) \big|^2 & \rightharpoonup & 1 - \big| U_{\gc^*}(\cdot - \gc^* T) \big|^2 & {\rm in} \ L^2(\R),\\
e^{- i \vartheta(t)} \Psi(\cdot + b(t), t + T) & \to & U_{\gc^*}(\cdot - \gc^* T) & {\rm in} \ L_{\rm loc}^\infty(\R), 
\end{array}
\end{equation}
as $t \to + \infty$. On the other hand, rewriting \eqref{eq:rump1} at time $t + T$, we have 
\begin{equation}
\label{eq:rump3}
\begin{array}{lcll}
e^{- i \vartheta(t + T)} \partial_x \Psi(\cdot + b(t + T), t + T) & \rightharpoonup & \partial_x U_{\gc^*} & {\rm in} \ L^2(\R),\\
1 - \big| e^{- i\vartheta(t + T)} \Psi(\cdot + b(t + T), t + T) \big|^2 & \rightharpoonup & 1 - \big| U_{\gc^*} \big|^2 & {\rm in} \ L^2(\R),\\
e^{- i \vartheta(t + T)} \Psi(\cdot + b(t + T), t + T) & \to & U_{\gc^*} & {\rm in} \ L_{\rm loc}^\infty(\R),
\end{array}
\end{equation}
as $t \to + \infty$. Since we already know by Theorem \ref{thm:stabasympt} that 
\begin{equation}
\label{eq:rump4}
b(t + T) - b(t) \to \gc^* T,
\end{equation}
as $t \to + \infty$, we deduce from \eqref{eq:rump2}, \eqref{eq:rump3} and \eqref{eq:rump4} that
$$\big( e^{i (\vartheta(t) - \vartheta(t + T))} - 1 \big) U_{\gc^*} \to 0 \quad {\rm in} \ L_{\rm loc}^\infty(\R),$$
as $t \to + \infty$. Therefore, we first have
$$\lim_{t \to + \infty} \vartheta(t + T) - \vartheta(t) = 0 \quad {\rm in} \ \R/(2 \pi \Z),$$
but then also in $\R$ for any lifting of $\vartheta$, since we have a global bound on the derivative of $\vartheta$.
 
As for the proof of Theorem \ref{thm:stabasympt}, the conclusion then follows from Lemma \ref{lem:utile33} applied to (any lifting of) $\vartheta$. This yields a function $\theta$ such that $\theta'(t) \to 0$, and $\vartheta(t) - \theta(t) \to 0$ as $t \to + \infty$. In particular, we may substitute $\vartheta(t)$ by $\theta(t)$ in \eqref{eq:rump1}, and obtain the desired conclusions. \qed

\numberwithin{cor}{section}
\numberwithin{equation}{section}
\numberwithin{lemma}{section}
\numberwithin{prop}{section}
\numberwithin{thm}{section}
\section{Proofs of localization and smoothness of the limit profile}

\subsection{Proof of Proposition \ref{prop:mono}}

First, we deduce from \eqref{HGP} the identity
\begin{equation}
\label{consint}
\begin{split}
\frac{d}{dt} \big[ I_{R + \sigma t}(t) \Big] = & - \frac{1}{2} (a'(t) + \sigma) \int_\R \big[ \eta v \big](x + a(t), t) \Phi'(x - R - \sigma t) \, dx\\
& + \frac{1}{2} \int_\R \Big[ (1 - 2 \eta) v^2 + \frac{\eta^2}{2} + \frac{(3 - 2 \eta) (\partial_x \eta)^2}{4(1 - \eta)^2} \Big](x + a(t), t) \Phi'(x - R - \sigma t) \, dx\\
& + \frac{1}{4} \int_\R \big[ \eta + \ln(1 - \eta) \big](x + a(t), t) \Phi'''(x - R - \sigma t) \, dx.
\end{split}
\end{equation}

Our goal is to provide a lower bound for the integrand in the right-hand side of \eqref{consint}. We will decompose the domain of integration into two parts, $[- R_0, R_0]$ and its complement, where $R_0$ is to be defined below. On $[- R_0, R_0]$, we will bound the integrand pointwise from below by a positive quadratic form in $(\eta, v)$. Exponentially small error terms will arise from integration on $\R \setminus [- R_0, R_0]$.

First notice that
$$\eta_\gc \leq \nu_\gc^2 \quad {\rm if} \quad \ch^2 \Big( \frac{\sqrt{2 - \gc^2}}{2} x \Big) \geq 32,$$
i.e., if
$$|x| \geq R_0 := \frac{2}{\sqrt{2 - \gc^2}} \ch^{- 1}(4 \sqrt{2}).$$
In particular, we infer from \eqref{eq:lvmh2} that
\begin{equation}
\label{eq:dior1}
\big| \eta(x + a(t), t) \big| \leq 2 \nu_\gc^2,
\end{equation}
for any $x \in [- R_0, R_0]$. Elementary real analysis and \eqref{eq:dior1} then imply that
\begin{equation}
\label{eq:dior2}
\big| \big[ \eta + \ln(1 - \eta) \big](x + a(t), t) \big| \leq \eta^2(x + a(t), t),
\end{equation}
for any $x \in [- R_0, R_0]$. Next, notice that the function $\Phi$ satisfies the inequality
\begin{equation}
\label{eq:dior0}
|\Phi'''| \leq 4 \nu_\gc^2 \Phi'. 
\end{equation}
Finally, in view of the bound \eqref{eq:lvmh1} on $a'(t)$ and the definition of $\sigma_\gc$, we obtain that
\begin{equation}
\label{eq:dior3}
\big| a'(t) + \sigma \big|^2 \leq \frac{3}{2} + \frac{\gc^2}{4}.
\end{equation}
Taking into account \eqref{eq:dior1}, \eqref{eq:dior2}, \eqref{eq:dior0} and \eqref{eq:dior3}, 
we may bound the integrand of \eqref{consint} on $[- R_0, R_0]$ from below by
$$\bigg[ \Big( \frac{1}{2} - 2 \nu_\gc^2 \Big) v^2 + \Big( \frac{1}{4} - \nu_\gc^2 \Big) \eta^2 - \sqrt{\frac{3}{8} + \frac{\gc^2}{16}} |\eta v| + \frac{1}{4} (\partial_x \eta)^2 \bigg](x + a(t), t) \Phi'(x - R - \sigma t).$$
Set $a := 1/4 - \nu_\gc^2 = 7/32 + \gc^2/64$ and $b := \sqrt{3/8 + \gc^2/16}$. In the above quadratic form, we may write
$$a \eta^2 - b |\eta v| + 2 a v^2 = \frac{b}{2 \sqrt{2}} \big( |\eta| - \sqrt{2} |v| \big)^2 + \Big( a - \frac{b}{2 \sqrt{2}} \Big) \big( \eta^2 + 2 v^2 \big) \geq \Big( a - \frac{b}{2 \sqrt{2}} \Big) \big( \eta^2 + 2 v^2 \big),$$
and compute
$$a - \frac{b}{2 \sqrt{2}} = \frac{a^2 - \frac{b^2}{8}}{a + \frac{b}{2 \sqrt{2}}} \geq 2 a^2 - \frac{b^2}{4} = \frac{(2 - \gc^2)^2}{2^{11}}.$$

We next consider the case $x \notin [- R_0, R_0]$. In that region, we simply bound the positive function $\Phi'(x - R - \sigma t)$ by a constant, 
$$\Phi'(x - R - \sigma t) \leq 2 \nu_\gc e^{- 2 \nu_\gc |R + \sigma t - R_0|} \leq 8 \nu_\gc e^{- 2 \nu_\gc |R + \sigma t|},$$
and control the remaining integral using the energy. More precisely, notice that for those $x$, $\eta_\gc \geq \nu_\gc^2$ and therefore by \eqref{eq:lvmh2}, we also have $\eta \geq 0$ (in the remaining part of the proof when we refer to $\eta$ or $v$ we mean the value at the point $(x + a(t), t)$). Next, we have $1 - \eta_\gc \geq \gc^2/2$, and therefore by \eqref{eq:lvmh2} also, $1 - \eta \geq \gc^2/4$. Finally, recall that $|a'(t) + \sigma|/2 \leq \sqrt{2}/2$, and that \eqref{eq:dior0} holds, so that combining the previous estimates and elementary real analysis, we may bound the integrand in the right-hand side of \eqref{consint} by
$$\bigg[ \Big( 4 + (2 - \gc^2) \ln \Big( \frac{\gc^2}{4} \Big) \Big) \eta^2 + 8 v^2 + \frac{48}{\gc^4} (\partial_x \eta)^2 \bigg] \nu_\gc e^{- 2 \nu_\gc |R + \sigma t|}.$$
Conclusion \eqref{eq:mono} follows from integration and a comparison with the energy of $(\eta, v)$, together with the explicit value $E(Q_\gc) = (2 - \gc^2)^\frac{3}{2}/3$ (see e.g. \cite{BetGrSm1}).

It remains to prove \eqref{eq:monobis}. For that purpose, we distinguish two cases, depending on the sign of $R$. If $R \geq 0$, we integrate \eqref{eq:mono} from $t = t_0$ to $t = (t_0 + t_1)/2$ with the choice $\sigma = \sigma_\gc$ and $R = R - \sigma_\gc t_0$, and then from $t = (t_0 + t_1)/2$ to $t = t_1$ with the choice $\sigma = - \sigma_\gc$ and $R = R + \sigma_\gc t_1$. In total, we hence integrate on a broken line starting and ending at a distance $R$ from the soliton. If $R \leq 0$, we argue similarly, choosing first $\sigma = - \sigma_\gc$, and next $\sigma = \sigma_\gc$. This yields \eqref{eq:monobis}, and completes the proof of Proposition \ref{prop:mono}. \qed

\subsection{Proof of Proposition \ref{prop:local0}}

We argue by contradiction and assume that there exists a positive number $\delta_0$ such that, for any positive number $R_{\delta_0}$, there exist two numbers $R \geq R_{\delta_0}$ and $ t \in \R$ such that either $|I_R^*(t)| \geq \delta_0$ or $|I_R^*(t) - P(\eta^*, v^*)| \geq \delta_0$. Since at time $t = 0$, we have $\lim_{R \to +\infty} I_R^*(0) = \lim_{R \to - \infty} I_R^*(0) - P(\eta^*, v^*) = 0$, we first fix $R_{\delta_0} > 0$ such that
\begin{equation}
\label{eq:sam0}
|I_R^*(0)| + |I_{- R}^*(0) - P(\eta^*, v^*)| \leq \frac{\delta_0}{4} \quad {\rm and} \quad 768 \frac{\sqrt{2 - \gc^2}}{\gc^4} e^{- 2\nu_\gc R} \leq \frac{\delta_0}{32},
\end{equation}
for any $R \geq R_{\delta_0}$. We next fix $R > 0$ and $t \in \R$ obtained from the contradiction assumption for that choice of $R_{\delta_0}$, so that either $|I_R^*(t)| \geq \delta_0$ or $|I_R^*(t) - P(\eta^*, v^*)| \geq \delta_0$. In the sequel, we assume that $I_R^*(t) \geq \delta_0$ holds, the three other cases would follow in a very similar manner. In particular, we infer from \eqref{eq:sam0} that
$$I_R^*(t) \geq \delta_0 \geq \frac{\delta_0}{4} + \frac{\delta_0}{16} \geq I_R^*(0) + 1536 \frac{\sqrt{2 - \gc^2}}{\gc^4} e^{- 2 \nu_\gc R},$$
and therefore it follows from the monotonicity formula in Proposition \ref{prop:mono}, applied to $(\eta^*, v^*)$, that $t > 0$. Finally, we fix $R'\geq R$ such that
\begin{equation}
\label{eq:sam1}
\big| I_{- R'}^*(t) - P(\eta^*,v^*) \big| \leq \frac{\delta_0}{4}.
\end{equation}
Since $R' \geq R$, we also deduce from \eqref{eq:sam0} that
\begin{equation}
\label{eq:sam2}
\big| I_{- R'}^*(0) - P(\eta^*, v^*) \big| \leq \frac{\delta_0}{4} \quad {\rm and} \quad 768 \frac{\sqrt{2 - \gc^2}}{\gc^4} e^{- 2 \nu_\gc R'} \leq \frac{\delta_0}{32}.
\end{equation}
Combining the inequality $|I_R^*(t)| \geq \delta_0$ with \eqref{eq:sam0}, \eqref{eq:sam1} and \eqref{eq:sam2}, we obtain
$$\big| I_{- R'}^*(t) - I_R^*(t) - P(\eta^*, v^*) \big| \geq \frac{3 \delta_0}{4} \quad {\rm and} \quad \big| I_{- R'}^*(0) - I_R^*(0) - P(\eta^*, v^*) \big| \leq \frac{\delta_0}{2},$$
and therefore
$$\Big| \big( I_{- R'}^*(0) - I_R^*(0) \big) - \big( I_{- R'}^*(t) - I_R^*(t) \big) \Big| \geq \frac{\delta_0}{4}.$$
Since the integrands of the expressions between parenthesis are localized in space, we deduce from Proposition \ref{prop:reprod} that there exists an integer $n_0$ such that
$$\Big| \big( I_{- R'}(t_n) - I_R(t_n) \big) - \big( I_{- R'}(t_n + t) - I_R(t_n + t) \big) \Big| \geq \frac{\delta_0}{8},$$
for any $n \geq n_0$. Rearranging the terms in the previous inequality yields
\begin{equation}
\label{eq:sam3}
\max \Big\{ \big| I_{- R'}(t_n) - I_{- R'}(t_n + t) \big|, \big| I_R(t_n) - I_R(t_n + t) \big| \Big\} \geq \frac{\delta_0}{16}.
\end{equation}
On the other hand, since $t \geq 0$, by the monotonicity formula in Proposition \ref{prop:mono}, \eqref{eq:sam0} and \eqref{eq:sam2}, we have
$$I_{- R'}(t_n) - I_{- R'}(t_n + t) \leq \frac{\delta_0}{32} \quad {\rm and} \quad I_R(t_n) - I_R(t_n + t) \leq \frac{\delta_0}{32},$$
and therefore we deduce from \eqref{eq:sam3} that, given any $n \geq n_0$,
$${\rm either} \quad I_{- R'}(t_n + t) - I_{- R'}(t_n) \geq \frac{\delta_0}{16}, \quad {\rm or} \quad I_R(t_n + t) - I_R(t_n) \geq \frac{\delta_0}{16}.$$
In particular, there exists an increasing sequence $(n_k)_{k \in \N}$ such that $t_{n_{k + 1}} \geq t_{n_k} + t$ for any $k \in \N$, and either
\begin{equation}
\label{eq:sami0}
I_R(t_{n_k} + t) - I_R(t_{n_k}) \geq \frac{\delta_0}{16},
\end{equation}
for any $k \in \N$, or
$$I_{- R'}(t_{n_k} + t) - I_{- R'}(t_{n_k}) \geq \frac{\delta_0}{16},$$
for any $k \in \N$. In the sequel, we assume that \eqref{eq:sami0} holds, here also the other case would follow in a very similar manner. Since $t_{n_{k + 1}} \geq t_{n_k} + t$, we obtain by the monotonicity formula of Proposition \ref{prop:mono}, \eqref{eq:sam0} and \eqref{eq:sami0}, that
\begin{equation}
\label{eq:sami1}
I_R(t_{n_{k + 1}}) \geq I_R(t_{n_k + t}) - \frac{\delta_0}{32} \geq I_R(t_{n_k}) + \frac{\delta_0}{32},
\end{equation}
for any $k \in \N$. On the other hand, we have 
$$\big| I_R(t_{n_k}) \big| \leq \frac{1}{2} \int_\R \big| \eta(x, t_{n_k}) \big| \big| v(x, t_{n_k}) \big| \, dx \leq \frac{1}{4} \int_\R \big( |\eta(x, t_{n_k})|^2 + |v(x, t_{n_k})|^2 \big) \, dx \leq \frac{2}{\gc^2} E(\eta, v),$$
where the last term does not depend on $k$ by conservation of energy. This yields a contradiction with \eqref{eq:sami1}. \qed 

\subsection{Proof of Proposition \ref{prop:local}}

Let $s \in \R$ and $R \geq 0$ be arbitrary. Integrating \eqref{eq:mono} of Proposition \ref{prop:mono}, and choosing successively $\sigma = \sigma_\gc$ and $\sigma = - \sigma_\gc$, we infer that we have both
$$I_R^*(s) \leq I_{R + \sigma_c \tau}^*(s + \tau) + 768 \frac{\sqrt{2 - \gc^2}}{\gc^4} e^{- 2 \nu_\gc R},$$
and
$$I_R^*(s) \geq I_{R + \sigma_c \tau}^*(s - \tau) - 768 \frac{\sqrt{2 - \gc^2}}{\gc^4} e^{- 2 \nu_\gc R},$$
for each positive number $\tau$. Taking the limit as $\tau \to + \infty$ in the previous two inequalities, we deduce from Proposition \ref{prop:local0} that 
$$\big| I_R^*(s) \big| \leq 768 \frac{\sqrt{2 - \gc^2}}{\gc^4} e^{- 2 \nu_\gc R},$$
for any $s \in \R$ and $R \geq 0$. Similarly, we obtain 
$$\big| I_R^*(s) - P(\eta^*, v^*) \big| \leq 768 \frac{\sqrt{2 - \gc^2}}{\gc^4} e^{- 2 \nu_\gc |R|},$$
for any $s \in \R$ and $R \leq 0$. Therefore, integrating \eqref{eq:mono} from $t$ to $t + 1$ with the choice $\sigma = 0$ yields
$$\int_t^{t + 1} \int_\R \big[ (\partial_x \eta^*)^2 + (\eta^*)^2 + (v^*)^2 \big](x + a^*(s), s) \Phi'(x - R) \, dx \, ds \leq 3 \frac{2^{14}}{\gc^4} \Big( 1 + \frac{64}{(2 - \gc^2)^\frac{3}{2}} \Big) e^{- 2 \nu_\gc |R|},$$
for any $R \in \R$. Since we have
$$\lim_{R \to \pm \infty} e^{2 \nu_\gc |R|} \Phi'(x - R) = 2 \nu_\gc e^{\pm 2 \nu_\gc x},$$
for any $x \in \R$, the conclusion follows from the Fatou lemma, the inequality
$$e^{2 \nu_\gc |x|} \leq e^{- 2 \nu_\gc x} + e^{2 \nu_\gc x},$$
and elementary real estimates. \qed

\section{Proofs of the rigidity properties for the limit profile}
\label{sec:rigidity}

\subsection{Proof of Proposition \ref{prop:monou1}}

In order to establish inequality \eqref{eq:petitplateau}, we first check that we are allowed to differentiate the quantity
$$\boI^*(t) := \int_\R x u^*_1(x, t) u^*_2(x, t) \, dx,$$
in the right-hand side of \eqref{eq:petitplateau}. This essentially follows from Proposition \ref{prop:smooth}. Combining \eqref{eq:smooth} with the explicit formulae for $\eta_c$ and $v_c$ in \eqref{form:etavc}, we indeed derive the existence of a positive number $A_{k, \gc}$ such that
\begin{equation}
\label{marne}
\int_\R \Big( \big( \partial_x^k \eps_\eta^*(x, t) \big)^2 + \big( \partial_x^k \eps_v^*(x, t) \big)^2 \Big) e^{2 \nu_\gc |x|} \, dx \leq A_{k, \gc},
\end{equation}
for any $k \in \N$ and any $t \in \R$. In view of the formulae for $u^*$ in \eqref{def:u*} and for $\boH_c$ in \eqref{def:boHc}, a similar estimate holds for $u^*$, for a further choice of the constant $A_{k, \gc}$. In view of \eqref{eq:pouru*}, this is enough to define properly the quantity $\boI^*$ and establish its differentiability with respect to time. Moreover, we can compute
\begin{equation}
\label{paris}
\begin{split}
\frac{d}{dt} \Big( \boI^* \Big) = & - 2 \int_\R \mu \big\langle \boH_{c^*}(\partial_x u^*), u^* \big\rangle_{\R^2} + \int_\R \mu \big\langle \boH_{c^*} \big( J \boR_{c^*} \eps^* \big), u^* \big\rangle_{\R^2} + \big( c^* \big)' \int_\R \mu \big\langle \partial_c \boH_{c^*}(\eps^*), u^* \big\rangle_{\R^2}\\
& - \big( c^* \big)' \int_\R \mu \big\langle \boH_{c^*}(\partial_c Q_{c^*}), u^* \big\rangle_{\R^2} + \big( (a^*)' - c^* \big) \int_\R \mu \big\langle \boH_{c^*}(\partial_x \eps^*), u^* \big\rangle_{\R^2},
\end{split}
\end{equation}
where we have set $\mu(x) = x$ for any $x \in \R$. In particular, the proof of Proposition \ref{prop:monou1} reduces to estimate each of the four integrals in the right-hand side of \eqref{paris}.

We split the proof into five steps. Concerning the first integral, we have

\begin{step}
\label{I1}
There exist two positive numbers $A_1$ and $R_1$, depending only on $\gc$, such that
\begin{equation}
\label{charenton}
\boI_1^*(t) := - 2 \int_\R \mu \big\langle \boH_{c^*}(\partial_x u^*), u^* \big\rangle_{\R^2} \geq \frac{2 - \gc^2}{16} \big\| u^*(\cdot, t) \big\|_{X(\R)}^2 - A_1 \big\| u^*(\cdot, t) \big\|_{X(B(0, R_1))}^2,
\end{equation}
for any $t \in \R$.
\end{step}

In order to prove inequality \eqref{charenton}, we replace the operator $\boH_{c^*}$ in the definition of $\boI_1^*(t)$ by its explicit formula (see \eqref{def:boHc}), and we integrate by parts to obtain
$$\boI_1^*(t) = \int_\R \iota_1^*(x, t) \, dx,$$
with
\begin{align*}
\iota_1^* = & \frac{1}{4} \Big( \frac{3 \partial_x \mu}{1 - \eta_{c^*}} - \frac{\mu \partial_x \eta_{c^*}}{(1 - \eta_{c^*})^2} \Big) (\partial_x u_1^*)^2 - c^* \partial_x \Big( \frac{\mu}{1 - \eta_{c^*}} \Big) u_1^* u_2^* + \partial_x \big( \mu (1 - \eta_{c^*}) \big) (u_2^*)^2\\
& + \frac{1}{4} \partial_x \Big( \mu \Big( 2 - \frac{\partial_{xx} \eta_{c^*}}{(1 - \eta_{c^*})^2} - \frac{(\partial_x \eta_{c^*})^2}{(1 - \eta_{c^*})^3} \Big) - \partial_x \Big( \frac{\partial_x \mu}{1 - \eta_{c^*}} \Big) \Big) (u_1^*)^2.
\end{align*}
Here, we have used the identity
$$\frac{c^*}{2} + v_{c^*} = \frac{c^*}{2 (1 - \eta_{c^*})},$$
so as to simplify the factor in front of $u_1^* u_2^*$. Since $\mu(x) = x$, the integrand $\iota_1^*$ may also be written as
\begin{align*}
\iota_1^* = & \frac{1}{4} \Big( \frac{3}{1 - \eta_{c^*}} - \frac{x \partial_x \eta_{c^*}}{(1 - \eta_{c^*})^2} \Big) (\partial_x u_1^*)^2 - c^* \Big( \frac{1 - \eta_{c^*} + x \partial_x \eta_{c^*}}{(1 - \eta_{c^*})^2} \Big) u_1^* u_2^* + \big( 1 - \eta_{c^*} - x \partial_x \eta_{c^*} \big) (u_2^*)^2\\
& + \frac{1}{4} \bigg( 2 - \frac{2 \partial_{xx} \eta_{c^*}}{(1 - \eta_{c^*})^2} - \frac{3 (\partial_x \eta_{c^*})^2}{(1 - \eta_{c^*})^3} - x \Big( \frac{\partial_{xxx} \eta_{c^*}}{(1 - \eta_{c^*})^2} + \frac{4 (\partial_x \eta_{c^*}) (\partial_{xx} \eta_{c^*})}{(1 - \eta_{c^*})^3} + \frac{3 (\partial_x \eta_{c^*})^3}{(1 - \eta_{c^*})^4} \Big) \bigg) (u_1^*)^2.
\end{align*}
Given a small positive number $\delta$, we next rely on the exponential decay of the function $\eta_c$ and its derivatives to guarantee the existence of a radius $R$, depending only on $\gc$ and $\delta$ (in view of the bound on $c^* - \gc$ in \eqref{eq:modul0bis}), such that
\begin{align*}
\iota_1^*(x, t) \geq & \frac{3}{4} \big( \partial_x u_1^*(x, t) \big)^2 + \frac{1}{2} u_1^*(x, t)^2 - c^*(t) u_1(x, t) u_2(x, t) + u_2^*(x, t)^2\\
& - \delta \big( (\partial_x u_1^*(x, t))^2 + u_1^*(x, t)^2 + u_2^*(x, t)^2 \big)\\
\geq & \Big( \frac{3}{4} - \delta \Big) \big( \partial_x u_1^*(x, t) \big)^2 + \Big( \frac{1}{2} - \frac{|c^*(t)|}{2 \sqrt{2}} - \delta \Big) u_1^*(x, t)^2 + \Big( 1 - \frac{|c^*(t)|}{\sqrt{2}} - \delta \Big) u_2^*(x, t)^2,
\end{align*}
when $|x| \geq R$. In this case, it is enough to choose $\delta = (2 - \gc^2)/32$ and fix the number $R_1$ according to the value of the corresponding $R$, to obtain
\begin{equation}
\label{saint-maur}
\int_{|x| \geq R_1} \iota_1^*(x, t) \, dx \geq \frac{2 - \gc^2}{16} \int_{|x| \geq R_1} \big( (\partial_x u_1^*(x, t))^2 + u_1^*(x, t)^2 + u_2^*(x, t)^2 \big) \, dx.
\end{equation}
On the other hand, it follows from \eqref{form:etavc}, and again \eqref{eq:modul0bis}, that
$$\int_{|x| \leq R_1} \iota_1^*(x, t) \, dx \geq \Big( \frac{2 - \gc^2}{16} - A_1 \Big) \int_{|x| \leq R_1} \big( (\partial_x u_1^*(x, t))^2 + u_1^*(x, t)^2 + u_2^*(x, t)^2 \big) \, dx,$$
for a positive number $A_1$ depending only on $\gc$. Combining with \eqref{saint-maur}, we obtain \eqref{charenton}.

We next turn to the second integral in the right-hand side of \eqref{paris}.

\begin{step}
\label{I2}
There exist two positive numbers $A_2$ and $R_2$, depending only on $\gc$, such that
\begin{equation}
\label{joinville}
\big| \boI_2^*(t) \big| := \bigg| \int_\R \mu \big\langle \boH_{c^*} \big( J \boR_{c^*} \eps^* \big), u^* \big\rangle_{\R^2} \bigg| \leq \frac{2 - \gc^2}{64} \big\| u^*(\cdot, t) \big\|_{X(\R)}^2 + A_2 \big\| u^*(\cdot, t) \big\|_{X(B(0, R_2))}^2,
\end{equation}
for any $t \in \R$.
\end{step}

Given a small positive number $\delta$, there exists a radius $R$, depending only on $\delta$ and $\gc$, such that
\begin{equation}
\label{champigny}
|x| \leq \delta e^\frac{\nu_\gc |x|}{2},
\end{equation}
for any $|x| \geq R$. As a consequence, we can estimate the integral $\boI_2^*(t)$ as
\begin{equation}
\label{champs}
\begin{split}
\big| \boI_2^*(t) \big| \leq & R \int_{|x| \leq R} \big| \boH_{c^*(t)} \big( J \boR_{c^*(t)} \eps^* \big)(x, t) \big| \big| u^*(x, t) \big| \, dx\\
& + \delta \int_{|x| \geq R} \big| \boH_{c^*(t)} \big( J \boR_{c^*(t)} \eps^* \big)(x, t) \big| \big| u^*(x, t) \big| e^\frac{\nu_\gc |x|}{2} \, dx.
\end{split}
\end{equation}
In order to estimate the two integrals in the right-hand side of \eqref{champs}, we first deduce from \eqref{def:boHc} the existence of a positive number $A_\gc$, depending only on $\gc$, again by \eqref{eq:modul0bis}, such that, given any pair $\eps \in H^2(\R) \times L^2(\R)$, we have
$$\big| \boH_{c^*}(\eps) \big| \leq A_\gc \Big( \big| \partial_{xx} \eps_\eta \big| + \big| \partial_x \eps_\eta \big| + \big| \eps_\eta \big| + \big| \eps_v \big| \Big).$$
In view of \eqref{def:J}, it follows that
\begin{equation}
\label{torcy}
\big| \boH_{c^*}(J \eps) \big| \leq 2 A_\gc \Big( \big| \partial_{xxx} \eps_v \big| + \big| \partial_{xx} \eps_v \big| + \big| \partial_x \eps_v \big| + \big| \partial_x \eps_\eta \big| \Big),
\end{equation}
when $\eps \in H^1(\R) \times H^3(\R)$.

On the other hand, given an integer $\ell$, we can apply the Leibniz rule to the second identity in \eqref{def:Reps} to compute
\begin{equation}
\label{noisy1}
\big| \partial_x^\ell [\boR_{c^*} \eps^*]_v \big| \leq \sum_{k = 0}^\ell \binom{\ell}{k} \big| \partial_x^k \eps_\eta^* \big| \big| \partial_x^{\ell - k} \eps_v^* \big| \leq K_\ell \sum_{k = 0}^\ell \big| \partial_x^k \eps^* \big|^2,
\end{equation}
where $K_\ell$ refers to some constant depending only on $\ell$. Similarly, we can combine the Leibniz rule with \eqref{form:etavc}, \eqref{eq:modul0bis} and \eqref{eq:lvmh2bis} to obtain

\begin{equation}
\label{noisy2}
\big| \partial_x [\boR_{c^*} \eps^*]_\eta \big| \leq A_\gc \Big( \sum_{k = 0}^1 \big| \partial_x^k \eps_v^* \big|^2 + \sum_{k = 0}^3 \big| \partial_x^k \eps_\eta^* \big|^2 + \big| \partial_x \eps_\eta^* \big|^3 \Big).
\end{equation}
Here, we have also applied the Sobolev embedding theorem to bound the norm $\| \eps_\eta^*(\cdot, t) \|_{L^\infty(\R)}$ by $A_\gc \alpha_0$ according to \eqref{eq:modul0bis}. Combining with \eqref{torcy}, we are led to
$$\big\| \boH_{c^*(t)} \big( J \boR_{c^*(t)} \eps^* \big)(\cdot, t) \big\|_{L^2(\R)^2}^2 \leq A_\gc \Big( \big\| \partial_x \eps_\eta^*(\cdot, t) \big\|_{L^6(\R)}^6 + \sum_{k = 0}^3 \big\| \partial_x^k \eps^*(\cdot, t) \big\|_{L^4(\R)^2}^4 \Big).$$
At this stage, we invoke again the Sobolev embedding theorem to write
\begin{equation}
\label{detroit1}
\int_\R \big( \partial_x^\ell f \big)^{2 p} = (- 1)^\ell \int_\R f \, \partial_x^\ell \Big( \big( \partial_x^\ell f \big)^{2 p - 1} \Big) \leq K \big\| f \big\|_{L^2(\R)} \big\| f \big\|_{H^{2 \ell + 1}(\R)}^{2 p - 1},
\end{equation}
for any $\ell \in \N$, any $p \geq 1$, and any $f \in H^{2 \ell + 1}(\R)$. Combining with \eqref{eq:modul0bis}, it follows that
\begin{equation}
\label{detroit2}
\begin{split}
\big\| \boH_{c^*(t)} \big( J \boR_{c^*(t)} \eps^* \big)(\cdot, t) \big\|_{L^2(\R)^2}^2 \leq & K \big\| \eps^*(\cdot, t) \big\|_{L^2(\R)^2} \Big( \big\| \eps_\eta^*(\cdot, t) \big\|_{H^3(\R)}^5 + \big\| \eps^*(\cdot, t) \big\|_{H^7(\R)^2}^3 \Big)\\
\leq & A_\gc \big\| \eps^*(\cdot, t) \big\|_{L^2(\R)^2}^2 \Big( \big\| \eps_\eta^*(\cdot, t) \big\|_{H^7(\R)}^\frac{5}{2} + \big\| \eps^*(\cdot, t) \big\|_{H^{15}(\R)^2}^\frac{3}{2} \Big).
\end{split}
\end{equation}
Since
\begin{equation}
\label{lognes}
\big\| \partial_x^\ell \eps^*(\cdot, t) \big\|_{L^2(\R)^2}^2 \leq \int_\R e^{2 \nu_\gc |x|} \big( \partial_x^\ell \eps^*(x, t) \big)^2 \, dx,
\end{equation}
we can invoke \eqref{marne} to conclude that
\begin{equation}
\label{noisiel}
\big\| \boH_{c^*(t)} \big( J \boR_{c^*(t)} \eps^* \big)(\cdot, t) \big\|_{L^2(\R)} \leq A_\gc \big\| \eps^*(\cdot, t) \big\|_{L^2(\R)^2}.
\end{equation}

On the other hand, we deduce from \eqref{torcy}, \eqref{noisy1} and \eqref{noisy2} as before that
$$\Big\| \boH_{c^*(t)} \big( J \boR_{c^*(t)} \eps^* \big)(\cdot, t) e^\frac{\nu_\gc |\cdot|}{2} \Big\|_{L^2(\R)}^2 \leq A_\gc \bigg( \int_\R \big( \partial_x \eps_\eta^*(x, t) \big)^6 e^{\nu_\gc |x|} \, dx + \sum_{k = 0}^3 \int_\R \big| \partial_x \eps^*(x, t) \big|^4 e^{\nu_\gc |x|} \, dx \bigg).$$
We also invoke the Sobolev embedding theorem to write
\begin{align*}
\int_\R \big( \partial_x^\ell f(x) \big)^{2 p} e^{\nu_\gc x} \, dx = & (- 1)^\ell \int_\R f(x) \partial_x^\ell \Big( \big( \partial_x^\ell f(x) \big)^{2 p - 1} e^{\nu_\gc x} \Big) \, dx\\
\leq & A_\gc \big\| f \big\|_{L^2(\R)} \big\| f \big\|_{H^{2 \ell + 1}(\R)}^{2 p - 2} \big\| f e^{\nu_\gc \cdot} \big\|_{H^{2 \ell + 1}(\R)}\\
\leq & A_\gc \big\| f \big\|_{L^2(\R)}^{p - 1} \big\| f \big\|_{H^{4 \ell + 3}(\R)}^{p - 1} \big\| f e^{\nu_\gc \cdot} \big\|_{H^{2 \ell + 1}(\R)},
\end{align*}
for any $\ell \in \N$, any $p \geq 2$, and any $f \in H^{4 \ell + 3}(\R)$, with $f e^{\nu_\gc |\cdot|} \in H^{2 \ell + 1}(\R)$. Since
\begin{equation}
\label{eq:exp}
e^{\nu_\gc |x|} \leq e^{\nu_\gc x} + e^{-\nu_\gc x} \leq 2 e^{\nu_\gc |x|},
\end{equation}
for any $x \in \R$, the same estimate holds with $e^{\nu_\gc |x|}$ replacing $e^{\nu_\gc x}$. As a consequence, we deduce as before from \eqref{eq:modul0bis}, \eqref{marne} and \eqref{lognes} that
$$\Big\| \boH_{c^*(t)} \big( J \boR_{c^*(t)} \eps^* \big)(\cdot, t) e^\frac{\nu_\gc |\cdot|}{2} \Big\|_{L^2(\R)} \leq A_\gc \big\| \eps^*(\cdot, t) \big\|_{L^2(\R)^2}.$$
Combining the previous inequality with \eqref{champs} and \eqref{noisiel}, we derive the estimate
\begin{equation}
\label{chelles}
\big| \boI_2^*(t) \big| \leq A_\gc \Big( R \big\| u^*(\cdot, t) \big\|_{X(B(0, R))} + \delta \big\| u^*(\cdot, t) \big\|_{X(\R)} \Big) \big\| \eps^*(\cdot, t) \big\|_{L^2(\R)^2},
\end{equation}

We finally recall that
$$S u^*(\cdot, t) = \boH_{c^*(t)}(\eps^*)(\cdot, t),$$
with $\langle \eps^*(\cdot, t), \partial_x Q_{c^*(t)} \rangle_{L^2(\R)^2}$ for any $t \in \R$ by \eqref{eq:orthobis}. In view of \eqref{eq:inv-Hc}, we infer that
\begin{equation}
\label{tremblay}
\big\| \eps^*(\cdot, t) \big\|_{X(\R)} \leq A_\gc \big\| S u^*(\cdot, t) \big\|_{L^2(\R)^2} \leq A_\gc \big\| u^*(\cdot, t) \big\|_{X(\R)},
\end{equation}
so that \eqref{chelles} may be written as
$$\big| \boI_2^*(t) \big| \leq A_\gc \Big( \frac{R^2}{\delta} \big\| u^*(\cdot, t) \big\|_{X(B(0, R))}^2 + 2 \delta \big\| u^*(\cdot, t) \big\|_{X(\R)}^2 \Big).$$
Fixing the number $\delta$ so that $2 A_\gc \delta \leq (2 - \gc^2)/64$, and letting $R_2$ denote the corresponding number $R$, we obtain \eqref{joinville}, with $A_2 = A_\gc R_2^2/\delta$.

Concerning the third term in the right-hand side of \eqref{paris}, we have

\begin{step}
\label{I3}
There exists a positive number $A_3$, depending only on $\gc$, such that
\begin{equation}
\label{bobigny1}
\big| \boI_3^*(t) \big| := \bigg| (c^*)' \int_\R \mu \big\langle \partial_c \boH_{c^*}(\eps^*), u^* \big\rangle_{\R^2} \bigg| \leq A_3 \alpha_0 \big\| u^*(\cdot, t) \big\|_{X(\R)}^2,
\end{equation}
for any $t \in \R$.
\end{step}

Coming back to \eqref{eq:modul0bis} and \eqref{eq:modul1bis}, we have
\begin{equation}
\label{aulnay}
\big| (c^*)'(t) \big| \leq A_\gc \alpha_0 \big\| \eps^*(\cdot, t) \big\|_{X(\R)}.
\end{equation}
On the other hand, we deduce from \eqref{eq:modul0bis}, \eqref{marne} and the explicit formula for $\boH_{c^*}$ in \eqref{def:boHc} that
$$\bigg| \int_\R \mu \big\langle \partial_c \boH_{c^*}(\eps^*), u^* \big\rangle_{\R^2} \bigg| \leq A_\gc \big\| u^*(\cdot, t) \big\|_{X(\R)}.$$
Combining with \eqref{tremblay}, we obtain \eqref{bobigny1}.

Similarly, we can combine \eqref{tremblay} and \eqref{aulnay} with the expression of $\boH_{c^*}$ in \eqref{def:boHc} and the exponential decay of the function $\partial_c Q_{c^*}$ and its derivatives, to establish

\begin{step}
\label{I4}
There exists a positive number $A_3$, depending only on $\gc$, such that
\begin{equation}
\label{bobigny2}
\big| \boI_3^*(t) \big| := \bigg| (c^*)' \int_\R \mu \big\langle \boH_{c^*}(\partial_c Q_{c^*}), u^* \big\rangle_{\R^2} \bigg| \leq A_3 \alpha_0 \big\| u^*(\cdot, t) \big\|_{X(\R)}^2,
\end{equation}
for any $t \in \R$.
\end{step}

Finally, we show

\begin{step}
\label{I5}
There exist two positive numbers $A_4$ and $R_4$, depending only on $\gc$, such that
\begin{equation}
\label{pantin}
\big| \boI_4^*(t) \big| := \bigg| \big( (a^*)' - c^* \big) \int_\R \mu \big\langle \boH_{c^*}(\partial_x \eps^*), u^* \big\rangle_{\R^2} \bigg| \leq \frac{2 - \gc^2}{64} \big\| u^*(\cdot, t) \big\|_{X(\R)}^2 + A_4 \big\| u^*(\cdot, t) \big\|_{X(B(0, R_4))}^2,
\end{equation}
for any $t \in \R$.
\end{step}

The proof is similar to the one of Step \ref{I2}. Given a small positive number $\delta$, we can use \eqref{champigny} to find a radius $R$ such that
$$\bigg| \int_\R \mu \big\langle \boH_{c^*}(\partial_x \eps^*), u^* \big\rangle_{\R^2} \bigg| \leq \big\| e^{\nu_\gc |\cdot|} \boH_{c^*(t)}(\partial_x \eps^*)(\cdot, t) \big\|_{L^2(\R)} \Big( R \big\| u^*(\cdot, t) \big\|_{X(B(0, R))} + \delta \big\| u^*(\cdot, t) \big\|_{X(\R)} \Big).$$
In view of \eqref{eq:modul1bis} and \eqref{torcy}, this gives
$$\big| \boI_4^*(t) \big| \leq A_\gc \big\| \eps^*(\cdot, t) \big\|_{X(\R)} \big\| e^{\nu_\gc |\cdot|} \eps^*(\cdot, t) \big\|_{H^4(\R)} \Big( R \big\| u^*(\cdot, t) \big\|_{X(B(0, R))} + \delta \big\| u^*(\cdot, t) \big\|_{X(\R)} \Big),$$
so that by \eqref{marne} and \eqref{tremblay},
$$\big| \boI_4^*(t) \big| \leq A_\gc \big\| u^*(\cdot, t) \big\|_{X(\R)} \Big( R \big\| u^*(\cdot, t) \big\|_{X(B(0, R))} + \delta \big\| u^*(\cdot, t) \big\|_{X(\R)} \Big).$$
Estimate \eqref{pantin} follows arguing as in the proof of \eqref{joinville}.

We are now in position to conclude the proof of Proposition \ref{prop:monou1}.

\begin{proof}[End of the proof of Proposition \ref{prop:monou1}]
Applying the estimates in Steps \ref{I1} to \ref{I5} to the identity \eqref{paris}, we have
$$\frac{d}{dt} \Big( \boI^*(t) \Big) \geq \Big( \frac{2 - \gc^2}{32} - A_3 \alpha_0 \Big) \big\| u^*(\cdot, t) \big\|_{X(\R)}^2 - \big( A_1 + A_ 2 + A_4 \big) \big\| u^*(\cdot, t) \big\|_{X(B(0, R_*))}^2,$$
with $R_* = \max \{ R_1, R_2, R_3 \}$. Choosing $\alpha_0$ small enough, we are led to \eqref{eq:petitplateau} with $A_* = A_1 + A_ 2 + A_4$.
\end{proof}

\subsection{Proof of Lemma \ref{lem:loc-coer}}

Identity \eqref{eq:loc-virial} derives from a somewhat tedious, but direct computation. For sake of completeness, we provide the following details.

When $u \in X^3(\R)$, the function $J S u = - 2 \partial_x u$ lies in the domain $H^2(\R) \times L^2(\R)$ of $\boH_c$. In view of \eqref{def:Mc}, the quantity in the right-hand side of \eqref{eq:loc-virial} is well-defined. Moreover, we can invoke \eqref{def:boHc} to write it as
\begin{equation}
\label{esquilin}
\begin{split}
2 \big\langle S M_c u, \boH_c(J S u) \big\rangle_{L^2(\R)^2} = & \int_\R \bigg( \frac{\partial_x \eta_c}{\eta_c} \Big( 2 - \frac{\partial_{xx} \eta_c}{(1 - \eta_c)^2} - \frac{(\partial_x \eta_c)^2}{(1 - \eta_c)^3} \Big) - c^2 \frac{\partial_x \eta_c}{(1 - \eta_c)^3} \bigg) u_1 \partial_x u_1\\
& - \int_\R \frac{\partial_x \eta_c}{\eta_c} \partial_x \Big( \frac{\partial_{xx} u_1}{1 - \eta_c} \Big) + 4 \int_\R \frac{\partial_x \eta_c (1 - \eta_c)}{\eta_c} u_2 \partial_x u_2\\
& + 2 c \int_\R \bigg( \frac{\partial_x \eta_c}{1 - \eta_c} u_1 \partial_x u_2 - \frac{\partial_x \eta_c}{\eta_c (1 - \eta_c)} \partial_x \big( u_1 u_2 \big) \bigg).
\end{split}
\end{equation}
In order to simplify the integrations by parts of the integrals in the right-hand side of \eqref{esquilin} which lead to \eqref{eq:loc-virial}, we recall that $\eta_c$ solves the equation
\begin{equation}
\label{eq:etac}
\partial_{xx} \eta_c = (2 - c^2) \eta_c - 3 \eta_c^2,
\end{equation}
so that we have
\begin{equation}
\label{aventin}
(\partial_x \eta_c)^2 = (2 - c^2) \eta_c^2 - 2\eta_c^3, \quad {\rm and} \quad \partial_x \Big( \frac{\partial_x \eta_c}{\eta_c} \Big) = - \eta_c.
\end{equation}
As a consequence, the third integral in the right-hand side of \eqref{esquilin} can be expressed as
\begin{equation}
\label{palatin}
4 \int_\R \frac{\partial_x \eta_c (1 - \eta_c)}{\eta_c} u_2 \partial_x u_2 = 2 \int_\R \mu_c u_2^2,
\end{equation}
with $\mu_c := \eta_c + \partial_{xx} \eta_c$. The last integral is similarly given by
\begin{equation}
\label{capitole}
\int_\R \bigg( \frac{\partial_x \eta_c}{1 - \eta_c} u_1 \partial_x u_2 - \frac{\partial_x \eta_c}{\eta_c (1 - \eta_c)} \partial_x \big( u_1 u_2 \big) \bigg) = - \int_\R \Big( \eta_c u_1 u_2 + \frac{\partial_x \eta_c}{1 - \eta_c} u_2 \partial_x u_1 \Big).
\end{equation}
Introducing \eqref{palatin} and \eqref{capitole} into \eqref{esquilin}, we obtain the identity
$$2 \big\langle S M_c u, \boH_c(J S u) \big\rangle_{L^2(\R)^2} = I + 2 \int_\R \mu_c \Big( u_2 - \frac{c \eta_c}{2 \mu_c} u_1 - \frac{c \partial_x \eta_c}{2 \mu_c (1 - \eta_c)} \partial_x u_1 \Big)^2,$$
where
\begin{align*}
I = & \int_\R \bigg( \frac{\partial_x \eta_c}{\eta_c} \Big( 2 - \frac{\partial_{xx} \eta_c}{(1 - \eta_c)^2} - \frac{(\partial_x \eta_c)^2}{(1 - \eta_c)^3} \Big) - c^2 \frac{\partial_x \eta_c}{(1 - \eta_c)^3} - c^2 \frac{\eta_c \partial_x \eta_c}{\mu_c (1 - \eta_c)} \bigg) u_1 \partial_x u_1\\
& - \int_\R \frac{\partial_x \eta_c}{\eta_c} \partial_x \Big( \frac{\partial_{xx} u_1}{1 - \eta_c} \Big) - \frac{c^2}{2} \int_\R \frac{\eta_c^2}{\mu_c} u_1^2 - \frac{c^2}{2} \int_\R \frac{(\partial_x \eta_c)^2}{\mu_c (1 - \eta_c)^2} (\partial_x u_1)^2.
\end{align*}
Relying again on \eqref{eq:etac} and \eqref{aventin}, we finally check that
$$I = \frac{3}{2} \int_\R \frac{\eta_c^2}{\mu_c} \Big( \partial_x u_1 - \frac{\partial_x \eta_c}{\eta_c} u_1 \Big)^2,$$
which is enough to complete the proof of identity \eqref{eq:loc-virial}. \qed

\subsection{Proof of Proposition \ref{prop:coer-Gc}}

In view of \eqref{form:etavc} and \eqref{eq:loc-virial}, the quadratic form $G_c$ is well-defined and continuous on $X^1(\R)$. Moreover, setting $v = (\sqrt{\eta_c} u_1, \sqrt{\eta_c} u_2)$ and using \eqref{aventin}, we can write it as
\begin{equation}
\label{jerusalem1}
G_c(u) = \frac{3}{2} \int_\R \frac{\eta_c}{\mu_c} \Big( \partial_x v_1 - \frac{3 \partial_x \eta_c}{2 \eta_c} v_1 \Big)^2 + 2 \int_\R \frac{\mu_c}{\eta_c} \Big( v_2 - \frac{c^3 \eta_c}{4 \mu_c (1 - \eta_c)} v_1 - \frac{c \partial_x \eta_c}{2 \mu_c (1 - \eta_c)} \partial_x v_1 \Big)^2,
\end{equation}
where we have set, as above, $\mu_c := \eta_c + \partial_{xx} \eta_c$. Introducing the pair
\begin{equation}
\label{hebron}
w = \Big( v_1, v_2 - \frac{c \partial_x \eta_c}{2 \mu_c (1 - \eta_c)} \partial_x v_1 \Big) = \Big( \sqrt{\eta_c} u_1, \sqrt{\eta_c} \Big( u_2 - \frac{c (\partial_x \eta_c)^2}{4 \mu_c \eta_c (1 - \eta_c)} u_1 - \frac{c \partial_x \eta_c}{2 \mu_c (1 - \eta_c)} \partial_x u_1 \Big) \Big),
\end{equation}
we obtain
\begin{equation}
\label{jerusalem2}
G_c(u) = \big\langle \boT_c(w), w \big\rangle_{L^2(\R)^2},
\end{equation}
with
\begin{equation}
\label{ramallah}
\boT_c(w) = \begin{pmatrix} - \partial_x \Big( \frac{3 \eta_c}{2 \mu_c} \partial_x w_1 \Big) + \Big( \frac{27 (\partial_x \eta_c)^2}{8 \mu_c \eta_c} + \frac{c^6 \eta_c}{8 \mu_c (1 - \eta_c)^2} + \partial_x \Big( \frac{9 \partial_x \eta_c}{4 \mu_c} \Big) \Big) w_1 - \frac{c^3}{2(1 - \eta_c)} w_2 \\ - \frac{c^3}{2(1 - \eta_c)} w_1 + \frac{2 \mu_c}{\eta_c} w_2 \end{pmatrix}.
\end{equation}
The operator $\boT_c$ in \eqref{ramallah} is self-adjoint on $L^2(\R)^2$, with domain $\Dom(\boT_c) = H^2(\R) \times L^2(\R)$. Moreover, it follows from \eqref{jerusalem1} and \eqref{jerusalem2} that $\boT_c$ is non-negative, with a kernel equal to
$$\Ker(\boT_c) = \Span \Big( \eta_c^\frac{3}{2}, \frac{c^3 \eta_c^\frac{5}{2}}{4 \mu_c (1 - \eta_c)} \Big).$$
In order to establish \eqref{eq:coer-Gc}, we now prove

\setcounter{step}{0}
\begin{step}
\label{G1}
Let $c \in (- \sqrt{2}, \sqrt{2}) \setminus \{ 0 \}$. There exists a positive number $\Lambda_1$, depending continuously on $c$, such that
\begin{equation}
\label{nazareth}
\langle \boT_c(w), w \rangle_{L^2(\R)^2} \geq \Lambda_1 \int_\R \Big( w_1^2 + w_2^2 \Big),
\end{equation}
for any pair $w \in X^1(\R)$ such that
\begin{equation}
\label{cana}
\Big\langle w, \Big( \eta_c^\frac{3}{2}, \frac{c^3 \eta_c^\frac{5}{2}}{4 \mu_c (1 - \eta_c)} \Big) \Big\rangle_{L^2(\R)^2} = 0.
\end{equation}
\end{step}

In order to prove Step \ref{G1}, we show that the essential spectrum of $\boT_c$ is given by
\begin{equation}
\label{bethleem}
\sigma_{\rm ess}(\boT_c) = \big[ \tau_c, + \infty \big),
\end{equation}
with
\begin{equation}
\label{naplouse}
\tau_c = \frac{(3 - c^2)(22 + c^2)}{16} - \frac{1}{2} \Big( \frac{(3 - c^2)^2 (22 + c^2)^2}{64} - 27 (2 - c^2) \Big)^\frac{1}{2} > 0.
\end{equation}
In this case, $0$ is an isolated eigenvalue in the spectrum of $\boT_c$. Inequality \eqref{nazareth} follows with $\Lambda_1$ either equal to $\tau_c$, or to the smallest positive eigenvalue of $\boT_c$. In each case, $\Lambda_1$ depends continuously on $c$ due to the analytic dependence on $c$ of the operator $\boT_c$.

The proof of \eqref{bethleem} relies as usual on the Weyl criterion. We deduce from \eqref{eq:etac} and \eqref{aventin} that
$$\frac{\mu_c(x)}{\eta_c(x)} \to 3 - c^2, \quad {\rm and} \quad \frac{\partial_x \eta_c(x)}{\eta_c(x)} \to \pm \sqrt{2 - c^2},$$
as $x \to \pm \infty$. Coming back to \eqref{ramallah}, we introduce the operator $\boT_\infty$ given by
$$\boT_\infty(w) = \begin{pmatrix} - \frac{3}{2 (3 - c^2)} \partial_{xx} w_1 + \frac{(3 - c^2)(6 + c^2)}{8} w_1 - \frac{c^3}{2} w_2 \\ - \frac{c^3}{2} w_1 + 2 (3 - c^2) w_2 \end{pmatrix}.$$
By the Weyl criterion, the essential spectrum of $\boT_c$ is equal to the spectrum of $\boT_\infty$.

We next apply again the Weyl criterion to establish that a real number $\lambda$ belongs to the spectrum of $\boT_\infty$ if and only if there exists a complex number $\xi$ such that
$$\lambda^2 - \Big( \frac{3}{2 (3 - c^2)} |\xi|^2 + \frac{(3 - c^2)(22 + c^2)}{8} \Big) \lambda + 3 |\xi|^2 + \frac{27}{4} \Big( 2 - c^2 \Big) = 0.$$
This is the case if and only if
$$\lambda = \frac{3 |\xi|^2}{4 (3 - c^2)} + \frac{(3 - c^2)(22 + c^2)}{16} \pm \frac{1}{4} \Big( \frac{9 |\xi|^4}{3 - c^2} + \frac{3 (c^2 - 10)}{2} |\xi|^2 + \frac{225}{4} - \frac{195}{4} c^2 + \frac{229}{16} c^4 + \frac{19}{8} c^6 + \frac{c^8}{16} \Big)^\frac{1}{2}.$$
Notice that the quantity in the square root of this expression is positive since its discriminant with respect to $|\xi|^2$ is $- 9 c^2/(3 - c^2)^2$. As a consequence, we obtain that
$$\sigma_{\rm ess}(\boT_c) = \sigma(\boT_\infty) = \big[ \tau_c, + \infty \big),$$
with $\tau_c$ as in \eqref{naplouse}. This completes the proof of Step \ref{G1}.

\begin{step}
\label{G2}
There exists a positive number $\Lambda_2$, depending continuously on $c$, such that
$$G_c(u) \geq \Lambda_2 \int_\R \eta_c \big( (\partial_x u_1)^2 + u_1^2 + u_2^2 \big),$$
for any pair $u \in X^1(\R)$ such that
\begin{equation}
\label{sidon}
\langle u, Q_c \rangle_{L^2(\R)^2} = 0.
\end{equation}
\end{step}

We start by improving the estimate in \eqref{nazareth}. Given a pair $w \in X^1(\R)$, we check that
$$\bigg| \langle \boT_c(w), w \rangle_{L^2(\R)^2} - \frac{3 \tau}{2} \int_\R \frac{\eta_c}{\mu_c} (\partial_x w_1)^2 \bigg| \leq A_c \int_\R (w_1^2 + w_2^2).$$
Here and in the sequel, $A_c$ refers to a positive number, depending continuously on $c$. For $0 < \tau < 1$, we deduce that
$$\langle \boT_c(w), w \rangle_{L^2(\R)^2} \geq \big( 1 - \tau \big) \langle \boT_c(w), w \rangle_{L^2(\R)^2} + \frac{3 \tau}{2} \int_\R \frac{\eta_c}{\mu_c} (\partial_x w_1)^2 - A_c \tau \int_\R (w_1^2 + w_2^2).$$
Since $\eta_c/\mu_c \geq 1/(3 - c^2)$, we are led to
$$\langle \boT_c(w), w \rangle_{L^2(\R)^2} \geq \Big( \big( 1 - \tau \big) \Lambda_1 - A_c \tau \Big) \int_\R (w_1^2 + w_2^2) + \frac{3 \tau}{2 (3 - c^2)} \int_\R (\partial_x w_1)^2,$$
under condition \eqref{cana}. For $\tau$ small enough, this provides the lower bound
\begin{equation}
\label{gaza1}
\langle \boT_c(w), w \rangle_{L^2(\R)^2} \geq A_c \int_\R \big( (\partial_x w_1)^2 + w_1^2 + w_2^2 \big),
\end{equation}
when $w$ satisfies condition \eqref{cana}.

When the pair $w$ depends on the pair $u$ as in \eqref{hebron}, we can express \eqref{gaza1} in terms of $u$. The left-hand side of \eqref{gaza1} is exactly equal to $G_c(u)$ by \eqref{jerusalem2}, whereas for the left-hand side, we have
$$\int_\R \big( (\partial_x w_1)^2 + w_1^2 \big) = \int_\R \eta_c \bigg( (\partial_x u_1)^2 + \Big( 1 + \frac{(\partial_x \eta_c)^2}{4 \eta_c} - \frac{\partial_{xx} \eta_c}{2 \eta_c} \Big) u_1^2 \bigg).$$
Since
$$1 + \frac{(\partial_x \eta_c)^2}{4 \eta_c} - \frac{\partial_{xx} \eta_c}{2 \eta_c} = \frac{2 + c^2}{4} + \eta_c \geq \frac{1}{2},$$
by \eqref{eq:etac} and \eqref{aventin}, we deduce that \eqref{gaza1} may be written as
$$G_c(u) \geq A_c \int_\R \eta_c \Big( (\partial_x u_1)^2 + \frac{1}{2} u_1^2 \Big) + A_c \int_\R \eta_c \Big( u_2 - \frac{c \partial_x \eta_c}{2 \mu_c (1 - \eta_c)} \partial_x u_1 - \frac{c (\partial_x \eta_c)^2}{4 _mu_c \eta_c (1 - \eta_c)} u_1 \Big)^2.$$
At this stage, recall that, given two vectors $a$ and $b$ in an Hilbert space $H$, we have
$$\big\| a - b \big\|_H^2 \geq \tau \big\| a \big\|_H^2 - \frac{\tau}{1 - \tau} \big\| b \big\|_H^2,$$
for any $0 < \tau < 1$. In particular, this gives
$$G_c(u) \geq A_c \int_\R \eta_c \Big( (\partial_x u_1)^2 + \frac{1}{2} u_1^2 + \tau u_2^2 \Big) - \frac{\tau A_\gc}{1 - \tau} \int_\R \eta_c \Big( \frac{c \partial_x \eta_c}{2 \mu_c (1 - \eta_c)} \partial_x u_1 - \frac{c (\partial_x \eta_c)^2}{4 _mu_c \eta_c (1 - \eta_c)} u_1 \Big)^2.$$
It then remains to choose $\tau$ small enough so that we can deduce from \eqref{form:etavc} that
\begin{equation}
\label{gaza2}
G_c(u) \geq A_c \int_\R \eta_c \big( (\partial_x u_1)^2 + u_1^2 + u_2^2 \big),
\end{equation}
when $w$ satisfies condition \eqref{cana}, i.e. when $u$ is orthogonal to the pair
\begin{equation}
\label{haifa}
\gu_c = \Big( \eta_c^2 - \frac{c^4 \eta_c^2 (\partial_x \eta_c)^2}{16 \mu_c^2 (1 - \eta_c)^2} + \partial_x \Big( \frac{c^4 \eta_c^3 (\partial_x \eta_c)}{8 \mu_c^2 (1 - \eta_c)^2} \Big), \frac{c^3 \eta_c^3}{4 \mu_c (1 - \eta_c)} \Big).
\end{equation}

The last point to verify is that \eqref{gaza2} remains true, decreasing possibly the value of $A_c$, when we replace this orthogonality condition by condition \eqref{sidon}. With this goal in mind, we remark that
$$\langle \gu_c, Q_c \rangle_{L^2(\R)^2} \neq 0.$$ Otherwise, we would deduce from \eqref{gaza2} that
$$0 = G_c(Q_c) \geq A_c \int_\R \eta_c \big( (\partial_x \eta_c)^2 + \eta_c^2 + v_c^2 \big) > 0,$$
which is impossible. Moreover, the quantity $\langle \gu_c, Q_c \rangle_{L^2(\R)^2}$ depends continuously on $c$ in view of \eqref{haifa}. We next consider a pair $u$ which satisfies \eqref{sidon}, and we denote $\lambda$ the real number such that $\gu = \lambda Q_c + u$ is orthogonal to $\gu_c$. Since $Q_c$ belongs to the kernel of $G_c$, we deduce from \eqref{gaza2} that
\begin{equation}
\label{tyr}
G_c(u) = G_c(\gu) \geq A_c \int_\R \eta_c \big( (\partial_x \gu_1)^2 + \gu_1^2 + \gu_2^2 \big).
\end{equation}
On the other hand, since $u$ satisfies \eqref{sidon}, we have
$$\lambda = \frac{\langle \gu, Q_c \rangle_{L^2(\R)^2}}{\| Q_c \|_{L^2(\R)^2}^2},$$
Using the Cauchy-Schwarz inequality, this leads to
$$\lambda^2 \leq A_c \bigg( \int_\R \Big( \eta_c + \frac{v_c^2}{\eta_c} \Big) \bigg) \bigg( \int_\R \eta_c \big( \gu_1^2 + \gu_2^2 \big) \bigg),$$
so that, by \eqref{form:etavc} and \eqref{tyr},
$$\lambda^2 \leq A_c G_c(\gu) = A_c G_c(u).$$
Combining again with \eqref{tyr}, we are led to
\begin{align*}
\int_\R \eta_c \big( (\partial_x u_1)^2 + u_1^2 + u_2^2 \big) \leq & 2 \bigg( \lambda^2 \int_\R \eta_c \big( (\partial_x \eta_c)^2 + \eta_c^2 + v_c^2 \big) + \int_\R \eta_c \big( (\partial_x \gu_1)^2 + \gu_1^2 + \gu_2^2 \big) \bigg)\\
\leq & A_c G_c(u),
\end{align*}
which completes the proof of Step \ref{G2}.

\begin{step}
\label{G3}
End of the proof.
\end{step}

We conclude the proof applying again the last argument in the proof of Step \ref{G2}. We decompose a pair $u \in X(\R)$, which satisfies the orthogonality condition in \eqref{eq:ortho-u}, as $u = \lambda Q_c + \gu$, with $\langle \gu, Q_c \rangle_{L^2(\R)^2} = 0$. Since $Q_c$ belongs to the kernel of $G_c$, we deduce from Step \ref{G2} that
\begin{equation}
\label{samarie}
G_c(u) = G_c(\gu) \geq \Lambda_2 \int_\R \eta_c \big( (\partial_x \gu_1)^2 + \gu_1^2 + \gu_2^2 \big).
\end{equation}
Relying on the orthogonality condition in \eqref{eq:ortho-u}, we next compute
$$\lambda = - \frac{\langle \gu, S \partial_c Q_c \rangle_{L^2(\R)^2}}{\langle Q_c, S \partial_c Q_c \rangle_{L^2(\R)^2}}.$$
Using the Cauchy-Schwarz inequality and invoking \eqref{eq:angle}, we obtain
$$\lambda^2 \leq \frac{1}{4 (2 - c^2)} \bigg( \int_\R \frac{1}{\eta_c} \Big( (\partial_c \eta_c)^2 + (\partial_c v_c)^2 \Big) \bigg) \bigg( \int_\R \eta_c \big( \gu_1^2 + \gu_2^2 \big) \bigg).$$
In view of \eqref{form:etavc}, we can check that
$$\partial_c \eta_c(x) = \frac{c}{\sqrt{2 - c^2}} \eta_c(x) \bigg( \frac{x \, \sh \big( \frac{\sqrt{2 - c^2} x}{2} \big)}{\ch \big( \frac{\sqrt{2 - c^2} x}{2} \big)} - \frac{2}{\sqrt{2 - c^2}} \bigg),$$
while
$$\partial_c v_c = \frac{\eta_c}{2 (1 - \eta_c)} + \frac{c \partial_c \eta_c}{2( 1 - \eta_c)^2},$$
so that
$$\int_\R \frac{1}{\eta_c} \Big( (\partial_c \eta_c)^2 + (\partial_c v_c)^2 \Big) \leq A_c.$$
As a consequence, we can derive from \eqref{samarie} that
$$\lambda^2 \leq A_c G_c(\gu) = A_c G_c(u).$$
Combining again with \eqref{samarie}, we are led to
\begin{align*}
\int_\R \eta_c \big( (\partial_x u_1)^2 + u_1^2 + u_2^2 \big) \leq & 2 \bigg( \lambda^2 \int_\R \eta_c \big( (\partial_x \eta_c)^2 + \eta_c^2 + v_c^2 \big) + \int_\R \eta_c \big( (\partial_x \gu_1)^2 + \gu_1^2 + \gu_2^2 \big) \bigg)\\
\leq & A_c G_c(u).
\end{align*}
It remains to recall that
$$\eta_c(x) \geq A_c e^{- \sqrt{2} |x|},$$
by \eqref{form:etavc}, to obtain \eqref{eq:coer-Gc}. This completes the proof of Proposition \ref{prop:coer-Gc}. \qed

\subsection{Proof of Proposition \ref{prop:monou2}}

Combining inequality \eqref{marne} with the definitions for $u^*$ in \eqref{def:u*}, and for $\boH_c$ in \eqref{def:boHc}, we know that there exists a positive number $A_{k, \gc}$ such that
\begin{equation}
\label{losangeles}
\int_\R \Big( \big( \partial_x^k u_1^*(x, t) \big)^2 + \big( \partial_x^k u_2^*(x, t) \big)^2 \Big) e^{2 \nu_\gc |x|} \, dx \leq A_{k, \gc},
\end{equation}
for any $k \in \N$ and any $t \in \R$. In view of \eqref{eq:pouru*} and \eqref{def:Mc}, this is enough to guarantee the differentiability with respect to time of the quantity
$$\boJ^*(t) := \big\langle M_{c^*(t)} u^*(\cdot, t), u^*(\cdot, t) \big\rangle_{L^2(\R)^2},$$
and to check that
\begin{equation}
\label{washington}
\begin{split}
\frac{d}{dt} \big( \boJ^* \big) = & 2 \big\langle S M_{c^*} u^*, \boH_{c^*}(J S u^*) \big\rangle_{L^2(\R)^2} + 2 \big\langle S M_{c^*} u^*, \boH_{c^*}(J \boR_{c^*} \eps^*) \big\rangle_{L^2(\R)^2}\\
& + 2 \big( (a^*)' - c^* \big) \big\langle S M_{c^*} u^*, \boH_{c^*}(\partial_x \eps^*) \big\rangle_{L^2(\R)^2} - 2 \big( c^* \big)' \big\langle S M_{c^*} u^*, \boH_{c^*}(\partial_c Q_{c^*}) \big\rangle_{L^2(\R)^2}\\ & + \big( c^* \big)' \big\langle \partial_c M_{c^*} u^*, u^* \big\rangle_{L^2(\R)^2} + 2 \big( c^* \big)' \big\langle M_{c^*} u^*, S \partial_c \boH_{c^*}(\eps^*) \big\rangle_{L^2(\R)^2}.
\end{split}
\end{equation}
In particular, the proof of \eqref{eq:moyenplateau} reduces to estimate the six terms in the right-hand side of \eqref{washington}. Concerning the first one, we derive from Proposition \ref{prop:coer-Gc} the following estimate. 

\setcounter{step}{0}
\begin{step}
\label{L1}
There exists a positive number $B_1$, depending only on $\gc$, such that
\begin{equation}
\label{newyork}
\boJ_1^*(t) := 2 \big\langle S M_{c^*} u^*, \boH_{c^*}(J S u^*) \big\rangle_{L^2(\R)^2}\\
\geq B_1 \int_\R \big[ (\partial_x u_1^*)^2 + (u_1^*)^2 + (u_2^*)^2 \big](x, t) e^{- \sqrt{2} |x|} \, dx,
\end{equation}
for any $t \in \R$.
\end{step}

Since the pair $u^*$ satisfies the orthogonality condition in \eqref{eq:ortho-u} by \eqref{eq:ortho-u*}, inequality \eqref{newyork} is exactly \eqref{eq:coer-Gc}, setting $B_1 = \Lambda_c$.

For the second term, we can prove

\begin{step}
\label{L2}
There exists a positive number $B_2$, depending only on $\gc$, such that
\begin{equation}
\label{boston}
\big| \boJ_2^*(t) \big| := 2 \Big| \big\langle S M_{c^*} u^*, \boH_{c^*}(J \boR_{c^*} \eps^*) \big\rangle_{L^2(\R)^2} \Big| \leq B_2 \big\| \eps^*(\cdot, t) \big\|_{X(\R)}^\frac{1}{2} \big\| u^*(\cdot, t) \big\|_{X(\R)}^2,
\end{equation}
for any $t \in \R$.
\end{step}

In view of \eqref{form:etavc}, \eqref{eq:modul0bis} and \eqref{def:Mc}, we first notice that there exists a positive number $A_\gc$, depending only on $\gc$, such that
\begin{equation}
\label{cleveland}
\big\| M_{c^*(t)} \big\|_{L^\infty(\R)} \leq A_\gc,
\end{equation}
for any $t \in \R$. As a consequence, we can write
\begin{equation}
\label{charlotte}
\big| \boJ_2^*(t) \big| \leq A_\gc \big\| u^*(\cdot, t) \big\|_{L^2(\R)^2} \big\| \boH_{c^*}(J \boR_{c^*} \eps^*)(\cdot, t) \big\|_{L^2(\R)^2}.
\end{equation}
Applying \eqref{detroit1} to the last inequality in \eqref{detroit2}, we next check that
\begin{align*}
\big\| \boH_{c^*(t)} \big( & J \boR_{c^*(t)} \eps^* \big)(\cdot, t) \big\|_{L^2(\R)^2}\\
\leq & A_\gc \big\| \eps^*(\cdot, t) \big\|_{L^2(\R)^2} \Big(\big\| \eps_\eta^*(\cdot, t) \big\|_{L^2(\R)}^\frac{5}{8} \big\| \eps_\eta^*(\cdot, t) \big\|_{H^{15}(\R)}^\frac{5}{8} + \big\| \eps^*(\cdot, t) \big\|_{L^2(\R)^2}^\frac{9}{16} \big\| \eps^*(\cdot, t) \big\|_{H^{63}(\R)^2}^\frac{3}{16} \Big),
\end{align*}
so that by \eqref{eq:modul0bis}, \eqref{marne} and \eqref{tremblay}, we have
$$\big\| \boH_{c^*(t)} \big( J \boR_{c^*(t)} \eps^* \big)(\cdot, t) \big\|_{L^2(\R)^2} \leq A_\gc 
\big\| \eps^*(\cdot, t) \big\|_{L^2(\R)^2}^\frac{1}{2} \big\| u^*(\cdot, t) \big\|_{X(\R)}.$$
Estimate \eqref{boston} follows combining with \eqref{charlotte}.

We now turn to the third term in the right-hand side of \eqref{washington}.

\begin{step}
\label{L3}
There exists a positive number $B_3$, depending only on $\gc$, such that
\begin{equation}
\label{chicago}
\big| \boJ_3^*(t) \big| := 2 \big| (a^*)' - c^* \big| \, \Big| \big\langle S M_{c^*} u^*, \boH_{c^*}(\partial_x \eps^*) \big\rangle_{L^2(\R)^2} \Big| \leq B_3 \big\| \eps^*(\cdot, t) \big\|_{X(\R)}^\frac{1}{2} \big\| u^*(\cdot, t) \big\|_{X(\R)}^2,
\end{equation}
for any $t \in \R$.
\end{step}

In view of \eqref{eq:modul1bis} and \eqref{cleveland}, we have
\begin{equation}
\label{philadelphie}
\big| \boJ_3^*(t) \big| \leq A_\gc \big\| \eps^*(\cdot, t) \big\|_{X(\R)} \big\| u^*(\cdot, t) \big\|_{L^2(\R)^2} \big\| \boH_{c^*(t)}(\partial_x \eps^*)(\cdot, t) \big\|_{L^2(\R)^2},
\end{equation}
for any $t \in \R$. Coming back to the definition for $\boH_c$ in \eqref{def:boHc}, we can write
$$\big\| \boH_{c^*(t)}(\partial_x \eps^*)(\cdot, t) \big\|_{L^2(\R)^2} \leq A_\gc \Big( \big\| \eps_\eta^*(\cdot, t) \big\|_{H^3(\R)} + \big\| \eps_v^*(\cdot, t) \big\|_{H^1(\R)} \Big).$$
Hence, by \eqref{detroit1} again,
$$\big\| \boH_{c^*(t)}(\partial_x \eps^*)(\cdot, t) \big\|_{L^2(\R)^2} \leq A_\gc \big\| \eps_\eta^* \big\|_{L^2(\R)^2}^\frac{1}{2} \Big( \big\| \eps_\eta^*(\cdot, t) \big\|_{H^7(\R)}^\frac{1}{2} + \big\| \eps_v^*(\cdot, t) \big\|_{H^3(\R)}^\frac{1}{2} \Big).$$
Combining the latter inequality with \eqref{marne}, \eqref{tremblay} and \eqref{philadelphie} yields estimate \eqref{chicago}.

For the fourth term, we have

\begin{step}
\label{L4}
There exists a positive number $B_4$, depending only on $\gc$, such that
\begin{equation}
\label{atlanta}
\big| \boJ_4^*(t) \big| := 2 \big| (c^*)' \big| \Big| \big\langle S M_{c^*} u^*, \boH_{c^*}(\partial_c Q_{c^*}) \big\rangle_{L^2(\R)^2} \Big| \leq B_4 \big\| \eps^*(\cdot, t) \big\|_{X(\R)}^\frac{1}{2} \big\| u^*(\cdot, t) \big\|_{X(\R)}^2,
\end{equation}
for any $t \in \R$.
\end{step}

Similarly, we deduce from \eqref{eq:modul1bis} and \eqref{cleveland} that
$$\big| \boJ_4^*(t) \big| \leq A_\gc \big\| \eps^*(\cdot, t) \big\|_{X(\R)}^2 \big\| u^*(\cdot, t) \big\|_{L^2(\R)^2}.$$
Estimate \eqref{atlanta} then appears as a consequence of \eqref{eq:modul0bis} and \eqref{tremblay}.

The fifth term is estimated in a similar way.

\begin{step}
\label{L5}
There exists a positive number $B_5$, depending only on $\gc$, such that
\begin{equation}
\label{greenville}
\big| \boJ_5^*(t) \big| := \big| (c^*)' \big| \Big| \big\langle \partial_c M_{c^*} u^*, u^* \big\rangle_{L^2(\R)^2} \Big| \leq B_5 \big\| \eps^*(\cdot, t) \big\|_{X(\R)}^\frac{1}{2} \big\| u^*(\cdot, t) \big\|_{X(\R)}^2,
\end{equation}
for any $t \in \R$.
\end{step}

We derive again from \eqref{form:etavc} and \eqref{def:Mc} the existence of a positive number $A_\gc$, depending only on $\gc$, such that
$$\big\| \partial_c M_{c^*(t)} \big\|_{L^\infty(\R)} \leq A_\gc,$$
for any $t \in \R$. As a consequence of \eqref{eq:modul1bis}, we infer that
$$\big| \boJ_5^*(t) \big| \leq A_\gc \big\| \eps^*(\cdot, t) \big\|_{X(\R)}^2 \big\| u^*(\cdot, t) \big\|_{L^2(\R)^2}^2.$$
This provides \eqref{greenville}, relying again on \eqref{eq:modul0bis}.

Finally, we infer from \eqref{eq:modul0bis}, \eqref{eq:modul1bis}, \eqref{tremblay}, \eqref{cleveland}, and the explicit formula for $\boH_{c^*}$ in \eqref{def:boHc} that

\begin{step}
\label{L6}
There exists a positive number $B_6$, depending only on $\gc$, such that
$$\big| \boJ_6^*(t) \big| := \big| (c^*)' \big| \Big| \big\langle M_{c^*} u^*, S \partial_c \boH_{c^*}(\eps^*) \big\rangle_{L^2(\R)^2} \Big| \leq B_6 \big\| \eps^*(\cdot, t) \big\|_{X(\R)}^\frac{1}{2} \big\| u^*(\cdot, t) \big\|_{X(\R)}^2,$$
for any $t \in \R$.
\end{step}

In order to conclude the proof of Proposition \ref{prop:monou2}, it remains to combine the five previous steps to obtain \eqref{eq:moyenplateau}, with $B_* := \max \big\{ 1/B_1, B_2 + B_3 + B_4 + B_5 + B_6 \}$. \qed

\subsection{Proof of Corollary \ref{cor:bomono}}

Corollary \ref{cor:bomono} is a consequence of Propositions \ref{prop:monou1} and \ref{prop:monou2}. As a matter of fact, combining the two estimates \eqref{eq:petitplateau} and \eqref{eq:moyenplateau} with the definition of $N(t)$, we obtain
$$\frac{d}{dt} \Big( \langle N(t) u^*(\cdot, t), u^*(\cdot, t) \rangle_{L^2(\R)^2} \Big) \geq \Big( \frac{2 - \gc^2}{64} - A_* B_* e^{\sqrt{2} R_*} \big\| \eps^*(\cdot, t) \big\|_{X(\R)}^\frac{1}{2} \Big) \big\| u^*(\cdot, t) \big\|_{X(\R)}^2,$$
for any $t \in \R$. Invoking \eqref{eq:modul0bis}, it remains to fix the parameter $\beta_\gc$ such that
$$\big\| \eps^*(\cdot, t) \big\|_{X(\R)}^\frac{1}{2} \leq \frac{2 - \gc^2}{128 A_* B_* e^{\sqrt{2} R_*}},$$
for any $t \in \R$, in order to obtain \eqref{eq:grandplateau}. Since the map $t \mapsto \langle N(t) u^*(\cdot, t), u^*(\cdot, t) \rangle_{L^2(\R)^2}$ is uniformly bounded by \eqref{losangeles} and \eqref{cleveland}, estimate \eqref{eq:petitpignon} follows by integrating \eqref{eq:grandplateau} from $t = - \infty$ to $t = + \infty$. Finally, statement \eqref{eq:derailleur} is a direct consequence of \eqref{eq:petitpignon}. \qed

\appendix
\section{On the regularity and smoothness of the Gross-Pitaevskii flow}
\label{sec:weak-conv}

\subsection{Continuity with respect to weak convergence in the energy space}
\label{sub:weak-conv}

It is shown in \cite{Zhidkov1} (see also \cite{Gallo3, Gerard2, BeGrSaS2}) that the Gross-Pitaevskii equation is globally well-posed in the spaces
$$\boX^k(\R) := \big\{ \psi \in L^\infty(\R), \ {\rm s.t.} \ \eta := 1 - |\Psi|^2 \in L^2(\R) \ {\rm and} \ \partial_x \psi \in H^k(\R) \big\},$$
equipped with the metric structure provided by the distance
$$d^k(\psi_1, \psi_2) := \big\| \psi_1 - \psi_2 \big\|_{L^\infty(\R)} + \big\| \partial_x \psi_1 - \partial_x \psi_2 \big\|_{H^k(\R)} + \big\| \eta_1 - \eta_2 \big\|_{L^2(\R)},$$
where we have set, as above, $\eta_1 := 1 - |\Psi_1|^2$ and $\eta_2 := 1 - |\Psi_2|^2$.

\begin{prop}[\cite{Zhidkov1}]
\label{prop:Cauchy-Psi}
Let $k \in \N$ and $\Psi_0 \in \boX^k(\R)$. There exists a unique solution $\Psi$ in $\boC^0(\R, \boX^k(\R))$ to \eqref{GP} with initial data $\Psi_0$. Moreover, the flow map $\Psi_0 \to \Psi(\cdot, T)$ is continuous on $\boX^k(\R)$ for any fixed $T \in \R$, and the map $t \to \Psi(\cdot, t)$ belongs to $\boC^1(\R, \boX^k(\R))$ when $\Psi_0$ belongs to $\boX^{k + 2}(\R)$. Finally, the Ginzburg-Landau energy is conserved along the flow, i.e.
\begin{equation}
\label{eq:cons-E}
\boE(\Psi(\cdot, t)) = \boE(\Psi_0),
\end{equation}
for any $t \in \R$.
\end{prop}

In order to establish the continuity of the Gross-Pitaevskii flow with respect to some suitable notion of weak convergence, it is helpful to enlarge slightly the range of function spaces in which it is possible to solve the Cauchy problem for \eqref{GP}. For $1/2 < s < 1$, we define the Zhidkov
spaces $\boZ^s(\R)$ as
$$\boZ^s(\R) := \big\{ \psi \in L^\infty(\R), \ {\rm s.t.} \ \partial_x \psi \in H^{s - 1}(\R) \big\},$$
and we endow them with the norm
$$\big\| \psi \big\|_{\boZ^s(\R)} := \big\| \psi \big\|_{L^\infty(\R)} + \big\| \partial_x \psi \big\|_{H^{s - 1}(\R)}.$$
We then prove

\begin{prop}
\label{prop:Cauchy-Zs}
Let $1/2 < s < 1$ and $\Psi_0 \in \boZ^s(\R)$. There exists a unique maximal solution $\Psi \in \boC^0((T_{\min}, T_{\max}), \boZ^s(\R))$ to \eqref{GP} with initial datum $\Psi^0$.
\end{prop}

\begin{proof}
Proposition \ref{prop:Cauchy-Zs} is essentially due to Gallo who has proved it in \cite{Gallo3} when $s \in \N^*$. Due to the Sobolev embedding theorem of $H^s(\R)$ into $L^\infty(\R)$ for $s > 1/2$, the proof in \cite{Gallo3} extends to the case $s > 1/2$. As a consequence, we refer to \cite{Gallo3} for a detailed proof.
\end{proof}

In the framework provided by Proposition \ref{prop:Cauchy-Psi}, we can introduce a notion of weak convergence for which the Gross-Pitaevskii flow is continuous. We consider a sequence of initial conditions $\Psi_{n, 0} \in \boX(\R)$ such that the energies $\boE(\Psi_{n, 0})$ are uniformly bounded with respect to $n$. Invoking the Rellich-Kondrachov theorem, there exists a function $\Psi_0 \in \boX(\R)$ such that, going possibly to a subsequence,
\begin{equation}
\label{hyp:w-cont1}
\partial_x \Psi_{n, 0} \rightharpoonup \partial_x \Psi_0 \quad {\rm in} \ L^2(\R), \qquad 1 - |\Psi_{n, 0}|^2 \rightharpoonup 1 - |\Psi_0|^2 \quad {\rm in} \ L^2(\R), 
\end{equation}
and, for any compact subset $K$ of $\R$,
\begin{equation}
\label{hyp:w-cont2}
\Psi_{n, 0} \to \Psi_0 \quad {\rm in} \ L^\infty(K), 
\end{equation}
as $n \to + \infty$. We claim that the convergences provided by \eqref{hyp:w-cont1} and \eqref{hyp:w-cont2} are conserved along the Gross-Pitaevskii flow.

\begin{prop}
\label{prop:w-cont-Psi}
We consider a sequence $(\Psi_{n, 0})_{n \in \N} \in \boX(\R)^\N$, and a function $\Psi_0 \in \boX(\R)$ such that assumptions \eqref{hyp:w-cont1} and \eqref{hyp:w-cont2} are satisfied, and we denote $\Psi_n$, respectively $\Psi$, the unique global solutions to \eqref{GP} with initial datum $\Psi_{n, 0}$, respectively $\Psi_0$, given by Proposition \ref{prop:Cauchy-Psi}. For any fixed $t \in \R$ and any compact subset $K$ of $\R$, we have 
\begin{equation}
\label{eq:w-cont1}
\Psi_n(\cdot, t) \to \Psi(\cdot, t) \quad {\rm in} \ L^\infty(K),
\end{equation}
when $n \to + \infty$, as well as
\begin{equation}
\label{eq:w-cont2}
\partial_x \Psi_n(\cdot, t) \rightharpoonup \partial_x \Psi(\cdot, t) \quad {\rm in} \ L^2(\R), \quad {\rm and} \quad 1 - |\Psi_n(\cdot, t)|^2 \rightharpoonup 1 - |\Psi(\cdot, t)|^2 \quad {\rm in} \ L^2(\R).
\end{equation}
\end{prop}

\begin{proof}
The proof is standard. For sake of completeness, we recall some details.

As usual, we first bound suitably the functions $\Psi_n$ and $\eta_n := 1 - |\Psi_n|^2$. In view of the weak convergences in assumption \eqref{hyp:w-cont1}, there exists a positive constant $M$ such that
$$\boE(\Psi_{n, 0}) \leq M^2,$$
for any $n \in \N$. Since the energy $\boE$ is conserved along the \eqref{GP} flow by \eqref{eq:cons-E}, we deduce that
\begin{equation}
\label{eq:w-bound1}
\| \partial_x \Psi_n(\cdot, t) \|_{L^2(\R)} \leq \sqrt{2} M, \quad {\rm and} \quad \| \eta_n(\cdot, t) \|_{L^2(\R)} \leq 2 M,
\end{equation}
for any $n \in \N$ and any $t \in \R$. Invoking the Sobolev embedding theorem, we next write
$$\| \Psi_n(\cdot, t) \|_{L^\infty(\R)}^2 \leq 1 + \| \eta_n(\cdot, t) \|_{L^\infty(\R)} \leq 1 + \| \eta_n(\cdot, t) \|_{L^2(\R)}^\frac{1}{2} \| \partial_x \eta_n(\cdot, t) \|_{L^2(\R)}^\frac{1}{2}.$$
Since
$$\| \partial_x \eta_n(\cdot, t) \|_{L^2(\R)} \leq 2 \| \Psi_n(\cdot, t) \|_{L^\infty(\R)} \| \partial_x \Psi_n(\cdot, t) \|_{L^2(\R)},$$
we obtain the uniform bounds
\begin{equation}
\label{eq:w-bound2}
\| \Psi_n(\cdot, t) \|_{L^\infty(\R)} \leq K_M, \quad {\rm and} \quad \| \partial_x \eta_n(\cdot, t) \|_{L^2(\R)} \leq K_M,
\end{equation}
where $K_M$ is a positive number depending only on $M$. In particular, given a fixed positive number $T$, we deduce that
\begin{equation}
\label{eq:w-bound3}
\int_0^T \int_\R |\partial_x \Psi_n(x, t)|^2 \, dx \, dt \leq M^2 T, \quad {\rm and} \quad \int_0^T \int_\R \eta_n(x, t)^2 \, dx \, dt \leq M^2 T.
\end{equation}

With bounds \eqref{eq:w-bound2} and \eqref{eq:w-bound3} at hand, we are in position to construct weak limits for the functions $\Psi_n$ and $\eta_n$. In view of \eqref{eq:w-bound3}, there exist two functions $\Phi_1 \in L^2(\R \times [0, T])$ and $N \in L^2(\R \times [0, T])$ such that, up to a further subsequence,
\begin{equation}
\label{casillas}
\partial_x \Psi_n \rightharpoonup \Phi_1 \quad {\rm in} \ L^2(\R \times [0, T]), \quad {\rm and} \quad \eta_n \rightharpoonup N \quad {\rm in} \ L^2(\R \times [0, T]),
\end{equation}
when $n \to \infty$. Similarly, we can invoke \eqref{eq:w-bound2} to exhibit a function $\Phi \in L^\infty(\R \times [0, T])$ such that, up to a further subsequence,
\begin{equation}
\label{ramos}
\Psi_n \overset{*}{\rightharpoonup} \Phi \quad {\rm in} \ L^\infty(\R \times [0, T]),
\end{equation}
when $n \to + \infty$. Combining with \eqref{casillas}, we remark that $\Phi_1 = \partial_x \Phi$ in the sense of distributions.

Our goal is now to check that the function $\Phi$ is solution to \eqref{GP}. This requires to improve the convergences in \eqref{casillas} and \eqref{ramos}. With this goal in mind, we introduce a cut-off function $\chi \in \boC_c^\infty(\R)$ such that $\chi \equiv 1$ on $[- 1, 1]$ and $\chi \equiv 0$ on $(- \infty, 2] \cup [2, + \infty)$, and we set $\chi_p(\cdot) := \chi(\cdot/p)$ for any integer $p \in \N^*$. In view of \eqref{eq:w-bound1} and \eqref{eq:w-bound3}, the sequence $(\chi_p \Psi_n)_{ n \in \N}$ is bounded in $\boC^0([0, T], H^1(\R))$. By the Rellich-Kondrachov theorem, the sets $\{ \chi_p \Psi_n(\cdot, t), n \in \N \}$ are relatively compact in $H^{- 1}(\R)$ for any fixed $t \in [0, T]$. On the other hand, the function $\Psi_n$ is solution to \eqref{GP}, so that its time derivative $\partial_t \Psi_n$ belongs to $\boC^0([0, T], H^{- 1}(\R))$ and satisfies
$$\| \partial_t \Psi_n(\cdot, t) \|_{H^{- 1}(\R)} \leq \| \partial_x \Psi_n(\cdot, t) \|_{L^2(\R)} + \| \Psi_n(\cdot, t) \|_{L^\infty(\R)} \| \eta_n(\cdot, t) \|_{L^2(\R)} \leq K_M.$$
As a consequence, the functions $\chi_p \Psi_n$ are equicontinuous in $\boC^0([0, T], H^{- 1}(\R))$. Applying the Arzela-Ascoli theorem and using the Cantor diagonal argument, we can find a further subsequence (independent of $p$), such that, for each $p \in \N^*$,
\begin{equation}
\label{pique}
\chi_p \Psi_n \to \chi_p \Phi \quad {\rm in} \ \boC^0([0, T], H^{- 1}(\R)),
\end{equation}
as $n \to + \infty$. Recalling that the functions $\chi_p \Psi_n$ are uniformly bounded in $\boC^0([0, T], H^1(\R))$, we deduce that the convergence in \eqref{pique} also holds in the spaces $\boC^0([0, T], H^s(\R))$ for any $s < 1$. In particular, by the Sobolev embedding theorem, we obtain
\begin{equation}
\label{valdes}
\chi_p \Psi_n \to \chi_p \Phi \quad {\rm in} \ \boC^0([0, T], \boC^0(\R)),
\end{equation}
as $n \to + \infty$.

Such convergences are enough to establish that $\Phi$ is solution to \eqref{GP}. Let $h$ be a function in $\boC_c^\infty(\R)$. Since the functions $\chi_p \Psi_n$ are uniformly bounded in $\boC^0([0, T], \boC^0(\R))$, we check (for an integer $p$ such that $\supp(h) \subset [- p, p]$) that
\begin{equation}
\label{jordialba}
h \eta_n(\cdot, t) = h \big( 1 - \chi_p^2 |\Psi_n(\cdot, t)|^2 \big) \to h \big( 1 - \chi_p^2 |\Phi(\cdot, t)|^2 \big) = h \big( 1 - |\Phi(\cdot, t)|^2 \big) \quad {\rm in} \ \boC^0(\R),
\end{equation}
as $n \to + \infty$, the convergence being uniform with respect to $t \in [0, T]$. In view of \eqref{casillas}, we deduce that $N = 1 - |\Phi|^2$. Similarly, we compute
\begin{equation}
\label{arbeloa}
h \Psi_n(\cdot, t) = h \chi_p \Psi_n(\cdot, t) \to h \chi_p \Phi(\cdot, t) = h \Phi(\cdot, t) \quad {\rm in} \ \boC^0(\R),
\end{equation}
as $n \to + \infty$. In view of \eqref{casillas}, we infer that
$$h \eta_n \Psi_n \to h (1 - |\Phi|^2) \Phi \quad {\rm in} \ L^2(\R \times [0, T]).$$
Going back to \eqref{casillas} and \eqref{ramos}, we recall that
$$i \partial_t \Psi_n \to i \partial_t \Phi \quad {\rm in} \ \boD'(\R \times [0, T]), \quad {\rm and} \quad \partial_{xx}^2 \Psi_n \to \partial_{xx}^2 \Phi \quad {\rm in} \ \boD'(\R \times [0, T]),$$
as $n \to + \infty$, so that it remains to take the limit $n \to + \infty$ in the expression
$$\int_0^T \int_\R \big( i \partial_t \Psi_n + \partial_{xx}^2 \Psi_n + \eta_n \Psi_n \big) \overline{h} = 0,$$
where $h \in \boC_c^\infty(\R \times [0, T])$, in order to establish that $\Phi$ is solution to \eqref{GP} in the sense of distributions. Moreover, we infer from \eqref{hyp:w-cont2} and \eqref{arbeloa} that $\Phi(\cdot, 0) = \Psi_0$.

In order to prove that the function $\Phi$ coincides with the solution $\Psi$ in Proposition \ref{prop:w-cont-Psi}, it is sufficient, in view of the uniqueness result in Proposition \ref{prop:Cauchy-Zs}, to establish the

\begin{claim}
The function $\Phi$ belongs to $\boC^0([0, T], \boZ^s(\R))$ for any $1/2 < s < 1$.
\end{claim}

\begin{proof}[Proof of the claim]
Let $t \in [0, T]$ be fixed. We deduce from \eqref{eq:w-bound1}, \eqref{pique} and \eqref{jordialba} that, up to a subsequence (depending on $t$),
\begin{equation}
\label{busquets}
\partial_x \Psi_n(\cdot, t) \rightharpoonup \partial_x \Phi(\cdot, t) \quad {\rm in} \ L^2(\R), \quad {\rm and} \quad \eta_n(\cdot, t) \rightharpoonup 1 - |\Phi(\cdot, t)|^2 \quad {\rm in} \ L^2(\R),
\end{equation}
as $n \to + \infty$. We also know that
\begin{equation}
\label{xabialonso}
\int_\R |\partial_x \Phi(\cdot, t)|^2 \leq 2 M^2, \quad {\rm and} \quad \int_\R \big( 1 - |\Phi(\cdot, t)|^2 \big)^2 \leq 4 M^2.
\end{equation}
In particular, the maps $\partial_x \Phi$ and $1 - |\Phi|^2$ belong to $L^\infty([0, T], L^2(\R))$, respectively $L^\infty([0, T], \linebreak[1] H^1(\R))$. Since
$$i \partial_t \big( \partial_x \Phi \big) = - \partial_{xxx}^3 \Phi - \partial_x( \eta \Phi),$$
the derivative $\partial_x \Phi$ actually belongs to $W^{1, \infty}([0, T], H^{- 2}(\R))$. Hence, it is continuous with values into $H^{- 2}(\R)$. By \eqref{xabialonso}, it remains continuous with values into $H^s(\R)$ for any $- 2 \leq s < 0$. Similarly, the functions $\eta_n$ solve the equations
\begin{equation}
\label{xavi}
\partial_t \eta_n = 2 \partial_x \big( \langle i \partial_x \Psi_n, \Psi_n \rangle_\C \big).
\end{equation}
Invoking \eqref{casillas} and \eqref{arbeloa}, we know that
$$h \langle i \partial_x \Psi_n, \Psi_n \rangle_\C \to h \langle i \partial_x \Phi, \Phi \rangle_\C \quad {\rm in} \ L^2(\R \times [0, T]),$$
for any $h \in \boC_c^\infty(\R)$. Using \eqref{jordialba} to take the limit $n \to + \infty$ into \eqref{xavi}, we are led to
$$\partial_t \big( 1 - |\Phi|^2 \big) = 2 \partial_x \big( \langle i \partial_x \Phi, \Phi \rangle_\C \big),$$
in the sense of distributions. We deduce as above that the map $1 - |\Phi|^2$ belongs to $W^{1, \infty}([0, T], \linebreak[1] H^{- 1}(\R))$, therefore that it is continuous with values into $H^{- 1}(\R)$, and finally with values into $H^s(\R)$ for any $- 1 \leq s < 1$. At this stage, it suffices to apply the Sobolev embedding theorem to guarantee that $\Phi$ is also in $\boC^0([0, T], L^\infty(\R))$, and, as a consequence, in $\boC^0([0, T], \boZ^s(\R))$ for any $1/2 < s < 1$, which proves the claim.
\end{proof}

By Proposition \ref{prop:Cauchy-Zs}, the maps $\Phi$ and $\Psi$ are therefore two identical solutions to \eqref{GP} in $\boC^0([0, T], \boZ^s(\R))$ for $1/2 < s < 1$. Arguing as in \eqref{busquets}, we conclude that, given any fixed number $t \in [0, T]$, we have, up to a subsequence (depending on $t$), 
\begin{equation}
\label{iniesta}
\partial_x \Psi_n \rightharpoonup \partial_x \Psi(\cdot, t) \quad {\rm in} \ L^2(\R), \quad {\rm and} \quad \eta_n \rightharpoonup 1 - |\Psi(\cdot, t)|^2 \quad {\rm in} \ L^2(\R).
\end{equation}
Given any compact subset $K$ of $\R$, we also deduce from \eqref{valdes} that
$$\Psi_n(\cdot, t) \to \Psi(\cdot, t) \quad {\rm in} \ L^\infty(K),$$
as $n \to + \infty$.

In order to complete the proof of Proposition \ref{prop:w-cont-Psi}, we now argue by contradiction assuming the existence of a positive number $T$, a function $h \in L^2(\R)$ and a positive number $\delta$ such that we have
$$\bigg| \int_\R \big( \partial_x \Psi_{\varphi(n)}(x, T) - \partial_x \Psi(x, T) \big) \, \overline{h(x)} \, dx \bigg| > \delta,$$
for a subsequence $(\Psi_{\varphi(n)})_{n \in \N}$. Up to the choice of a further subsequence (possibly depending on $T$), this in contradiction with \eqref{iniesta}. Here, we have made the choice to deny one of the weak convergences in \eqref{eq:w-cont2}, but a contradiction identically appears when \eqref{eq:w-cont1} or the other convergence in \eqref{eq:w-cont2} is alternatively denied. Since the proof extends with no change to the case where $T$ is negative, this concludes the proof of Proposition \ref{prop:w-cont-Psi}.
\end{proof}

A natural framework for solving the hydrodynamical form of the Gross-Pitaevskii equation is provided by the functions spaces
$$\boN\boV^k(\R) := \big\{ (\eta, v) \in X^k(\R), \ {\rm s.t.} \ \max_{x \in \R} \eta(x) < 1 \big\},$$
where we have set
$$X^k(\R) := H^{k + 1}(\R) \times H^k(\R).$$
A counter-part of Proposition \ref{prop:Cauchy-Psi} in terms of \eqref{HGP} is stated as follows.

\begin{prop}[\cite{Tartous0}]
\label{prop:Cauchy-Q}
Let $k \in \N$ and $(\eta^0, v^0) \in \boN\boV^k(\R)$. There exists a maximal time $T_{\max} > 0$ and a unique solution $(\eta, v) \in \boC^0([0, T_{\max}), \boN\boV^k(\R))$ to \eqref{HGP} with initial datum $(\eta^0, v^0)$. The maximal time $T_{\max}$ is continuous with respect to the initial datum $(\eta^0, v^0)$, and is characterized by 
$$\lim_{t \to T_{\max}} \max_{x \in \R} \eta(x, t) = 1 \quad {\rm if} \quad T_{\rm max} < + \infty.$$ 
Moreover, the energy $E$ and the momentum $P$ are constant along the flow. 
\end{prop}

In this setting, it is possible to establish the following version of the weak continuity of the hydrodynamical flow.

\begin{prop}
\label{prop:w-cont-Q}
We consider a sequence $(\eta_{n, 0}, v_{n, 0})_{n \in \N} \in \boN\boV(\R)^\N$, and a pair $(\eta_0, v_0) \in \boN\boV(\R)$ such that
\begin{equation}
\label{hyp:w-cont-Q}
\eta_{n, 0} \rightharpoonup \eta_0 \quad {\rm in} \ H^1(\R), \quad {\rm and} \quad v_{n, 0} \rightharpoonup v_0 \quad {\rm in} \ L^2(\R),
\end{equation}
as $n \to + \infty$. We denote by $(\eta_n, v_n)$ the unique solutions to \eqref{HGP} with initial data $(\eta_{n, 0}, v_{n, 0})$ given by Proposition \ref{prop:Cauchy-Q}, and we assume that there exists a positive number $T$ such that the solutions $(\eta_n, v_n)$ are defined on $[- T, T]$, and satisfy the condition
\begin{equation}
\label{loire}
\sup_{n \in \N} \, \sup_{t \in [- T, T]} \, \max_{x \in \R} \eta_n(x, t) \leq 1 - \sigma,
\end{equation}
for a given positive number $\sigma$. Then, the unique solution $(\eta, v)$ to \eqref{HGP} with initial data $(\eta_0, v_0)$ is also defined on $[- T, T]$, and for any $t \in [- T, T]$, we have
\begin{equation}
\label{eq:w-cont-Q}
\eta_n(t) \rightharpoonup \eta(t) \quad {\rm in} \ H^1(\R), \quad {\rm and} \quad v_n(t) \rightharpoonup v(t) \quad {\rm in} \ L^2(\R),
\end{equation}
as $n \to + \infty$.
\end{prop}

\begin{proof}
The proof relies on applying Proposition \ref{prop:w-cont-Psi} to the solutions $\Psi_n$ and $\Psi$ to \eqref{GP} with initial data
$$\Psi_{n, 0} := \sqrt{1 - \eta_{n, 0}} e^{i \varphi_{n, 0}}, \quad {\rm and} \quad \Psi_0 := \sqrt{1 - \eta_0} e^{i \varphi_0},$$
where we have set
\begin{equation}
\label{beaugency}
\varphi_{n, 0}(x) := \int_0^x v_{n, 0}(y) \, dy, \quad {\rm and} \quad \varphi_0(x) := \int_0^x v_0(y) \, dy.
\end{equation}
The weak convergences in \eqref{eq:w-cont-Q} then follow from the convergences in \eqref{eq:w-cont1} and \eqref{eq:w-cont2}.

With this goal in mind, we first remark that the map $\varphi_0$ in \eqref{beaugency} defines a continuous function with derivative $v_0$ in $L^2(\R)$, while $\sqrt{1 - \eta_0}$ defines a function in $H^1(\R)$. As a consequence, the function $\Psi_0$, and similarly the functions $\Psi_{n, 0}$, are well-defined on $\R$ and belong to $\boX(\R)$, with derivatives
\begin{equation}
\label{tours}
\partial_x \Psi_0 = \Big( - \frac{\partial_x \eta_0}{2 \sqrt{1 - \eta_0}} + i \sqrt{1 - \eta_0} v_0 \Big) e^{i \varphi_0}, \qquad \partial_x \Psi_{n, 0} = \Big( - \frac{\partial_x \eta_{n, 0}}{2 \sqrt{1 - \eta_{n, 0}}} + i \sqrt{1 - \eta_{n, 0}} v_{n, 0} \Big) e^{i \varphi_{n, 0}}.
\end{equation}

We now check the first assumption in \eqref{hyp:w-cont1}, as well as \eqref{hyp:w-cont2}. The second assumption in \eqref{hyp:w-cont1} is already included in \eqref{hyp:w-cont-Q}. In view of \eqref{beaugency}, we write
$$\varphi_{n, 0}(x) - \varphi_0(x) = \langle v_{n, 0} - v_0, 1_{[0, x]} \rangle_{L^2(\R)},$$
for any $x \in \R$, so that, by \eqref{hyp:w-cont-Q},
$$\varphi_{n, 0}(x) \to \varphi_0(x),$$
as $n \to + \infty$. On the other hand, it again follows from \eqref{beaugency} that 
$$\big| \varphi_{n, 0}(x) - \varphi_{n, 0}(y) \big| \leq |x - y|^\frac{1}{2} \| v_{n, 0} \|_{L^2(\R)},$$
for any $(x, y) \in \R^2$. Given a compact subset $K$ of $\R$, we deduce from the Ascoli-Arzela theorem and the Cantor diagonal argument that, passing to a subsequence independent of $K$, we have
$$\varphi_{n, 0} \to \varphi_0 \quad {\rm in} \ L^\infty(K),$$
as $n \to + \infty$. In particular, 
\begin{equation}
\label{blois1}
e^{i \varphi_{n, 0}} \to e^{i \varphi_0} \quad {\rm in} \ L^\infty(K),
\end{equation}
as $n \to + \infty$. Similarly, if follows from \eqref{hyp:w-cont-Q} and the Rellich-Kondrachov theorem that, up to a further subsequence,
\begin{equation}
\label{blois2}
\sqrt{1 - \eta_{n, 0}} \to \sqrt{1 - \eta_0} \quad {\rm in} \ L^\infty(K),
\end{equation}
as $n \to + \infty$. Since the maps $e^{i \varphi_{n, 0}}$ are uniformly bounded by $1$, we conclude that
$$\Psi_{n, 0} \to \Psi_0 \quad {\rm in} \ L^\infty(K),$$
as $n \to + \infty$.

The proof of the first assumption in \eqref{hyp:w-cont1} is similar. We deduce from \eqref{loire} and \eqref{blois2} that
$$\sqrt{1 - \eta_{n, 0}} \geq \sqrt{\sigma}, \quad {\rm and} \quad \sqrt{1 - \eta_0} \geq \sqrt{\sigma} \ {\rm on} \ \R.$$
Combining \eqref{tours} with the convergences in \eqref{hyp:w-cont-Q}, \eqref{blois1} and \eqref{blois2}, we are led to
$$\partial_x \Psi_{n, 0} \rightharpoonup \partial_x \Psi_0 \quad {\rm in} \ L^2(\R),$$
as $n \to + \infty$.

As a consequence, we can apply Proposition \ref{prop:w-cont-Psi} to the solutions $\Psi_n$ and $\Psi$ to \eqref{GP} with initial data $\Psi_{n, 0}$, respectively $\Psi_0$. Given any number $t \in \R$, we obtain in the limit $n \to + \infty$,
\begin{equation}
\label{chambord1}
\Psi_n(\cdot, t) \to \Psi(\cdot, t) \quad {\rm in} \ L^\infty(K),
\end{equation}
for any compact subset $K$ of $\R$, as well as
\begin{equation}
\label{chambord2}
\partial_x \Psi_n(\cdot, t) \rightharpoonup \partial_x \Psi(\cdot, t) \quad {\rm in} \ L^2(\R), \quad {\rm and} \quad 1 - |\Psi_n(\cdot, t)|^2 \rightharpoonup 1 - |\Psi(\cdot, t)|^2 \quad {\rm in} \ L^2(\R).
\end{equation}
Setting
$$\widetilde{\eta}_n := 1 - |\Psi_n|^2, \quad {\rm and} \quad \widetilde{\eta} := 1 - |\Psi|^2,$$
we infer similarly from \eqref{chambord1}, \eqref{chambord2} and the identities $\partial_x \widetilde{\eta}_{(n)} = - 2 \langle \Psi_{(n)}, \partial_x \Psi_{(n)} \rangle_\C$ that
\begin{equation}
\label{chenonceaux}
\widetilde{\eta}_n(\cdot, t) \rightharpoonup \widetilde{\eta}(\cdot, t) \quad {\rm in} \ H^1(\R),
\end{equation}
as $n \to + \infty$. In order to derive the first convergence in \eqref{eq:w-cont-Q}, it remains to check that the functions $\widetilde{\eta}_n$ and $\widetilde{\eta}$ are equal to $\eta_n$, respectively $\eta$. This can be done by invoking the uniqueness result in Proposition \ref{prop:Cauchy-Q} for the solutions to \eqref{HGP}.

With this goal in mind, we first derive from the Sobolev embedding theorem that
$$1 - |\psi_p|^2 \to 1 - |\psi|^2 \quad {\rm in} \ L^\infty(\R),$$
when $\psi_p \to \psi$ in $\boX(\R)$ as $p \to + \infty$. Since $\eta_{n, 0}$ satisfies \eqref{loire}, and $\Psi_n$ is continuous from $\R$ to $\boX(\R)$, we can exhibit a number $\tau_n \in (0, T)$, depending possibly on $n$, such that
\begin{equation}
\label{amboise}
\sup_{t \in [- \tau_n, \tau_n]} \, \max_{x \in \R} \tilde \eta_n(x, t) \leq 1 - \frac{\sigma}{2}.
\end{equation}
As a consequence, we can define a function $\widetilde{v}_n : \R \times [- \tau_n, \tau_n] \to \R$ according to the expression
$$\widetilde{v}_n = \frac{\langle i \Psi_n, \partial_x \Psi_n \rangle_\C}{1 - |\Psi_n|^2}.$$
Since $\Psi_n$ is in $\boC^0([- \tau_n, \tau_n], L^\infty(\R))$, the function $\widetilde{v}_n$ is continuous on $[- \tau_n, \tau_n]$ with values into $L^2(\R)$. Similarly, $\widetilde{\eta}_n$ is continuous on $[- \tau_n, \tau_n]$ with values in $H^1(\R)$. In view of \eqref{amboise}, we conclude that the pair $(\widetilde{\eta}_n, \widetilde{v}_n)$ belongs to $\boC^0([- \tau_n, \tau_n], \boN\boV(\R))$. Moreover, the map $\Psi_n$ being a solution to \eqref{GP}, the pair $(\widetilde{\eta}_n, \widetilde{v}_n)$ solves \eqref{HGP} in the sense of distributions for an initial data equal to $(\eta_{n, 0}, v_{n, 0})$. As a conclusion, this pair coincides with the solution $(\eta_n, v_n)$ on $[- \tau_n, \tau_n]$. Using a standard connectedness argument, we derive that the function $v_n$ is well-defined in $\boC^0([- T, T], L^\infty(\R))$, and that
$$(\widetilde{\eta}_n(x, t), \widetilde{v}_n(x,t)) = (\eta_n(x, t), v_n(x,t)),$$
for any $x \in \R$ and $t \in [- T, T]$.

Due to \eqref{blois2}, one can rely on the same approach to establish that the function
$$\widetilde{v} = \frac{\langle i \Psi, \partial_x \Psi \rangle_\C}{1 - |\Psi|^2},$$
is well-defined in $\boC^0([- T, T], L^\infty(\R))$, and that
$$(\widetilde{\eta}(x, t), \widetilde{v}(x,t)) = (\eta(x, t), v(x,t)),$$
for any $x \in \R$ and any $t \in [- T, T]$. The first convergence in \eqref{eq:w-cont-Q} is then exactly \eqref{chenonceaux}. Concerning the second one, we deduce from \eqref{loire}, \eqref{chambord1} and \eqref{chambord2} that
$$\frac{\langle i \Psi_n, \partial_x \Psi_n \rangle_\C}{1 - |\Psi_n|^2} \rightharpoonup \frac{\langle i \Psi, \partial_x \Psi \rangle_\C}{1 - |\Psi|^2} \quad {\rm in} \ L^2(\R),$$
as $n \to + \infty$. This is exactly the desired convergence.

However, the two convergences are only available for a subsequence, so that we have to argue by contradiction as in the proof of Proposition \ref{prop:w-cont-Psi} to conclude the proof of Proposition \ref{prop:w-cont-Q}. 
\end{proof}

\subsubsection{Proof of Proposition \ref{prop:reprod}}

In order to establish \eqref{sologne1}, we apply Proposition \ref{prop:w-cont-Q}. Relying on assumption \eqref{eq:assump2} and the explicit formula for $Q_{c(t_n)}$ in \eqref{form:etavc}, we check that
$$Q_{c(t_n)} \to Q_{c_0^*} \quad {\rm in} \ X(\R),$$
as $n \to + \infty$. Combining with \eqref{eq:assump1}, we are led to
$$\big( \eta(\cdot + a(t_n), t_n), v(\cdot + a(t_n), t_n) \big) \rightharpoonup \eps_0^* + Q_{c_0^*} \quad {\rm in} \ X(\R),$$
as $n \to + \infty$. The weak convergence in \eqref{sologne1} then appears as a direct consequence of \eqref{eq:w-cont-Q} since $t \mapsto (\eta(\cdot + a(t_n), t_n + t), v(\cdot + a(t_n), t_n + t))$ and $(\eta^*, v^*)$ are the solutions to \eqref{HGP} with initial data $(\eta(\cdot + a(t_n), t_n), v(\cdot + a(t_n), t_n))$, respectively $\eps_0^* + Q_{c_0^*}$, and since assumption \eqref{loire} is satisfied in view of \eqref{eq:max-eta}.
 
Concerning \eqref{sologne2}, we rely on \eqref{eq:modul0} to claim that the map $t \mapsto c(t)$ is bounded on $\R$. We next combine \eqref{eq:modul0} and \eqref{eq:modul1} to show that $a'$ is a bounded function on $\R$. As a consequence, the sequences $(a(t_n + t) - a(t_n))_{n \in \N}$ and $(c(t_n + t))_{n \in \N}$ are bounded, so that the proof of \eqref{sologne2} reduces to establish that the unique possible accumulation points for these sequences are $a^*(t)$, respectively $c^*(t)$.

In order to derive this further property, we assume that, up to a possible subsequence, we have
\begin{equation}
\label{beuvron}
a(t_n + t) - a(t_n) \to \alpha, \quad {\rm and} \quad c(t_n + t) \to \sigma,
\end{equation}
as $n \to + \infty$. Given a function $\phi \in H^1(\R)$, we next write
\begin{align*}
\big\langle \eta(\cdot + a(t_n + t), t_n + t), \phi \big\rangle_{H^1(\R)} = & \big\langle \eta(\cdot + a(t_n), t_n + t), \phi(\cdot - a(t_n + t) + a(t_n)) - \phi(\cdot - \alpha) \big\rangle_{H^1(\R)}\\
& + \big\langle \eta(\cdot + a(t_n), t_n + t), \phi(\cdot - \alpha) \big\rangle_{H^1(\R)}.
\end{align*}
Combining \eqref{sologne1} and \eqref{beuvron} with the well-known fact that
$$\phi(\cdot + h) \to \phi \quad {\rm in} \ H^1(\R),$$
when $h \to 0$, we deduce that
$$\eta(\cdot + a(t_n + t), t_n + t) \rightharpoonup \eta^*(\cdot + \alpha, t) \quad {\rm in} \ H^1(\R),$$
as $n \to + \infty$. Similarly, we have
$$v(\cdot + a(t_n + t), t_n + t) \rightharpoonup v^*(\cdot + \alpha, t) \quad {\rm in} \ L^2(\R).$$
Since
$$Q_{c(t_n + t)} \to Q_{\sigma} \quad {\rm in} \ X(\R),$$
as $n \to + \infty$ by \eqref{beuvron}, we also obtain
\begin{equation}
\label{neung}
\eps(\cdot, t_n + t) \rightharpoonup \big( \eta^*(\cdot + \alpha, t), v^*(\cdot + \alpha, t) \big) - Q_\sigma \quad {\rm in} \ X(\R),\end{equation}
as $n \to + \infty$.

At this stage, we again rely on the second convergence in \eqref{beuvron} to prove that
$$\partial_x Q_{c(t_n + t)} \to \partial_x Q_\sigma \quad {\rm in} \ L^2(\R)^2,$$
as $n \to + \infty$, and the similar convergence for $P'(Q_{c(t_n + t)})$. With \eqref{neung} at hand, this is enough to take the limit $n \to + \infty$ in the two orthogonality conditions in \eqref{eq:ortho} in order to get the identities
$$\big\langle (\eta^*(\cdot + \alpha, t), v^*(\cdot + \alpha, t)) - Q_\sigma, \partial_x Q_\sigma \big\rangle_{L^2(\R)^2} = P'(Q_\sigma) \big( (\eta^*(\cdot + \alpha, t), v*(\cdot + \alpha, t)) - Q_\sigma \big) = 0.$$
Using the uniqueness of the parameters $\alpha^*(t)$ and $c^*(t)$ in \eqref{def:eps*}, we deduce that
\begin{equation}
\label{romorantin}
\alpha = a^*(t), \quad {\rm and} \quad \sigma = c^*(t),
\end{equation}
which is enough to complete the proof of \eqref{sologne2}. Convergence \eqref{sologne3} follows combining \eqref{def:eps*} with \eqref{neung} and \eqref{romorantin}. \qed

\subsection{Smoothing properties for space localized solutions}
\label{sec:smoothing}

We consider a solution $u \in \boC^0(\R, L^2(\R))$ to the inhomogeneous linear Schr\"odinger equation \eqref{eq:LS}, with $F \in L^2(\R, L^2(\R))$, and we assume that
\begin{equation}
\label{morvan}
\int_{- T}^T \int_\R \big( |u(x ,t)|^2 + |F(x, t)|^2 \big) e^{\lambda x} \, dx \, dt < + \infty,
\end{equation}
for any positive number $T$. Our goal is to establish that the exponential decay for $u$ and $F$ in \eqref{morvan} induces a smoothing effect on $u$ in such a way that $\partial_x u$ belongs to $L_{\rm loc}^2(\R \times \R)$. In order to derive this effect, we rely on the following virial type identity. We refer to the work by Escauriaza, Kenig, Ponce and Vega \cite{EsKePoV5} for useful extensions in related contexts.

\begin{lemma}
\label{lem:viriel}
Let $u$ be a solution in $\boC^0(\R, H^1(\R))$ to \eqref{eq:LS}, with $F \in L^2(\R, H^1(\R))$. We consider two real numbers $a < b$, a function $\chi \in \boC^2(\R)$ such that $\chi(a) = \chi(b) = 0$, and a bounded function $\Phi \in \boC^4(\R)$, with bounded derivatives. Then, we have
\begin{equation}
\label{eq:viriel}
\begin{split}
& 4 \int_a^b \int_\R |\partial_x u(x, t)|^2 \Phi''(x) \chi(t) \, dx \, dt = \int_\R \Big( |u(x, a)|^2 \chi'(a) - |u(x, b)|^2 \chi'(b) \Big) \Phi(x) \, dx\\
& + \int_a^b \int_\R |u(x, t)|^2 \Big( \Phi(x) \chi''(t) + \Phi^{(4)}(x) \chi(t) \Big) dx \, dt + 2 \int_a^b \int_\R \langle F(x, t), i \, u(x, t) \rangle_\C \Phi(x) \chi'(t) \, dx \, dt\\
& - 2 \int_a^b \int_\R \langle F(x, t), u(x, t) \rangle_\C \Phi''(x) \chi(t) \, dx \, dt - 4 \int_a^b \int_\R \langle F(x, t), \partial_x u(x, t) \rangle_\C \Phi'(x) \chi(t) \, dx \, dt.
\end{split}
\end{equation}
\end{lemma}

\begin{proof}
We introduce the map $\Xi$ given by
$$\Xi(t) = \int_\R |u(x, t)|^2 \Phi(x) \, dx,$$
for any $t \in \R$. When $u$ is a smooth solution to \eqref{eq:LS}, we are allowed to compute
$$\Xi'(t) = 2 \int_\R \langle F(x, t), i u(x, t) \rangle_\C \Phi(x) \, dx + 2 \int_\R \langle \partial_x u(x, t), i u(x, t) \rangle_\C \Phi'(x) \, dx,$$
as well as
\begin{equation}
\begin{split}
\label{avallon}
\Xi''(t) & = 2 \partial_t \bigg( \int_\R \langle F(x, t), i u(x, t) \rangle_\C \Phi(x) \, dx \bigg) + 4 \int_\R \langle F(x, t), \partial_x u(x, t) \rangle_\C \Phi'(x) \, dx\\
& + 2 \int_\R \langle F(x, t), u(x, t) \rangle_\C \Phi''(x) \, dx + 4 \int_\R |\partial_x u(x, t)|^2 \Phi''(x) \, dx - \int_\R |u(x, t)|^2 \Phi^{(4)}(x) \, dx.
\end{split}
\end{equation}
Formula \eqref{eq:viriel} follows by writing the identity
$$\int_a^b \Xi''(t) \chi(t) \, dt = - \Xi(b) \chi'(b) + \Xi(a) \chi'(a) + \int_a^b \Xi(t) \chi''(t) \, dt,$$
and integrating by parts (with respect to $t$) the first integral in the right-hand side of \eqref{avallon}. 

When $u$ is only in $\boC^0(\R, H^1(\R))$, we introduce a sequence of smooth functions $(u_{m, a})_{m \in \N}$ and $(F_m)_{m \in \N}$ such that
\begin{equation}
\label{brassy}
u_{m, a} \to u(\cdot, a) \quad {\rm in} \ H^1(\R), \quad {\rm and} \quad F_m \to F \quad {\rm in} \ L^2(\R, H^1(\R)),
\end{equation}
as $m \to + \infty$. We denote by $u_m$ the unique solution to \eqref{eq:LS}, in which $F$ is replaced by $F_m$, with $u_m(\cdot, a) = u_{m, a}$. Since $u_m$ is a smooth solution to \eqref{eq:LS}, identity \eqref{eq:viriel} holds for the functions $u_m$ and $F_m$. On the other hand, we can deduce from the convergences in \eqref{brassy} applying an energy method to \eqref{eq:LS} that
$$u_m \to u \quad {\rm in} \ \boC^0(\R, H^1(\R)),$$
when $m \to + \infty$. Combining with \eqref{brassy} and taking the limit $m \to + \infty$, we obtain identity \eqref{eq:viriel} for the functions $u$ and $F$.
\end{proof}

\subsubsection{Proof of Proposition \ref{prop:smoothing}}

We apply Lemma \ref{lem:viriel} with $a = - T - 1$ and $b = T + 1$, and for a function $\chi \in \boC_c^2(\R, [0, 1])$, with compact support in $[- T - 1, T + 1]$, and such that $\chi = 1$ on $[- T, T]$.

Concerning the choice of the function $\Phi$, we would like to set $\Phi(x) = e^{\lambda x}$ for any $x \in \R$. However, this function is not bounded, as well as its derivatives. In order to by-pass this difficulty, we introduce a function $\phi \in \boC^\infty(\R, [0, 1])$ with compact support in $[- 2, 2]$ and such that $\phi = 1$ on $[- 1, 1]$, and we set
$$\phi_n(x) = \phi \Big( \frac{x}{n} \Big),$$
for any $n \in \N^*$ and any $x \in \R$. We then apply Lemma \ref{lem:viriel} to the function
$$\Phi_n(x) = \phi_n(x) e^{\lambda x},$$
which is bounded, with bounded derivatives.

At this stage, we have to face a second difficulty. Lemma \ref{lem:viriel} is available for functions $u$ and $F$ in $\boC^0(\R, H^1(\R))$, respectively $L^2(\R, H^1(\R))$, but we would like to apply it when $u$ and $F$ are only in $\boC^0(\R, L^2(\R))$, respectively $L^2(\R, L^2(\R))$. As a consequence, we first mollify the functions $u$ and $F$ by introducing a smooth function $\mu \in \boC_c^\infty(\R \times \R)$, with compact support in $[- 1, 1]^2$ and such that $\int_{\R^2} \mu = 1$, and by setting
\begin{equation}
\label{ouroux}
u_m = u \star \mu_m, \quad {\rm and} \quad F_m = F \star \mu_m,
\end{equation}
with $\mu_m(x, t) = m^2 \mu(m x, m t)$ for any $m \in \N$ and any $(x, t) \in \R^2$. In a second step, we will complete the proof by taking the limit $m \to + \infty$.

Since $F$ is in $L^2(\R, L^2(\R))$, we first deduce from \eqref{ouroux} and the Young inequality that $F_m$ belongs to $L^2(\R, H^1(\R))$, with the bounds
$$\| \partial_x^\ell F_m \|_{L^2(\R, L^2(\R))} \leq m^\ell \| F \|_{L^2(\R, L^2(\R))} \| \partial_x^\ell \mu \|_{L^1(\R^2)},$$
for $\ell \in \{ 0, 1 \}$. Similarly, we compute
\begin{align*}
\int_\R \big| \partial_x^\ell u_m(x, t) & - \partial_x^\ell u_m(x, t_0) \big|^2 dx\\
& \leq m^{2 \ell} \| \partial_x^\ell \mu \|_{L^1(\R^2)} \int_{- 1}^1 \int_{- 1}^1 \Big\| u \big( \cdot, t - \frac{\tau}{m} \big) - u \big( \cdot, t_0 - \frac{\tau}{m} \big) \Big\|_{L^2(\R)}^2 |\partial_x^\ell \mu(\tau)| \, d\tau,
\end{align*}
so that $u_m$ belongs to $\boC^0(\R, H^1(\R))$, with the bound
$$\| \partial_x^\ell u_m \|_{\boC^0([- T - 1, T + 1], L^2(\R))} \leq m^\ell \| u \|_{\boC^0([- T - 1, T + 1], L^2(\R))} \| \partial_x^\ell \mu \|_{L^1(\R^2)},$$
which can be derived using the same arguments. As a consequence, we are in position to apply Lemma \ref{lem:viriel} to obtain the identity
\begin{equation}
\label{vezelay1}
\begin{split}
4 \int_\R \int_\R |\partial_x u_m(x, t)|^2 \, \big( & \phi_n(x) e^{\lambda x} \big)'' \chi(t) \, dx \, dt\\
= & 2 \int_\R \int_\R \langle F_m(x, t), i \, u_m(x, t) \rangle_\C \, \phi_n(x) e^{\lambda x} \chi'(t) \, dx \, dt\\
& - 2 \int_\R \int_\R \langle F_m(x, t), u_m(x, t) \rangle_\C \, \big( \phi_n(x) e^{\lambda x} \big)'' \chi(t) \, dx \, dt\\
& - 4 \int_\R \int_\R \langle F_m(x, t), \partial_x u_m(x, t) \rangle_\C \, \big( \phi_n(x) e^{\lambda x} \big)' \chi(t) \, dx \, dt\\
& + \int_\R \int_\R |u_m(x, t)|^2 \, \Big( \phi_n(x) e^{\lambda x} \chi''(t) + \big( \phi_n(x) e^{\lambda x} \big)^{(4)} \chi(t) \Big) dx \, dt.
\end{split}
\end{equation}

In order to take the limit $n \to + \infty$, we first combine \eqref{morvan} with \eqref{ouroux} to obtain the bound
\begin{align*}
& \int_{- T - 1}^{T + 1} \int_\R \big( |\partial_x^\ell u_m(x, t)|^2 + | \partial_x^\ell F_m(x, t)|^2 \big) e^{\lambda x} \, dx \, dt\\
\leq m^{2 \ell} \bigg( \int_\R & \int_\R |\partial_x^\ell \mu(x, t)| e^\frac{\lambda x}{2} \, dx \, dt \bigg)^2 \times \int_{- T - 2}^{T + 2} \int_\R \big( |u(x, t)|^2 + |F(x, t)|^2 \big) e^{\lambda x} \, dx \, dt < + \infty,
\end{align*}
for $\ell \in \{ 0, 1 \}$. It follows that all the integrals in \eqref{vezelay1} can be written under the form
$$I_n(k, G) = \int_{- T - 1}^{T + 1} \int_\R G(x, t) \phi_n^{(k)}(x) \, dx \, dt,$$
with $G \in L^1([- T - 1, T + 1], L^1(\R))$ and $0 \leq k \leq 4$. Since
$$I_n(k, G) \to \delta_{k, 0} \int_{- T - 1}^{T + 1} \int_\R G(x, t) \, dx \, dt,$$
as $n \to + \infty$ by the dominated convergence theorem, we obtain in the limit $n \to + \infty$,
\begin{equation}
\label{vezelay2}
\begin{split}
4 \lambda^2 \int_{- T - 1}^{T + 1} \int_\R |\partial_x u_m(x, t)|^2 \, e^{\lambda x} \chi(t) \, dx \, dt = & 2 \int_{- T - 1}^{T + 1} \int_\R \langle F_m(x, t), i \, u_m(x, t) \rangle_\C \, e^{\lambda x} \chi'(t) \, dx \, dt\\
& - 2 \lambda^2 \int_{- T - 1}^{T + 1} \int_\R \langle F_m(x, t), u_m(x, t) \rangle_\C \, e^{\lambda x} \chi(t) \, dx \, dt\\
& - 4 \lambda \int_{- T - 1}^{T + 1} \int_\R \langle F_m(x, t), \partial_x u_m(x, t) \rangle_\C \, e^{\lambda x} \chi(t) \, dx \, dt\\
& + \int_{- T - 1}^{T + 1} \int_\R |u_m(x, t)|^2 \, e^{\lambda x} \Big( \chi''(t) + \lambda^4 \chi(t) \Big) dx \, dt.
\end{split}
\end{equation}

We now use the inequality $2 \alpha \beta \leq \alpha^2 + \beta^2$ to write
\begin{equation}
\label{bazoches}
\begin{split}
& \bigg| 2 \int_{- T - 1}^{T + 1} \int_\R \langle F_m(x, t), i \, u_m(x, t) \rangle_\C \, e^{\lambda x} \chi'(t) \, dx \, dt \bigg|\\
\leq K_1 \bigg( \int_{- T - 1}^{T + 1} & \int_\R |u_m(x, t)|^2 \, e^{\lambda x} \, dx \, dt + \int_{- T - 1}^{T + 1} \int_\R |F_m(x, t)|^2 \, e^{\lambda x} \, dx \, dt \bigg),
\end{split}
\end{equation}
with $K_1 := \| \chi' \|_{L^\infty(\R)}$. Similarly, we have
\begin{align*}
& \bigg| 2 \int_{- T - 1}^{T + 1} \int_\R \Big( 2 \lambda \langle F_m(x, t), \partial_x u_m(x, t) \rangle_\C + \lambda^2 \langle F_m(x, t), u_m(x, t) \rangle_\C \Big) e^{\lambda x} \chi(t) \, dx \, dt \bigg|\\
\leq & \lambda^2 \int_{- T - 1}^{T + 1} \int_\R |u_m(x, t)|^2 \, e^{\lambda x} \, dx \, dt + 2 \lambda^2 \int_{- T - 1}^{T + 1} \int_\R |\partial_x u_m(x, t)|^2 \, e^{\lambda x} \chi(t) \, dx \, dt\\
& + \big( 2 + \lambda^2 \big) \int_{- T - 1}^{T + 1} \int_\R |F_m(x, t)|^2 \, e^{\lambda x} \, dx \, dt.
\end{align*}
Combining with \eqref{vezelay2} and \eqref{bazoches}, we obtain the inequality
\begin{align*}
2 \lambda^2 \int_{- T - 1}^{T + 1} \int_\R |\partial_x u_m(x, t)|^2 \, e^{\lambda x} \chi(t) \, dx \, dt \leq & \big( K_1 + K_2 + \lambda^2 + \lambda^4 \big) \int_{- T - 1}^{T + 1} \int_\R |u_m(x, t)|^2 e^{\lambda x} \, dx \, dt\\
& + \big( K_1 + \lambda^2 + 2 \big) \int_{- T - 1}^{T + 1} \int_\R |F_m(x, t)|^2 \, e^{\lambda x} \, dx \, dt,
\end{align*}
with $K_2 := \| \chi'' \|_{L^\infty(\R)}$. At this stage, we rely on the properties of the function $\chi$ to obtain the inequality
\footnote{The choice of $\chi$ can indeed be made so that $K_1$ and $K_2$ are independent of $T$.}
\begin{equation}
\label{chateau-chinon}
2 \lambda^2 \int_{- T}^T \int_\R |\partial_x u_m(x, t)|^2 e^{\lambda x} \, dx \, dt \leq K_\lambda \int_{- T - 1}^{T + 1} \int_\R \big( |u_m(x, t)|^2 + |F_m(x, t)|^2 \big) e^{\lambda x} \, dx \, dt,
\end{equation}
for some positive constant $K_\lambda$, depending only on $\lambda$.

In order to conclude the proof, we finally consider the limit $m \to + \infty$. Using the linearity of \eqref{eq:LS}, we can transform \eqref{chateau-chinon} into
\begin{equation}
\label{saulieu}
\begin{split}
& 2 \lambda^2 \int_{- T}^T \int_\R |\partial_x u_m(x, t) - \partial_x u_p(x, t)|^2 \, e^{\lambda x} \, dx \, dt\\
\leq K_\lambda \int_{- T - 1}^{T + 1} & \int_\R \big( |u_m(x, t) - u_p(x, t)|^2 + |F_m(x, t) - F_p(x, t)|^2 \big) \, e^{\lambda x} \, dx \, dt,
\end{split}
\end{equation}
for any $(m, p) \in (\N^*)^2$. On the other hand, we can check that
\begin{align*}
& \int_{- T - 1}^{T + 1} \int_\R |u_m(x, t) - u(x, t)|^2 \, e^{\lambda x} \, dx \, dt\\
\leq 4 \| \mu \|_{L^1(\R^2)}^2 & \sup_{|s| \leq 1, |y| \leq 1} \int_{- T - 1}^{T + 1} \int_\R \Big| u \Big( x - \frac{y}{m}, t - \frac{s}{m} \Big) - u(x, t) \Big|^2 \, e^{\lambda x} \, dx \, dt.
\end{align*}
Setting $v(x, t) = u(x, t) e^\frac{\lambda x}{2}$, we observe that
\begin{align*}
& \int_{- T - 1}^{T + 1} \int_\R \Big| u \Big( x - \frac{y}{m}, t - \frac{s}{m} \Big) - u(x, t) \Big|^2 \, e^{\lambda x} \, dx \, dt\\
\leq \int_{- T - 1}^{T + 1} \int_\R & \Big( |v(x, t)|^2 \Big| e^\frac{\lambda y}{2 m} - 1 \Big|^2 + \Big| v \Big( x - \frac{y}{m}, t - \frac{s}{m} \Big) - v(x, t) \Big|^2 e^\frac{\lambda y}{m} \Big) \, dx \, dt.
\end{align*}
Since $v \in L^2([- T - 2, T+ 2], L^2(\R))$ by \eqref{morvan}, we obtain the convergence
$$\int_{- T - 1}^{T + 1} \int_\R |u_m(x, t) - u(x, t)|^2 \, e^{\lambda x} \, dx \, dt \to 0,$$
as $m \to + \infty$. Due to \eqref{morvan} again, similar convergence holds for the functions $F_m$ and $F$.

In particular, we infer from \eqref{saulieu} that the functions $(x, t) \mapsto \partial_x u_m(x, t) e^\frac{\lambda x}{2}$ form a Cauchy sequence in $L^2([- T, T], L^2(\R))$. In view of \eqref{ouroux}, their limit in the sense of distributions is the map $(x, t) \mapsto \partial_x u(x, t) e^\frac{\lambda x}{2}$. As a consequence, this map belongs to $L^2([- T, T], L^2(\R))$, with
$$\int_{- T}^T \int_\R |\partial_x u_m(x, t) - \partial_x u(x, t)|^2 \, e^{\lambda x} \, dx \, dt \to 0,$$
as $m \to + \infty$. It is then enough to take the limit $m \to + \infty$ into \eqref{chateau-chinon} to obtain inequality \eqref{eq:smoothing}. This completes the proof of Proposition \ref{prop:smoothing}. \qed

\begin{remark}
Inequalities similar in spirit to \eqref{eq:smoothing} can be obtained with similar proofs replacing the weight function $e^{\lambda x}$ by $e^{\phi}$ where $\phi : \R \to \R$ is a smooth function with bounded derivatives and such that $\phi'' + (\phi')^2$ is bounded from below on $\R$. In those cases, we obtain inequalities of the form
$$\int_{- T}^T \int_\R |\partial_x u(x, t)|^2 e^{\phi(x)} \, dx \, dt \leq K_\phi \int_{- T - 1}^{T + 1} \int_\R \big( |u(x, t)|^2 + |F(x, t)|^2 \big) e^{\phi(x)} \, dx \, dt,$$
where $K_\phi$ is a positive constant depending only on $\phi$.
\end{remark}

\subsubsection{Proof of Proposition \ref{prop:smooth}}

We denote by $\Psi^* \in \boC(\R, \boX(\R))$ a solution (uniquely determined up to a constant phase shift) to \eqref{GP} corresponding to the solution $(\eta^*, v^*)$ to \eqref{HGP}. Formally, we may differentiate \eqref{GP} $k$ times with respect to the space variable and write the resulting equation as 
\begin{equation}
\label{eq:psikder}
i \partial_t \big( \partial_x^k \Psi^* \big) + \partial_{xx} \big( \partial_x^k \Psi^* \big) = R_k(\Psi^*),
\end{equation}
where, in view of the cubic nature of \eqref{GP},
\begin{equation}
\label{eq:decompRk}
\big| R_k(\Psi^*) \big| \leq |\partial_x^k \Psi^*| + \sum_{\substack{\alpha \leq \beta \leq \gamma \\ \alpha + \beta + \gamma = k}} K_{\alpha, \beta, \gamma} |\partial_x^\alpha \Psi^*| \, |\partial_x^\beta \Psi^*| \, |\partial_x^\gamma \Psi^*|.
\end{equation}
In particular, our strategy to establish Proposition \ref{prop:smooth} consists in applying inductively Proposition \ref{prop:smoothing} to the derivatives $\partial_x^k \Psi^*$ in order to improve their smoothness properties, and then translate the resulting properties in terms of the pair $(\eta^*, v^*)$. As a consequence, we split the proof into four steps.

\setcounter{step}{0}
\begin{step}
\label{R1}
Let $k \geq 1$. There exists a positive number $A_{k, \gc}$, depending only on $k$ and $\gc$, such that
\begin{equation}
\label{eq:Rk2}
\int_t^{t + 1} \int_\R \big| \partial_x^k \Psi^*(x + a^*(t), s) \big|^2 e^{2 \nu_\gc |x|} \, dx \, ds \leq A_{k, \gc},
\end{equation}
for any $t\in \R$.
\end{step}

The proof is by induction on $k \geq 1$. More precisely, we are going to prove by induction that \eqref{eq:Rk2} and
\begin{equation}
\label{eq:Rk}
\int_t^{t + 1} \int_\R \big| R_k(\Psi^*)(x + a^*(t), s) \big|^2 e^{2 \nu_\gc |x|} \, dx \, ds \leq A_{k, \gc},
\end{equation}
hold simultaneously for any $t \in \R$. Notice that \eqref{eq:Rk2} implies that $\partial_x^k \Psi^* \in L^2_{\rm loc}(\R, L^2(\R))$, while \eqref{eq:Rk} implies that $R_k(\Psi^*) \in L^2_{\rm loc}(\R, L^2(\R))$. Therefore, if \eqref{eq:Rk2} and \eqref{eq:Rk} are established for some $k \geq 1$, then \eqref{eq:psikder} can be justified by a standard approximation procedure, so that we are in position to apply Proposition \ref{prop:smoothing} to (suitable translates of) $\partial_x^k \Psi^*$.

For $k = 1$, recall that
$$|\partial_x \Psi^*|^2 = \frac{(\partial_x \eta^*)^2}{4 (1 - \eta^*)} + (1 - \eta^*) (v^*)^2.$$
 it follows that
$$\frac{1}{A_\gc} |\partial_x \Psi^*|^2 \leq (\partial_x \eta^*)^2 + (v^*)^2 \leq A_\gc |\partial_x \Psi^*|^2,$$
where the constant $A_\gc$, here as in the sequel, depends only on $\gc$. It then follows from Proposition \ref{prop:local} that \eqref{eq:Rk2} and \eqref{eq:Rk} are satisfied. Indeed, since
$$|R_1(\Psi^*)| \leq A |\partial_x \Psi^*| \big( 1 + |\Psi^*|^2 \big),$$
we have
$$\big| R_1(\Psi^*)(x + a^*(t), s) \big|^2 e^{2 \nu_\gc |x|} \leq A^2 |\partial_x \Psi^*(x + a^*(s), s)|^2 e^{2 \nu_\gc (|a^*(t) - a^*(s)| + |x|)} \big( 1 + \| \Psi^* \|_{L^\infty(\R)}^2 \big)^2,$$
and we may rely on Proposition \ref{prop:local}, and the fact that $|a^*(t) - a^*(s)|$ is bounded independently of $t$ for $s \in [t, t + 1]$. 

Assume next that \eqref{eq:Rk2} and \eqref{eq:Rk} are satisfied for any integer $k \leq k_0$ and any $t \in \R$. We apply Proposition \ref{prop:smoothing} with $u := \partial_x^{k_0} \Psi^*(\cdot + a^*(t), \cdot - (t + 1/2))$, $T := 1/2$ and successively $\lambda := \pm 2 \nu_\gc$. In view of \eqref{eq:psikder}, \eqref{eq:Rk2}, \eqref{eq:Rk}, and the fact that $|a^*(t) - a^*(s)|$ is uniformly bounded for $s \in [t - 1, t + 2]$, this yields
\begin{equation}
\label{eq:Rk3}
\int_t^{t + 1} \int_\R |\partial_x^{k_0 + 1} \Psi^*(x + a^*(t), s)|^2 e^{2 \nu_\gc |x|} \, dx \, ds \leq \frac{A_\gc A_{k_0, \gc} K_{2 \nu_\gc}}{4 \nu_\gc^2},
\end{equation}
so that \eqref{eq:Rk2} is satisfied for $k = k_0 + 1$, if we set $A_{k_0 + 1, \gc} = A_\gc A_{k_0, \gc} K_{2 \nu_\gc}/4 \nu_\gc^2$.

We now turn to \eqref{eq:Rk} which we wish to establish for $k = k_0 + 1$. First notice that the linear term in the right-hand side is already bounded by \eqref{eq:Rk3}, so that we only have to handle with the cubic terms. Notice also that we have by \eqref{eq:psikder}, \eqref{eq:Rk2} and \eqref{eq:Rk},
$$\partial_x^j \Psi^* \in L_{\rm loc}^2(\R, H^2(\R)), \quad {\rm and} \quad \partial_x^j \Psi^* \in H_{\rm loc}^1(\R, L^2(\R)),$$
for any $1 \leq j < k_0$, with bounds depending only on $k_0 + 1$ and $\gc$ on any time interval of length $1$. By interpolation, we obtain similar bounds for $\partial_x^j \Psi^* \in H_{\rm loc}^s(\R, H^{2 - 2 s}(\R))$ for any $0 \leq s \leq 1$. Taking for instance $s = 2/3$ and using the Sobolev embedding theorem, we obtain a global bound for $\partial_x^j \Psi^*$ in $L^\infty(\R \times \R)$. Since the latter also holds for $j = 0$, we thus have
\begin{equation}
\label{eq:borninfini}
\big\| \partial_x^j \Psi^* \big\|_{L^\infty(\R \times \R)} \leq A_{k_0 + 1, \gc},
\end{equation}
for any $0 \leq j < k_0$, where the value of $A_{k_0 + 1, \gc}$ possibly needs to be increased with respect to its prior value, but depending only on $k_0 + 1$ and $\gc$.

In order to estimate the sum in \eqref{eq:decompRk}, we next distinguish two cases according to the possible values of $\alpha$, $\beta$ and $\gamma$. 

\begin{case}
If $\beta < k_0$, then
\begin{align*}
\int_t^{t + 1} \int_\R \big[ |\partial_x^\alpha \Psi^*|^2 \, & |\partial_x^\beta \Psi^*|^2 \, | \partial_x^\gamma \Psi^*|^2 \big](x + a^*(t), s) e^{2 \nu_\gc |x|} \, dx \, ds\\
& \leq \big\| \partial_x^\alpha \Psi^* \big\|_{L^\infty(\R \times \R)}^2 \big\|\partial_x^\beta \Psi^* \big\|_{L^\infty(\R \times \R)}^2 \int_t^{t + 1} \int_\R |\partial_x^{\gamma} \Psi^*(x + a^*(t), s)|^2 e^{2 \nu_\gc |x|} \, dx \, ds,
\end{align*}
and we may rely on \eqref{eq:borninfini}, as well as \eqref{eq:Rk2} or \eqref{eq:Rk3}, depending on the value of $\gamma$.
\end{case}

\begin{case}
Since $\alpha \leq \beta \leq \gamma$ and $\alpha + \beta + \gamma = k_0 + 1$, the only remaining case is $\alpha = 0$, $\beta = \gamma = k_0 = 1$. In that situation, we write
\begin{align*}
\int_t^{t + 1} \int_\R & \big[ |\Psi^*|^2 \, |\partial_x \Psi^*|^4 \big](x + a^*(t), s) e^{2 \nu_\gc |x|} \, dx \, ds\\
& \leq \big\| \Psi^* \big\|_{L^\infty(\R \times \R)}^2 \bigg( \sup_{s \in [t, t + 1]} \int_\R |\partial_x \Psi^*(x, s)|^2 \, dx \bigg) \int_t^{t + 1} \big\| \partial_x \Psi^*(\cdot + a^*(t), s) e^{\nu_\gc |x|} \big\|^2_{L^\infty(\R)} \, ds.
\end{align*}
By conservation of the energy, we have
$$\sup_{s \in [t, t + 1]} \int_\R |\partial_x \Psi^*(x, s)|^2 \, dx \leq 2 \boE(\Psi^*(\cdot,0)).$$
while, by the Sobolev embedding theorem,
\begin{align*}
\big\| \partial_x \Psi^*(\cdot + a^*(t), s) e^{\nu_\gc |\cdot|} & \big\|_{L^\infty(\R)}^2\\
\leq & A_\gc \Big( \big\| \partial_{xx} \Psi^*(\cdot + a^*(t), s) e^{\nu_\gc |\cdot|} \big\|_{L^2(\R)}^2 + \big\| \partial_x \Psi^*(\cdot + a^*(t), s) e^{\nu_\gc |\cdot|} \big\|_{L^2(\R)}^2 \Big).
\end{align*}
The conclusion then follows also from \eqref{eq:Rk2} and \eqref{eq:Rk3}. 
\end{case}

At this stage, we have established by induction that \eqref{eq:Rk2} and \eqref{eq:Rk} hold for any $k \geq 1$. In order to finish the proof of Proposition \ref{prop:smooth}, we now turn these $L_{\rm loc}^2$ in time estimates into $L^\infty$ in time estimates, and then in uniform estimates.

\begin{step}
\label{R2}
Let $k \geq 1$. There exists a positive number $A_{k, \gc}$, depending only on $k$ and $\gc$, such that
\begin{equation}
\label{tlemcen}
\int_\R \big| \partial_x^k \Psi^*(x + a^*(t), t) \big|^2 e^{2 \nu_\gc |x|} \, dx \leq A_{k, \gc},
\end{equation}
for any $t\in \R$. In particular, we have
\begin{equation}
\label{oran}
\big\| \partial_x^k \Psi^*(\cdot + a^*(t), t) e^{ \nu_\gc |\cdot|} \big\|_{L^\infty(\R)} \leq A_{k, \gc},
\end{equation}
for any $t \in \R$, and a further positive constant $A_{k, \gc}$, depending only on $k$ and $\gc$.
\end{step}

Here also, we first rely on the Sobolev embedding theorem and \eqref{eq:psikder}. By the Sobolev embedding theorem, we have
\begin{align*}
\big\| \partial_x^k \Psi^*(\cdot + a^*(t), t) e^{\nu_\gc |\cdot|} \big\|_{L^2(\R)}^2 \leq & K \Big( \big\| \partial_s \big( \partial_x^k \Psi^*(\cdot + a^*(t), s) e^{\nu_\gc |\cdot|} \big) \big\|_{L^2([t - 1, t + 1], L^2(\R))}^2\\
& + \big\| \partial_x^k \Psi^*(\cdot + a^*(t), s) e^{\nu_\gc |\cdot|} \big\|_{L^2([t - 1, t + 1],L^2(\R))}^2 \Big),
\end{align*}
while, by \eqref{eq:psikder},
\begin{align*}
\big\| \partial_s \big( \partial_x^k \Psi^*(\cdot + a^*(t), s) e^{\nu_\gc |\cdot|} \big) \big\|_{L^2([t - 1, t + 1], L^2(\R))}^2 \leq & 2 \Big( \big\| \partial_x^{k + 2} \Psi^*(\cdot + a^*(t), s) e^{\nu_\gc |\cdot|} \big\|_{L^2([t - 1, t + 1], L^2(\R))}^2\\
& + \big\| R_k(\Psi^*)(\cdot + a^*(t), s) e^{\nu_\gc |\cdot|} \big\|_{L^2([t - 1, t + 1], L^2(\R))}^2 \Big),
\end{align*}
so that we finally deduce \eqref{tlemcen} from \eqref{eq:Rk} and \eqref{eq:Rk2}. Estimate \eqref{oran} follows applying the Sobolev embedding theorem. 

We now translate \eqref{tlemcen} and \eqref{oran} into estimates for $\eta^*$.

\begin{step}
\label{R3}
Let $k \in \N$. There exists a positive number $A_{k, \gc}$, depending only on $k$ and $\gc$, such that
\begin{equation}
\label{alger}
\int_\R \big( \partial_x^k \eta^*(x + a^*(t), t) \big)^2 e^{2 \nu_\gc |x|} \, dx \leq A_{k, \gc},
\end{equation}
and
\begin{equation}
\label{setif}
\big\| \partial_x^k \eta^*(\cdot + a^*(t), t) e^{ \nu_\gc |\cdot|} \big\|_{L^\infty(\R)} \leq A_{k, \gc},
\end{equation}
for any $t \in \R$.
\end{step}

Concerning \eqref{alger}, we first recall that
$$\partial_s \big( \eta^*(\cdot + a^*(t), s) \big) = 2 \langle i \Psi^*(\cdot + a^*(t), s), \partial_{xx} \Psi^*(\cdot + a^*(t), s) \rangle_\C.$$
Since $\Psi^*$ is uniformly bounded on $\R \times \R$ in view of \eqref{eq:lvmh2bis}, we can rely on \eqref{tlemcen} to claim that
$$\big\| \partial_s \big( \eta^*(\cdot + a^*(t), s) e^{\nu_\gc |\cdot|} \big) \big\|_{L^2([t - 1, t + 1], L^2(\R))} \leq A_\gc \big\| \partial_{xx} \Psi^*(\cdot + a^*(t), s) e^{\nu_\gc |\cdot|} \big\|_{L^2([t - 1, t + 1], L^2(\R))} \leq A_\gc.$$
Since
$$\big\| \eta^*(\cdot + a^*(t), s) e^{\nu_\gc |\cdot|} \big\|_{L^2([t - 1, t + 1], L^2(\R))} \leq A_\gc,$$
by Proposition \ref{prop:local}, and since $|a^*(t) - a^*(s)|$ is bounded independently of $t$ for $s \in [t - 1, t + 1]$, we can invoke again the Sobolev embedding theorem to obtain \eqref{alger} for $k = 0$.

When $k \geq 1$, we recall that
$$\partial_x^k \eta^* = - 2 \sum_{j = 0}^{k - 1} \binom{k - 1}{j} \big\langle \partial_x^j \Psi^*, \partial_x^{k - j} \Psi^* \big\rangle_\C,$$
by the Leibniz rule, so that \eqref{alger} follows from \eqref{tlemcen}, \eqref{oran}, and the property that $\Psi^*$ is uniformly bounded on $\R \times \R$ by \eqref{eq:lvmh2bis}. The uniform bound in \eqref{setif} is then a consequence of the Sobolev embedding theorem arguing as for \eqref{oran}.

Finally, we provide the estimates for the function $v^*$.

\begin{step}
\label{R4}
Let $k \in \N$. There exists a positive number $A_{k, \gc}$, depending only on $k$ and $\gc$, such that
\begin{equation}
\label{constantine}
\int_\R \big( \partial_x^k v^*(x + a^*(t), t) \big) e^{2 \nu_\gc |x|} \, dx \leq A_{k, \gc},
\end{equation}
and
\begin{equation}
\label{annaba}
\big\| \partial_x^k v^*(\cdot + a^*(t), t) e^{ \nu_\gc |\cdot|} \big\|_{L^\infty(\R)} \leq A_{k, \gc},
\end{equation}
for any $t \in \R$.
\end{step}

Here, we recall that
$$v^* = \big( 1 - \eta^* \big)^{- \frac{1}{2}} \big\langle i \partial_x \Psi^*, \Psi^* \big\rangle_\C,$$
so that, by the Leibniz rule, we have
$$\partial_x^k v^* = \sum_{j = 0}^k \sum_{\ell = 0}^{k - j} \binom{k}{j} \binom{k - j}{\ell} \partial_x^j \Big( (1 - \eta^*)^{- \frac{1}{2}} \Big) \big\langle i \partial_x^{\ell + 1} \Psi^*, \partial_x^{k - j - \ell}\Psi^* \big\rangle_\C.$$
At this stage, we can combine the Faa di Bruno formula with \eqref{eq:lvmh2bis} and \eqref{setif} to guarantee that
$$\Big\| \partial_x^j \Big( (1 - \eta^*)^{- \frac{1}{2}} \Big)(\cdot + a^*(t), t) \Big\|_{L^\infty(\R)} \leq A_{j, \gc},$$
for any $j \in \N$ and any $t \in \R$. In view of \eqref{tlemcen} and \eqref{oran}, this leads to \eqref{constantine}. The uniform bound in \eqref{annaba} follows again from the Sobolev embedding theorem.

In view of \eqref{setif} and \eqref{annaba}, we conclude that the pair $(\eta^*, v^*)$ is smooth on $\R \times \R$, with exponential decay. Estimate \eqref{eq:smooth} is a direct consequence of \eqref{alger} and \eqref{constantine}. This completes the proof of Proposition \ref{prop:smooth}. \qed

\section{Complements on orbital stability and the operator $\boH_c$}
\label{sec:Hc}

\subsection{Properties of the operator $\boH_c$}
\label{sub:prop-Hc}

In this subsection, we recall and slightly extend some properties of the operator $\boH_c$ which were established in \cite{LinZhiw1, BetGrSm1}. 

For $c \in (- \sqrt{2}, \sqrt{2}) \setminus \{ 0 \}$, the operator $\boH_c$ is given in explicit terms by
\begin{equation}
\label{def:boHc}
\boH_c(\eps) = \begin{pmatrix} - \frac{1}{4} \partial_x \Big( \frac{\partial_x \eps_\eta}{1 - \eta_c} \Big) + \frac{1}{4} \Big( 2 - \frac{\partial_{xx} \eta_c}{(1 - \eta_c)^2} - \frac{(\partial_x \eta_c)^2}{(1 - \eta_c)^3} \Big) \eps_\eta - \Big( \frac{c}{2} + v_c \Big) \eps_v\\ - \Big( \frac{c}{2} + v_c \Big) \eps_\eta + (1 - \eta_c) \eps_v \end{pmatrix}.
\end{equation}
It follows from the Weyl theorem and criterion that $\boH_c$ is self-adjoint on $L^2(\R) \times L^2(\R)$, with domain $H^2(\R) \times L^2(\R)$, and that its essential spectrum is equal to 
$$\sigma_{\rm ess}(\boH_c) = \Big[ \frac{2 - c^2}{3 + \sqrt{1 + 4 c^2}}, + \infty \Big).$$
It was proved in \cite{LinZhiw1, BetGrSm1} that $\boH_c$ has a unique negative eigenvalue, that its 
kernel is spanned by $\partial_x Q_c$, and that there exists a positive constant $\Lambda_c$, depending only and continuously on $c$, such that we have the estimate $H_c(\eps) \geq \Lambda_c \|\eps\|_X^2$, for any pair $\eps \in X(\R)$ which satisfies the orthogonality conditions $\langle \eps, \partial_x Q_c \rangle_{L^2(\R)^2} = P'(Q_c)(\eps) = 0$.

It follows from the characterization of the kernel here above that the operator $\boH_c$ is an isomorphism from $\Dom(\boH_c) \cap \Span(\partial_x Q_c)^\perp$ onto $\Span(\partial_x Q_c)^\perp$. Moreover, given any $k \in \N$, there exists a positive number $A_c$, depending continuously on $c$, such that the inverse mapping $\boH_c^{- 1}$ satisfies
\begin{equation}
\label{eq:inv-Hc}
\big\| \boH_c^{- 1}(f, g) \big\|_{H^{k + 2}(\R) \times H^k(\R)} \leq A_c \big\| (f, g) \big\|_{H^k(\R)^2},
\end{equation}
for any $(f, g) \in H^k(\R)^2 \cap \Span(\partial_x Q_c)^\perp$.

Indeed, the pair $\eps = \boH_c^{- 1}(f, g)$ is a solution in $H^2(\R) \times L^2(\R)$ to the equations
\begin{equation}
\label{cesar}
\left\{ \begin{array}{ll} - \frac{1}{4} \partial_x \Big( \frac{\partial_x \eps_\eta}{1 - \eta_c} \Big) = f - \frac{1}{4} \Big( 2 - \frac{\partial_{xx}^2 \eta_c}{(1 - \eta_c)^2} - \frac{(\partial_x \eta_c)^2}{(1 - \eta_c)^3} \Big) \eps_\eta + \Big( \frac{c}{2} + v_c \Big) \eps_v,\\
(1 - \eta_c) \eps_v = g + \Big( \frac{c}{2} + v_c \Big) \eps_\eta, \end{array} \right.
\end{equation}
which satisfies the bound
\begin{equation}
\label{brutus}
\| \eps_\eta \|_{L^2(\R)} + \| \eps_v \|_{L^2(\R)} \leq \kappa_c \big( \| f \|_{L^2(\R)} + \| g \|_{L^2(\R)} \big),
\end{equation}
with
$$\kappa_c := \min \Big\{ \frac{1}{\lambda}, \ \lambda \neq 0 \ {\rm s.t.} \ \lambda \in \sigma(\boH_c) \Big\}.$$
In particular, since $\boH_c$ depends analytically on $c$, and its eigenvalue $0$ is isolated, the constant $\kappa_c$ is positive and depends continuously on $c$. Since
\begin{equation}
\label{eq:max-etac}
\min_{x \in \R} \big\{ 1 - \eta_c(x) \big\} = \frac{c^2}{2} > 0,
\end{equation}
we can apply standard elliptic theory to the first equation in \eqref{cesar} to obtain
$$\| \eps_\eta \|_{H^2(\R)} \leq A_c \big( \| f \|_{L^2(\R)} + \| g \|_{L^2(\R)} \big),$$
where $A_c$ also depends continuously on $c$. Combining the second equation in \eqref{cesar} with \eqref{brutus} and \eqref{eq:max-etac}, it follows that
$$\| \eps_v \|_{H^{\min \{ k, 2 \}}(\R)} \leq A_c \big( \| f \|_{H^k(\R)} + \| g \|_{H^k(\R)} \big).$$
when $(f, g) \in H^k(\R)^2$. Applying again standard elliptic theory to the first equation in \eqref{cesar}, we are led to
$$\| \eps_\eta \|_{H^{\min \{ k + 2, 4 \}}(\R)} \leq A_c \big( \| f \|_{H^k(\R)} + \| g \|_{H^k(\R)} \big).$$
A bootstrap argument then yields \eqref{eq:inv-Hc}, with a constant $A_c$ which depends continuously on $c$.

\subsection{Proof of Theorem \ref{thm:orbistab}}
\label{sec:orbistab}

As mentioned in the introduction, the proof of Theorem \ref{thm:orbistab} consists in a few adaptations with respect to the arguments in \cite{LinZhiw1, BetGrSm1}.

The global existence of the solution $(\eta, v)$ to \eqref{HGP} for an initial data $(\eta_0, v_0)$ which satisfies the condition \eqref{cond:alpha} is indeed established in \cite[Theorem 2]{BetGrSm1}.

The existence for a fixed number $t \in \R$ of the modulation parameters $a(t)$ and $c(t)$ in \eqref{def:eps} is shown in \cite[Proposition 2]{BetGrSm1}, as well as the two estimates in \eqref{eq:modul0}. Combining these estimates with the Sobolev embedding theorem of $H^1(\R)$ into $\boC^0(\R)$ and the bound \eqref{eq:max-etac} on $1 - \eta_c$, we can write
$$\max_{x \in \R} \eta(x, t) \geq \big\| \eta_{c(t)} \big\|_{L^\infty(\R)} - \big\| \eps_\eta(\cdot, t) \big\|_{L^\infty(\R)} \geq 1 - \frac{c(t)^2}{2} - K_\gc \alpha_0 \geq 1 - \frac{\gc^2}{2} - K_\gc \alpha_\gc.$$
For $\alpha_\gc$ small enough, estimate \eqref{eq:max-eta} follows with $\sigma_\gc := \gc^2/2 + K_\gc \alpha_\gc$.

Concerning the $\boC^1$-dependence on $t$ of the numbers $a(t)$ and $c(t)$, it is proved in \cite[Proposition 4]{BetGrSm1}, as well as the linear estimate
\begin{equation}
\label{guardi}
\big| c'(t) \big| + \big| a'(t) - c(t) \big| \leq A_\gc \big\| \eps(\cdot, t) \big\|_{X(\R)}.
\end{equation}
The only remaining point to verify is that the linear dependence on $\eps$ of $c'(t)$ in \eqref{guardi} is actually quadratic.

In order to prove this further property, we differentiate the second orthogonality relation in \eqref{eq:ortho} with respect to time. Combining with \eqref{eq:poureps}, we obtain
\begin{equation}
\label{canaletto}
\begin{split}
c'\frac{d}{dc} \big( P(Q_c) \big) & = P'(Q_c) \big( J \boH_c(\eps) \big) + \big( a' - c \big) P'(Q_c) \big( \partial_x \eps + \partial_x Q_c \big)\\
& + c' \big\langle P''(Q_c)(\partial_c Q_c), \eps \big\rangle_{L^2(\R)^2} + P'(Q_c) \big( J \boR_c \eps \big),
\end{split}
\end{equation}
at any time $t \in \R$. The first term in the right-hand side of \eqref{canaletto} vanishes since 
\begin{equation}
\label{giudecca}
P'(Q_c) \big( J \boH_c(\eps) \big) = 2 \langle \partial_x Q_c, \boH_c(\eps) \rangle_{L^2(\R)^2} = 2 \langle \boH_c(\partial_x Q_c), \eps \rangle_{L^2(\R)^2} = 0,
\end{equation}
by \eqref{eq:Ker-Hc}. Concerning the second one, we have
\begin{equation}
\label{murano}
P'(Q_c) \big( \partial_x Q_c \big) = \int_\R \partial_x \big( \eta_c(x) v_c(x) \big) \, dx = 0,
\end{equation}
while we can deduce from \eqref{guardi} that
\begin{equation}
\label{lido}
\big| a' - c \big| \big| P'(Q_c) \big( \partial_x \eps \big) \big| \leq A_\gc \big\| \eps \big\|_{X(\R)}^2.
\end{equation}
Similarly, the third term can be estimated as
\begin{equation}
\label{sanmarco}
\big| c' \big| \big| \langle P''(Q_c)(\partial_c Q_c), \eps \rangle_{L^2(\R)^2} \big| \leq A_\gc \big\| \eps \big\|_{X(\R)}^2.
\end{equation}
For the last term, we recall that
\begin{equation}
\label{def:Reps}
\begin{split}
\big[ \boR_{c(t)} \eps(\cdot, t) \big]_\eta := & \frac{(\partial_x \eta_c)^2 \eps_\eta^2 (3 - \eta_c - 2 \eta)}{8 (1 - \eta_c)^3 (1 - \eta)^2} + \frac{(\partial_x \eta_c) \eps_\eta (\partial_x \eps_\eta) (2 - \eta_c - \eta)}{4 (1 - \eta_c)^2 (1 - \eta)^2} + \frac{(\partial_x \eps_\eta)^2}{8 (1 - \eta)^2} - \frac{\eps_v^2}{2}\\
& - \partial_x \Big( \frac{\eps_\eta (\partial_x \eps_\eta)}{4 (1 - \eta_c) (1 - \eta)} + \frac{(\partial_x \eta_c) \eps_\eta^2}{4 (1 - \eta_c)^2 (1 - \eta)} \Big),\\
\big[ \boR_{c(t)} \eps(\cdot, t) \big]_v := & - \eps_\eta \eps_v,
\end{split}
\end{equation}
so that we can compute
\begin{align*}
P'(Q_c) \big( J \boR_c \eps \big) = & \int_\R (\partial_{xx} \eta_c) \Big( \frac{\eps_\eta (\partial_x \eps_\eta)}{2 (1 - \eta_c) (1 - \eta)} + \frac{(\partial_x \eta_c) \eps_\eta^2}{2 (1 - \eta_c)^2 (1 - \eta)} \Big)\\
- & \int_\R \Big( 2 (\partial_x v_c) \eps_\eta \eps_v + (\partial_x \eta_c) \eps_v^2 \Big) - \int_\R (\partial_x \eta_c) \Big( \frac{(\partial_x \eta_c)^2 \eps_\eta^2 (3 - \eta_c - 2 \eta)}{4 (1 - \eta_c)^3 (1 - \eta)^2}\\
& + \frac{(\partial_x \eta_c) \eps_\eta (\partial_x \eps_\eta) (2 - \eta_c - \eta)}{2 (1 - \eta_c)^2 (1 - \eta)^2} + \frac{(\partial_x \eps_\eta)^2}{4 (1 - \eta)^2} \Big).
\end{align*}
It is then enough to apply again the Sobolev embedding theorem and to use the control on $1 - \eta$ and $c$ provided by \eqref{eq:max-eta}, respectively \eqref{eq:modul0}, to obtain
$$\big| P'(Q_c) \big( J \boR_c \eps \big) \big| \leq A_\gc \big\| \eps \big\|_{X(\R)}^2.$$
Recalling that
$$\frac{d}{dc} \big( P(Q_c) \big) = - \big( 2 - c^2 \big)^\frac{1}{2} \neq 0,$$
we can combine the identity \eqref{canaletto} with the estimates \eqref{giudecca}, \eqref{murano}, \eqref{lido} and \eqref{sanmarco} to prove the quadratic estimate of $c'(t)$ in \eqref{eq:modul1}. This concludes the proof of Theorem \ref{thm:orbistab}. \qed

\begin{merci}
The authors are grateful to C. Gallo, Y. Martel, F. Merle and L. Vega for interesting and helpful discussions. P.G. and D.S. are partially sponsored by the projects ``Around the dynamics of the Gross-Pitaevskii equation'' (JC09-437086) and ``Schr\"odinger equations and applications'' (ANR-12-JS01-0005-01) of the Agence Nationale de la Recherche. 
\end{merci}

\bibliographystyle{plain}
\bibliography{Bibliogr.bib}

\end{document}